\documentclass{amsart}
\usepackage{graphicx}
\usepackage{amsmath, amsfonts,amssymb}
\usepackage{amsthm}
\usepackage{mathrsfs}
\usepackage{etex}
\usepackage[all,cmtip,arrow,2cell]{xy}
\usepackage[colorlinks,
             linkcolor=black,
             citecolor=blue,
             bookmarksnumbered=true]{hyperref}
\usepackage{appendix}
\usepackage{xrtwo}
\newtheorem{prop}{Proposition}[section]
\newtheorem{thm}{Theorem}[section]
\newtheorem{cor}{Corollary}[section]
\newtheorem{lem}[thm]{Lemma}

\theoremstyle{definition}
\newtheorem{dfn}{Definition}[section]
\newtheorem{ex}{Example}

\theoremstyle{remark}
\newtheorem{rmk}{Remark}[section]
\theoremstyle{remark}

\theoremstyle{remark}
\newcommand{\set}[1]{\left\{#1\right\}}

\def\R{\mathscr{R}}
\newcommand{\G}         {\mathcal{G}}
\newcommand{\h}         {\mathcal{H}}
\newcommand{\K}         {\mathcal{K}}
\newcommand{\Ll}         {\mathcal{L}}
\newcommand{\B}         {\mathcal{B}}
\newcommand{\E}         {\mathcal{E}}
\newcommand{\F}         {\mathcal{F}}

\def\b{\mathscr B}

\def\C{\mathscr C}
\def\X{\mathscr X}
\def\Y{\mathscr Y}
\def\Z{\mathscr Z}
\def\W{\mathscr W}
\def\S{\mathscr S}
\def\g{\mathscr G}

\newcommand{\Top} {\mathfrak{Top}}
\newcommand{\Et} {\mathfrak{Et}}
\newcommand{\Eft} {\mathfrak{EffEt}}
\def\Ha{\mathscr{H}\!}

\def\RR{\mathbb{R}}
\def\xm{X\mbox{-}\mathbf{Mfd}}
\def\Mfd{\mathbf{Mfd}}
\DeclareMathOperator{\Lan}{Lan}
\def\HH{\mathbb{H}\!\mspace{1.5mu}}
\newcommand{\T} {\mathbb{T}\mathbf{op}}
\def\Sts{\St\left(S\right)}

\newcommand{\hocolim}{\operatorname{hocolim}}

\newcommand{\hc}        {\underset{{-\!\!\!-\!\!\!-\!\!\!-\!\!\!-\!\!\!-\!\!\!-\!\!\!-\!\!\!-\!\!\!\longrightarrow}} \hocolim  \:}

\newcommand{\Hom}{\operatorname{Hom}}
\DeclareMathOperator{\St}{St}
\DeclareMathOperator{\Set}{Set}
\DeclareMathOperator{\Gpd}{Gpd}
\DeclareMathOperator{\Sh}{Sh}
\DeclareMathOperator{\sit}{Site}
\newcommand{\Ef}{\operatorname{Eff}}

\DeclareMathAlphabet{\scr}{OT1}{pzc}%
                                 {m}{it}

\newcommand{\rt} { \rotatebox[origin=C]{90}{$\perp$} }

\def\rrrarrow{\hspace{.05cm}\mbox{\,\put(0,-3){$\rightarrow$}\put(0,1){$\rightarrow$}\put(0,5){$\rightarrow$}\hspace{.45cm}}}
\def\rrarrow{  \hspace{.05cm}\mbox{\,\put(0,-2){$\rightarrow$}\put(0,2){$\rightarrow$}\hspace{.45cm}}}
\def\acts{\hspace{.1cm}{\setlength{\unitlength}{.30mm}\linethickness{.09mm}
                        \begin{picture}(8,8)(0,0)\qbezier(7,6)(4.5,8.3)(2,7)\qbezier(2,7)(-1.5,4)(2,1)\qbezier(2,1)(4.5,-.3)(7,2)
                                                 \qbezier(7,6)(6.1,7.5)(6.8,9)\qbezier(7,6)(5,6.1)(4.2,4.4)
                        \end{picture}\hspace{.1cm}}}

\def\longlongrightarrow{-\!\!\!-\!\!\!-\!\!\!-\!\!\!-\!\!\!-\!\!\!\longrightarrow}

\begin{document}

\title{\'Etale Stacks as Prolongations}
\author{David Carchedi}
\address{D. Carchedi \hspace{16pt}\mbox{Max Planck Institute for Mathematics, Bonn, Germany. }
}
\begin{abstract}
In this article, we derive many properties of \'etale stacks in various contexts, and prove that \'etale stacks may be characterized categorically as those stacks that arise as prolongations of stacks on a site of spaces and local homeomorphisms. Moreover, we show that the bicategory of \'etale differentiable stacks and local diffeomorphisms is equivalent to the $2$-topos of stacks on the site of smooth manifolds and local diffeomorphisms. An analogous statement holds for other flavors of manifolds (topological, $C^k,$ complex, super...), and topological spaces locally homeomorphic to a given space $X.$ A slight modification of this result also holds in an even more general context, including all \'etale topological stacks, and Zariski \'etale stacks, and we also sketch a proof of an analogous characterization of Deligne-Mumford algebraic stacks. We go on to characterize effective \'etale stacks as precisely those stacks arising as the prolongations of sheaves. It follows that \'etale stacks (and in particular orbifolds) induce a small gerbe over their effective part, and all gerbes over effective \'etale stacks arise in this way. As an application, we show that well known Lie groupoids arising in foliation theory give presentations for certain moduli stacks. For example, there exists a classifying stack for Riemannian metrics, presented by Haefliger's groupoid $R\Gamma$ \cite{Haefliger} and submersions into this stack classify Riemannian foliations, and similarly for symplectic structures, with the role of $R\Gamma$ replaced with $\Gamma^{Sp}.$ We also prove some unexpected results, for example: the category of smooth $n$-manifolds and local diffeomorphisms has binary products.

\end{abstract}

\maketitle
\markboth{David Carchedi}{\'Etale Stacks as Prolongations}


\tableofcontents

\section{Introduction}
In this article, we derive many properties of \'etale stacks in various contexts, and prove that \'etale stacks may be characterized categorically as those stacks that arise as prolongations of stacks on a site of spaces and local homeomorphisms.

Geometrically, \'etale stacks model quotients of spaces by certain symmetries, and their points can possess intrinsic (discrete) automorphism groups. A widely studied class of such stacks in the differentiable setting is orbifolds, which arose initially out of foliation theory, but currently enjoy a wide variety of other uses. \'Etale differentiable stacks are an important class of stacks as they include not only all orbifolds, but more generally, all stacky leaf spaces of foliated manifolds. They have been studied by various authors \cite{Ie,Dorette,stacklie,hepworth,morsifold,Wockel,Giorgio,gerbes,etalspme}.

In this paper, we shall work over a suitable category $S$ whose objects we shall call ``spaces.'' (See Section \ref{sec:conventions} and Appendix \ref{sec:catspaces2} for more details.) An \emph{\'etale stack} over a category of spaces $S,$ is a stack which assigns to a space the groupoid of $\G$-torsors over that space, for an \'etale groupoid object $\G$ in spaces. To be concrete, let us take the category of spaces in question to be the category of smooth manifolds. An \'etale groupoid object is then the same thing as a Lie groupoid $\G$ whose source and target maps are local diffeomorphisms, and an \'etale stack then is an \'etale \emph{differentiable stack}, i.e. a stack which assigns to a manifold $M$ the groupoid of principal $\G$-bundles over $M,$ for such an \'etale Lie groupoid $\G$. If $$j:\Mfd^{et} \to \Mfd$$ denotes the inclusion of the category of smooth manifolds and local diffeomorphisms into the category of smooth manifolds and \emph{all} smooth maps, there is an evident restriction functor between their associated bicategories of stacks $$j^*:\St\left(\Mfd\right) \to \St\left(\Mfd^{et}\right),$$ which has a left adjoint $j_!$ called the \emph{prolongation} functor. One of the main results of this paper is that the essential image of $j_!$ is precisely \'etale differentiable stacks. This result holds for various other categories of spaces, and may be thought of as a representability criterion for \'etale stacks.

For certain categories of spaces, e.g. any flavor of manifolds (smooth, $C^k,$ complex, super...), and topological spaces locally homeomorphic to a given space $X,$ a stronger result is true: the bicategory of \'etale stacks and local homeomorphisms is equivalent to the $2$-topos of stacks on the site whose objects are the spaces in question, and whose arrows are local homeomorphisms between them, where following Section \ref{sec:conventions}, we use the term ``local homeomorphism'' to mean ``the appropriate notion of local homeomorphism'' for the category of spaces of interest. For example, when the category of spaces is the category of smooth manifolds, this result says that the bicategory of \'etale differentiable stacks and local diffeomorphisms is equivalent to $\St\left(\Mfd^{et}\right).$ In more detail, there is an equivalence of bicategories
$$\xymatrix@C=2cm{\St\left(S^{et}\right) \ar@<-0.65ex>[r]_-{\Theta} & \Et\left(S\right) \ar@<-0.65ex>[l]_-{\tilde y^{et}}},$$
such that for any stack $\Z$ on the site $S^{et}$ of spaces and local homeomorphisms between them, which we can view as a (generalized) \emph{moduli problem} which is functorial with respect to local homeomorphisms,
there exists a unique \'etale stack $\Theta \Z,$  such that for a given space $Y,$ the groupoid $\Z\left(Y\right)$ is equivalent to the groupoid of local homeomorphisms from $Y$ to $\Theta \Z:$
$$\Z\left(Y\right)\simeq \left\{f:Y \to \Theta \Z\mspace{3mu}|\mspace{3mu} f\mbox{ a local homeo.}\right\}.$$
Conversely, given any \'etale stack $\X$, it determines a stack on spaces and local homeomorphisms $\tilde y^{et}\left(\X\right)$ by assigning to a space $Y$ the groupoid
$$y^{et}\left(\X\right)\left(Y\right)=\left\{f:Y \to \X\mspace{3mu}|\mspace{3mu} f\mbox{ a local homeo.}\right\}$$
of local homeomorphisms from $Y$ to $\X,$ and these operations are inverse to each other.

Besides being of theoretical interest, these results also allow one to construct classifying stacks for a wide range of geometric structures which can be defined locally, but which are not functorial with respect to arbitrary smooth maps, e.g., Riemannian metrics and symplectic structures. The existence of these moduli stacks also give a clear way of defining what it means to have geometric structures on other \'etale stacks (for instance orbifolds), which we show to agree with the existing definitions in the case of Riemannian metrics and symplectic structures. We also show that many of these moduli stacks have presentations in terms of certain Lie groupoids that are well known in the foliation theory literature.

Among \'etale stacks, there is a certain nice class of \'etale stacks, called \emph{effective} \'etale stacks. After picking an \'etale atlas $$X \to \X$$ of an \'etale stack (with $X$ a space), one gets a bijection between points of $X$ and geometric points of $\X.$ The automorphism group $Aut\left(x\right)$ of any geometric point $x$ of an \'etale stack acts on the germ of $x$ in $X$. Geometrically, effective \'etale stacks are those \'etale stacks such that each of these actions are faithful. Effective \'etale stacks include all stacky leaf spaces of foliated manifolds, and in particular effective orbifolds (sometimes called reduced orbifolds). In terms of groupoids, \'etale stacks are precisely those stacks arising from effective \'etale groupoids, which in turn, are those \'etale groupoids arising from pseudogroups of homeomorphisms, in the sense of Cartan in \cite{cartan}.

In this paper, we give a categorical characterization of effectivity: Effective \'etale stacks are precisely those stacks arising as the prolongations of sheaves on the site of spaces and local homeomorphisms. Put another way, effective \'etale stacks are those \'etale stacks $\X$ such that the stack of local homeomorphisms into $\X$ is a sheaf of sets, rather than a stack of groupoids. This gives a much more simple and conceptual proof of the fact that \'etale stacks (and in particular orbifolds) induce a small gerbe over their effective part, and all gerbes over effective \'etale stacks arise in this fashion, a result proven in \cite{gerbes}.

The results of this paper rely heavily on the results of \cite{etalspme}, and may be thought of as a natural continuation of that paper. Furthermore, several results of this paper also appear in the author's previous preprint \cite{gerbes}, however, with the notable exception of Sections \ref{sec:effective} which are exactly the same as in Op. cit., the proofs in this article are quite different.

\subsection{Small gerbes and ineffective isotropy data}
This paper also offers new insights into the meaning of \emph{ineffective data} of \'etale stacks. We will now provide a geometric interpretation of these results. Suppose that $G$ is a finite group acting on a manifold $M.$ The stacky-quotient $M//G$ is an \'etale differentiable stack, and in particular, an orbifold. Points of this stacky-quotient are the same as points of the naive quotient, that is, orbits of the action. These are precisely images of points of $M$ under the quotient map $M \to M//G.$ For a particular point $x \in M$, if $\left[x\right]$ denotes the point in $M//G$ which is its image, then $$Aut\left(\left[x\right]\right)\cong G_x.$$
If this action is not faithful, then  there exists a non-trivial kernel $K$ of the homomorphism
\begin{equation}\label{eq:actionmap}
\rho:G \to \mathit{Diff}\left(M\right).
\end{equation}
In this case, any element $k$ of $K$ acts trivially and is tagged-along as extra data in the automorphism group $$Aut\left(\left[x\right]\right)\cong G_x$$ of each point $\left[x\right]$ of the stack $M//G$. In fact, $$\bigcap\limits_{x \in M} G_x=Ker\left(\rho\right)=K.$$ In particular, $\rho$ restricted to $Aut\left(\left[x\right]\right)$ becomes a homomorphism
\begin{equation}\label{eq:locact}
\rho_x:Aut\left(\left[x\right]\right) \to \mathit{Diff}\left(M\right)_x
\end{equation}
to the group of diffeomorphisms of $M$ which fix $x$. This homomorphism is injective for all $x$ if and only if the kernel of $\rho$ is trivial. The kernel of each of these homomorphisms can be thought of as an ``inflated'' part of each automorphism group, and is called the \emph{ineffective isotropy group} of $\left[x\right]$. Up to the identification $$Aut\left(\left[x\right]\right)\cong G_x,$$ each of these ineffective isotropy groups is $K.$ This extra information is stripped away when considering the stacky-quotient $$M//\left(G/K\right),$$ that is to say, $M//\left(G/K\right)$ is the \emph{effective part} of $M//G.$

Hence, having a kernel of the action (\ref{eq:actionmap}) means that a point's automorphism group can contain automorphisms that do not play an essential role in the geometric structure of the stack around that point.  As an extreme example, suppose the action $\rho$ is trivial, and consider the stacky quotient $M//G.$ It is the same thing as $M$ except each point $x$, has the group $G,$ rather the trivial group, as an automorphism group. These automorphisms are somehow artificial, since the action $\rho$ sees nothing of $G.$ In this case, the entire automorphism group of each point is its ineffective isotropy group, and this is an example of a purely ineffective orbifold. Since these arguments are local, the situation when $\X$ is an \'etale stack formed by gluing together stacks of the form $M_\alpha//G_\alpha$ for actions of finite groups, i.e. a general orbifold, is completely analogous.

For a more general \'etale stack $\X$, for example a stack of the form $M//G$ where $G \acts M$ is a smooth action of a Lie group with discrete (but not finite) stabilizers, there is in general no such local action of the automorphisms groups as in (\ref{eq:locact}), but the situation can be mimicked at the level of germs. There exists a manifold $V$ and a (representable) local diffeomorphism $$V \to \X$$ such that for every point $$x:* \to \X,$$
\begin{itemize}
\item[i)] the point $x$ factors (up to isomorphism) as $* \stackrel{\tilde x}{\longrightarrow} V\stackrel{p}{\longrightarrow} \X,$ and
\item[ii)] there is a canonical homomorphism $\tilde\rho_x: Aut\left(x\right) \to \mathit{Diff}_{\tilde x}\left(V\right),$
\end{itemize}
where $\mathit{Diff}_{\tilde x}\left(V\right)$ is the group of germs of locally defined diffeomorphisms of $V$ that fix $\tilde x.$ The kernel of each of these maps can again be considered as the ``inflated'' part of the  automorphism group. In the case where $\X$ is of the form $M//G$ for a finite group $G$ (or more generally, when $\X$ is an orbifold) the kernel of $\tilde \rho_x$ is the same as the kernel of (\ref{eq:locact}), for each $x$. In general, each $Ker\left(\tilde \rho_x\right)$ is called an \emph{ineffective isotropy group}. Unlike in the case of a global quotient $M//G,$ these groups need not be isomorphic for different points of the stack. However, these kernels may be killed off to obtain the so-called \emph{effective part} of the \'etale stack.

There is another way of artificially adding automorphisms to a stack which play no geometric role, and this is through gerbes. In this paper, we show that having ineffective isotropy data is the same as the presence of a gerbe. As a starting example, if $M$ is a manifold, a gerbe over $M$ is a stack $\g$ over $M$ such that over each point $x$ of $M,$ the stalk $\g_x$ is equivalent to a group (regarded as a one-object groupoid). From such a gerbe, one can construct an \'etale stack which looks just like $M$ except each point $x,$ now instead of having a trivial automorphism group, has (a group equivalent to) $\g_x$ as its automorphism group. However, these automorphism groups are entirely ineffective. This construction was alluded to in \cite{pres2}. One can use this construction to show that \'etale stacks whose effective parts are manifolds are the same thing as manifolds equipped with a gerbe. In this paper, we show that this result extends to general \'etale stacks, namely that any \'etale stack $\X$ encodes a small gerbe over its effective part $\Ef\left(\X\right)$, and moreover, every small gerbe over an effective \'etale stack $\Y$ arises uniquely from some \'etale stack $\Z$ whose effective part is equivalent to $\Y$. The construction of an \'etale stack  $\Z$ out of an effective \'etale stack $\Y$ equipped with a small gerbe $\g,$ is precisely the \'etal\'e realization of the gerbe $\g$, in the sense of \cite{etalspme}, i.e. the ``\'etal\'e space'' of the gerbe. In such a situation, there  is a natural bijection between the points of $\Z$ and the points of $\Y$, the only difference being that points of $\Z$ have more automorphisms. For $x$ a point $\Z$, its ineffective isotropy group, i.e. the kernel of $$Aut\left(x\right) \to \mathit{Diff}_{\tilde x}\left(V\right),$$ is equivalent to the stalk $\g_x.$

\subsection{Organization and main results}

In section \ref{sec:effective}, we introduce the concept of an effective \'etale stack and show how to associate to every \'etale stack $\X$ an effective \'etale stack $\Ef\left(\X\right)$, which we call its effective part. Although this construction is not functorial with respect to all maps, we show that it is functorial with respect to any category of open maps which is \'etale invariant (see Definition \ref{dfn:local}). Examples of open \'etale invariant classes of maps include open maps, local homeomorphisms, and submersions. We also recall the concept of a Haefliger groupoid, originally introduced in \cite{Haefliger}. Haefliger groupoids play a pivotal role in the rest of the paper.

Section \ref{sec:amazing} derives many of the surprising properties of the Haefliger groupoid construction (Definition \ref{dfn:haf}). The results in this section play a fundamental role in the paper, but first we need some terminology:

For a given space $X,$ one can consider the category $\xm$ of $X$-manifolds, which is the full subcategory on all spaces which can be covered by open subspaces of $X.$ E.g., when spaces means smooth manifolds, and $X=\RR^n,$ $\xm$ is the category of smooth $n$-manifolds, and when $X=\coprod\limits_{n=0}^\infty \RR^n,$ $\xm$ is all smooth manifolds. By an \'etale stack on $X$-manifolds, or an \'etale $X$-manifold stack, we mean a stack on the site of $X$-manifolds arising from an \'etale groupoid object $\G$ in spaces, such that its object space and arrow space are $X$-manifolds.

The two concepts, that of an $X$-manifold, and that of the Haefliger groupoid of a space $X,$ are intimately related, as the following results shows:

\begin{thm}
For any space $X,$ there is a canonical equivalence of topoi $$\Sh\left(\HH\left(X\right)\right) \simeq \Sh\left(\xm^{et}\right),$$ between sheaves on the Haefliger stack $\HH\left(X\right)-$ which is the \'etale stack associated to the Haefliger groupoid of $X-$ and sheaves on $\xm^{et}-$ the subcategory of $\xm$ spanned by local homeomorphisms, and sheaves are taken with respect to the induced open-cover topology.
\end{thm}

We also prove another, equally central result:

\begin{thm}
For any space $X,$ the Haefliger stack $\HH\left(X\right)$ is a terminal object in the bicategory consisting of \'etale stacks on $X$-manifolds and local homeomorphisms between them.
\end{thm}

As a consequence, we prove several of the main results of this paper:

\begin{thm}\label{thm:nn}
There is an equivalence of bicategories $$\St\left(\xm^{et}\right) \stackrel{\sim}{\longrightarrow} \Et\left(\xm\right)^{et}$$ between stacks on the site of $X$-manifolds and local homeomorphisms, and the bicategory of \'etale stacks on $X$-manifolds, and local homeomorphisms between them.
\end{thm}

\begin{cor}
There is an equivalence of bicategories $$\St\left(\Mfd^{et}\right) \stackrel{\sim}{\longrightarrow} \Et\left(\Mfd\right)^{et}$$ between \'etale differentiable stacks and local diffeomorphisms, and stacks on the site of manifolds and local diffeomorphisms.
\end{cor}

Given a space $X,$ one can consider the inclusion $$j:\xm^{et} \to \xm$$ from the category of $X$-manifolds and local homeomorphisms into the category of $X$-manifolds and all maps. There is an evident restriction functor $$j^*:\St\left(\xm\right) \to \St\left(\xm^{et}\right)$$ which has a left adjoint $j_!$ called the prolongation functor.

\begin{thm}\label{thm:stillok}
$\X \in \St\left(\xm\right)$ is an \'etale stack on $X$-manifolds if and only if it is in the essential image of the prolongation functor $j_!.$
\end{thm}

\begin{cor}
A stack $\X \in \St\left(\Mfd\right)$ is an \'etale differentiable stack if and only if it is in the essential image of the prolongation functor $j_!.$
\end{cor}

In Section \ref{sec:examples}, we turn our attention to some examples. First, we show how to compute binary products in $\Et\left(\xm\right)^{et}-$ which exist in light of Theorem \ref{thm:nn} $-$ and arrive at a rather odd corollary:

\begin{cor}
The category of $n$-manifolds and local diffeomorphisms admits binary products.
\end{cor}

We then show how to construct groupoid presentations of moduli stacks in a functorial way, and work out some examples. For the moduli stacks of Riemannian metrics and symplectic forms respectively, the Lie groupoids that come out of our construction turn out to be widely studied Lie groupoids arising in foliation theory, namely $R\Gamma$ and $\Gamma^{Sp}.$ Their classifying spaces classify homotopy classes of foliations with transverse Riemannian and symplectic structures respectively, on open manifolds (see \cite{Haefliger}). We show later, in Section \ref{sec:gerbeff}, that the associated \'etale stacks are finer objects than the classifying spaces $BR\Gamma$ and $B\Gamma^{Sp},$ as they are capable of classifying such foliations up to isomorphism, on all manifolds.

In section \ref{sec:large}, we study the case of large categories of spaces (e.g. the category of all topological spaces, or all schemes) and show a minor modification of Theorem \ref{thm:stillok} still holds in this setting. We also sketch a proof of an analogous statement for Deligne-Mumford algebraic stacks.

Section \ref{sec:gerbeff} begins with a modern treatment of the theory of gerbes in terms of the internal homotopy theory of the $2$-topos of stacks on a site. We then go on to give a complete characterization of effective \'etale stacks:

\begin{thm}
An \'etale stack $\X$ is effective if and only it arises as the prolongation of some sheaf $F$ on the category of $X$-manifolds and local homeomorphisms, for some space $X.$
\end{thm}

\begin{thm}
For any space $X,$ the equivalence
$$\Et\left(\xm\right)^{et}\stackrel{\sim}{\longrightarrow}\St\left(\xm^{et}\right),$$
between \'etale $X$-manifold stacks and local homeomorphisms and the bicategory of stacks on the site of $X$-manifolds and local homeomorphisms, restricts to an equivalence
$$\Eft\left(\xm\right)^{et}\stackrel{\sim}{\longrightarrow}\Sh\left(\xm^{et}\right),$$
between effective \'etale $X$-manifold stacks and local homeomorphisms, and sheaves on the site of $X$-manifolds and local homeomorphisms. (In particular $\Eft\left(\xm\right)^{et}$ is equivalent to a $1$-category.)
\end{thm}

\begin{cor}
A stack $\X$ on smooth manifolds is an effective \'etale differentiable stack if and only if it arises as the prolongation of a sheaf on the site of manifolds and local diffeomorphisms. Moreover, the equivalence
$$\Et\left(\Mfd\right)^{et}\stackrel{\sim}{\longrightarrow}\St\left(\Mfd^{et}\right),$$
between \'etale differentiable stacks and local diffeomorphisms and stacks on the site of manifolds and local diffeomorphisms, restricts to an equivalence
$$\Eft\left(\Mfd\right)^{et}\stackrel{\sim}{\longrightarrow}\Sh\left(\Mfd^{et}\right),$$
between effective \'etale differentiable stacks and local diffeomorphisms, and sheaves on the site of manifolds and local diffeomorphisms.
\end{cor}

We then go on to relate the theory of effective \'etale stacks to gerbes:

\begin{thm}
Let $f:\Z \to \X$ be a local homeomorphism of \'etale stacks, with $\X$ effective. Then its stack of sections $$\Gamma\left(f\right) \in \St\left(\X\right)$$ is a gerbe if and only if the induced map $\Ef\left(\Z\right) \to \X$ from the effective part of $\Z$ to $\X$ is an equivalence.
\end{thm}

We also give a characterization of gerbes over a general \'etale stack, which is a subtle correction to the characterization in \cite{gerbes}. See Theorem \ref{thm:localfullgerb} for the precise statement.

Furthermore, as an example of this machinery, we prove in this section that submersions into the classifying stack for Riemannian metrics (introduced in Section \ref{sec:amazing}) classify Riemannian foliations.

In Appendix \ref{sec:catspaces2}, we formalize exactly what properties are needed of a category of spaces for the results of this paper to apply to it. Finally, in Appendix \ref{sec:prelim}, we give a brief review of the theory of \'etale stacks, and of the results of \cite{etalspme}, which are crucial for this paper. We also give a brief discussion about stalks of stacks.

\vspace{0.2in}\noindent{\bf Acknowledgment:}
I would like to thank Andr\'e Henriques, Ieke Moerdijk, and Urs Schreiber for useful conversations. I would also like to thank Christian Blohmann, Marius Crainic, and Ieke Moerdijk for giving me the opportunity to speak about these results at the ``Higher Geometric Structures Along the Lower Rhine'' workshop in January, 2012. Finally, I am grateful to the Max Planck Institute for Mathematics for providing me with the stimulating environment in which I conducted much of this research.
 \newpage

\subsection{Conventions}\label{sec:conventions}
Throughout this article, $S$ shall denote a fixed category whose objects we shall call ``spaces,'' equipped with an appropriate class of morphisms, which we will refer to simply as local homeomorphisms. We follow nearly the same conventions as in \cite{etalspme}, with the notable exception that we exclude locales. It is possible to include locales, however the technical annoyance of making all proofs ``point-free'' outweighs the benefit. We will give a more broad treatment in \cite{higherme}, which will also cover the case of locales.

Here is a list of possibilities for $S.$
\begin{itemize}
\item[I)] Sober topological spaces and local homeomorphisms.
\item[II)] Any type of manifold (e.g. smooth manifolds, $C^k$ manifolds, analytic manifolds, complex manifolds, super manifolds...) with the appropriate version of local diffeomorphism, provided we remove all separation conditions. For example, manifolds will neither be assumed paracompact nor Hausdorff.
\item[III)] Schemes over any fixed base and Zariski local homeomorphisms. When viewed as maps of locally ringed spaces, Zariski local homeomorphisms are those maps $$\left(f,\varphi\right): \left(X,\mathcal{O}_X\right) \to \left(Y,\mathcal{O}_Y\right)$$ such that $f$ is a local homeomorphism and $\varphi: f^*\left(\mathcal{O}_Y\right) \stackrel{\sim}{\longrightarrow} \mathcal{O}_X$ is an isomorphism. Again, we do not impose any separation conditions.
\end{itemize}
This list need not be exhaustive. See Appendix \ref{sec:catspaces2} for a more systematic treatment. A morphism in $S$ will simply be called ``continuous.'' For example, if $S$ is taken to be the category of smooth manifolds, the phrase ``continuous map'' will mean a smooth map, and ``local homeomorphism'' will mean local diffeomorphism. Similarly for the other examples above. The reason for not imposing any separation conditions is to consider the \'etal\'e space (espace \'etal\'e) of a sheaf over a manifold, scheme etc., as a manifold or scheme itself.

We review the basics of \'etale stacks in this setting, and establish notational conventions and terminology used throughout this article in Appendix \ref{sec:prelim}. In the same appendix, we summarize many of the results of \cite{etalspme} which are used in an essential way in this paper.

\section{Effective Stacks}\label{sec:effective}
\subsection{Basic definitions}
We begin by recalling a special class of \'etale stacks, called effective \'etale stacks. These pop up in various guises, particularly in the study of foliation theory. Effective \'etale stacks include all stacky leaf spaces of foliated manifolds. They are precisely those \'etale stacks arising from pseudogroups of homeomorphisms, in the sense of Cartan in \cite{cartan}, or equivalently, from \'etale groupoids of germs. We start with a summary of results well known in the groupoid literature, expressed in a more stack-oriented language. We will begin first by defining effectiveness for orbifolds, as this definition lends more to geometric intuition. This will also make the general definition for an arbitrary \'etale stack more intuitive.

\begin{dfn}
A morphism $f:\X \to \Y$ between \'etale stacks is \textbf{proper} if the induced morphism $$\Sh\left(\X\right) \to \Sh\left(\Y\right)$$ is a proper geometric morphism of topoi in the sense of \cite{proper}, where the topoi above are the topoi of small sheaves in the sense of Definition \ref{dfn:smallsheaves}.
\end{dfn}

\begin{dfn}
An \'etale stack $\X$ is called an \textbf{orbifold} if the diagonal map $$\Delta:\X \to \X \times \X$$ is proper.
\end{dfn}

\begin{dfn}
An $S$-groupoid is an \textbf{orbifold groupoid} if it is \'etale and \textbf{proper}, i.e. the map $$\left(s,t\right):\h_1 \to \h_0 \times \h_0$$ is proper.
\end{dfn}

\begin{prop}
$\X$ is an orbifold if and only if there exists an orbifold groupoid $\h$ such that $\X \simeq \left[\h\right]$.
\end{prop}

\begin{proof}
For any \'etale $\h$ such that $\left[\h\right] \simeq \X$,
$$\xymatrix{\h_1 \ar[d]_{\left(s,t\right)} \ar[r] & \X \ar[d]^{\Delta}\\
\h_0 \times \h_0 \ar[r]^{a \times a} & \X \times \X}$$
is a weak pullback diagram, where $a:\h_0 \to \X$ is the atlas associated to $\h$. If $\X$ is an orbifold, it is now immediate that $\h$ is an orbifold groupoid. The converse follows from \cite{proper}, Proposition 2.7, and the equivalence of bicategories of Theorem \ref{thm:etendue}.
\end{proof}

Recall the following definition:
\begin{dfn}
If $G$ is a $S$-group acting on a space $X$, the action is \textbf{effective} (or faithful), if $\bigcap\limits_{x \in X} G_x=e$, i.e., for all non-identity elements $g \in G$, there exists a point $x \in X$ such that $g \cdot x \neq x.$ Equivalently, the induced homomorphism $$\rho:G \to \mathit{Diff}\left(X\right),$$ where $\mathit{Diff}\left(X\right)$ is the group of homeomorphisms of $X,$ is a monomorphism. (Recall that when $S$ is the category of smooth manifolds, the word homeomorphisms means diffeomorphisms.) These two definitions are equivalence since $$Ker\left(\rho\right)=\bigcap\limits_{x \in X} G_x=e.$$
\end{dfn}
If $\rho$ above has a non-trivial kernel $K,$ then there is an inclusion of $K$ into each isotropy group of the action, or equivalently into each automorphism group of the quotient stack (the stack associated to the action groupoid). In this case $K$ is ``tagged-along'' as extra data in each automorphism group. Each of these copies of $K$ is the kernel of the induced homomorphism $$\left(\rho\right)_x:Aut\left(\left[x\right]\right) \to \mathit{Diff}\left(M\right)_x,$$ where $\mathit{Diff}\left(M\right)_x$ is the group of homeomorphisms of $M$ which fix $x$. In the \emph{differentiable} setting, when $G$ is \emph{finite}, these kernel are called the \emph{ineffective isotropy groups} of the associated \'etale stack. In this case, the \emph{effective part} of this stack is the stacky-quotient of $X$ by the induced action of $G/K$. This latter stack has only trivial ineffective isotropy groups.

\begin{rmk}
If $G$ is not finite, this notion of ineffective isotropy group may not agree with Definition \ref{dfn:ineffgp}, since non-identity elements can induce the germ of the identity around a point. It the topological setting, this problem can occur even when $G$ is finite.
\end{rmk}

\begin{prop}\label{prop:orbeff}
Suppose $\X$ is an orbifold and $x:* \to \X$ is a point. Then there exists a representable local homeomorphism $p:V_x \to \X$ from a space $V_x$ such that:
\begin{itemize}
\item[i)] the point $x$ factors (up to isomorphism) as $* \stackrel{\tilde x}{\longrightarrow} V_x \stackrel{p}{\longrightarrow} \X$
\item[ii)] the automorphism group $Aut\left(x\right)$ acts on $V_x$.
\end{itemize}
\end{prop}

\begin{proof}
The crux of this proof comes from \cite{orbcon}. For a point $x$ of an \'etale stack $\X$, $Aut\left(x\right)$ fits into the $2$-Cartesian diagram \cite{NoohiF}

$$\xymatrix{Aut\left(x\right) \ar[d] \ar[r] & \mathrm{*} \ar[d]^-{x}\\
\mathrm{*} \ar[r]^-{x} & \X,}$$
and is a group object in spaces. If $\X \simeq \left[\h\right]$ for an $S$-groupoid $\h$, there is a point $\tilde x \in \h_0$ such that $x \cong a \circ \tilde x,$ where $a:\h_0 \to \X$ is the atlas associated to the groupoid $\h,$ and moreover, $\h_{\tilde x} \cong Aut\left(\tilde x\right)$, where $\h_{\tilde x}=s^{-1}\left(\tilde x\right) \cap t^{-1}\left(\tilde x\right)$ is the $S$-group of automorphisms of $\tilde x$. (In particular, this implies that if $\X$ is \'etale, then $Aut\left(x\right)$ is discrete for all $x$.) Suppose now that $\X$ and $\h$ are \'etale. Then for each $h \in \h_{\tilde x},$ there exists an open neighborhood $U_h$ such that the two maps $$s:U_h \to s\left(U_h\right)$$ $$t:U_h \to t\left(U_h\right)$$ are homeomorphisms. Now, suppose that $\X$ is in fact an orbifold (so that $\h$ is an orbifold groupoid). Then, it follows that $\h_{\tilde x}$ is finite. Given $f$ and $g$ in $\h_{\tilde x}$, we can find a small enough neighborhood $W$ of $\tilde x$ in $\h_0$ such that for all $z$ in $W$, $$t \circ s^{-1}|_{U_g}\left(z\right) \in s\left(U_f\right),$$ $$t \circ s^{-1}|_{U_g} \left(t \circ s^{-1}|_{U_f}\left(z\right) \right) \in W,$$ and

\begin{equation}\label{eq:391}
s^{-1}|_{U_g}\left(z\right) \cdot s^{-1}|_{U_f}\left(z\right) \in U_{gf}.
\end{equation}
In this case, by plugging in $z=\tilde x$ in (\ref{eq:391}), we see that (\ref{eq:391}) as a function of $z$ must be the same as $$s^{-1}|_{U_{gf}}.$$ Therefore, on $W$, the following equation holds
\begin{equation}\label{eq:lnof5}
t \circ s^{-1}|_{U_g} \left(t \circ s^{-1}|_{U_f}\left(z\right) \right)= t \circ s^{-1}|_{U_{gf}}.
\end{equation}
Since $\h_{\tilde x}$ is finite, we may shrink $W$ so that equation (\ref{eq:lnof5}) holds for all composable arrows in $\h_{\tilde x}$. Let $$V_x:=\bigcap\limits_{h \in \h_{\tilde x}} \left(t \circ s^{-1}|_{U_f}\left(W\right)\right).$$ Then, for all $h \in h_{\tilde x}$, $$t \circ s^{-1}|_{U_h}\left(V_x\right)=V_x.$$ So $t \circ s^{-1}|_{U_h}$ is a homeomorphism from $V_x$ to itself for all $x$, and since equation (\ref{eq:lnof5}) holds, this determines an action of $\h_{\tilde x}\cong Aut\left(x\right)$ on $V_x$. Finally, define $p$ to be the atlas $a$ composed with the inclusion $V_x \hookrightarrow \h_0$
\end{proof}

\begin{dfn}
An orbifold $\X$ is an \textbf{effective} orbifold, if the actions of $Aut\left(x\right)$ on $V_x$ as in the previous lemma can be chosen to be effective.
\end{dfn}

The finiteness of the stabilizer groups played a crucial role in the proof of Proposition \ref{prop:orbeff}. Without this finiteness, one cannot arrange (in general) for even a single arrow in an \'etale groupoid to induce a self-homeomorphism of an open subset of the object space. Additionally, even if each arrow had such an action, there is no guarantee that the (infinite) intersection running over all arrows in the stabilizing group of these neighborhoods will be open. Hence, for a general \'etale groupoid, the best we can get is a germ of a locally defined homeomorphism. It is using these germs that we shall extend the definition of effectiveness to arbitrary \'etale groupoids and stacks.

Given a space $X$ and a point $x \in X$, let $\mathit{Diff}_x\left(X\right)$ denote the group of germs of (locally defined) homeomorphisms that fix $x$.

\begin{prop}
Let $\X$ be an \'etale stack and pick an \'etale atlas $$V \to \X.$$ Then for each point $x:* \to \X$,
\begin{itemize}
\item[i)] the point $x$ factors (up to isomorphism) as $* \stackrel{\tilde x}{\longrightarrow} V\stackrel{p}{\longrightarrow} \X,$ and
\item[ii)] there is a canonical homomorphism $Aut\left(x\right) \to \mathit{Diff}_{\tilde x}\left(V\right)$.
\end{itemize}
\end{prop}

\begin{proof}
Following the proof of Proposition \ref{prop:orbeff}, let $V=\h_0$ and let the homomorphism send each $h \in \h_x$ to the germ of $t \circ s^{-1}|_{U_h}$, which is a locally defined homeomorphism of $V$ fixing $\tilde x$.
\end{proof}

\begin{cor}
For $\h$ an \'etale $S$-groupoid, for each $x \in \h_0$, there exists a canonical homomorphism of groups $\h_x \to \mathit{Diff}_{x}\left(\h_0\right)$.
\end{cor}

\begin{dfn}\label{dfn:ineffgp}
Let $x$ be a point of an \'etale stack $\X$. The \textbf{ineffective isotropy group} of $x$ is the kernel of the induced homomorphism $$Aut\left(x\right) \to \mathit{Diff}_{\tilde x}\left(V\right).$$ Similarly for $\h$ an \'etale groupoid.
\end{dfn}

\begin{dfn}
An \'etale stack $\X$ is \textbf{effective} if the ineffective isotropy group of each of its points is trivial. Similarly for $\h$ an \'etale groupoid.
\end{dfn}

\begin{prop}
An orbifold $\X$ is an effective orbifold if and only if it is effective when considered as an \'etale stack.
\end{prop}

\begin{proof}
This follows from \cite{Fol}, Lemma 2.11.
\end{proof}

\subsection{Haefliger groupoids}
We now describe a very useful construction, originally due to Haefliger \cite{Haefliger} in his study of foliation theory. Out of any space $X,$ this construction produces a canonically associated \'etale $S$-groupoid $\Ha\left(X\right).$ As we shall see in Section \ref{sec:amazing}, this groupoid has many special properties.

\begin{dfn}\label{dfn:haf}\cite{Haefliger}
Let $X$ be a space. Consider the presheaf $$\mathit{Emb}:\mathcal{O}\left(X\right)^{op} \to \Set,$$ which assigns an open subset $U$ the set of open embeddings of $U$ into $X$. Denote by $\Ha\left(X\right)_1$ the total space  of the \'etal\'e space of the associated sheaf. Denote the map to $X$ by $s$. The stalk at $x$ is the set of germs of locally defined homeomorphisms (which no longer need to fix $x$). If $germ_x\left(f\right)$ is one such germ, the element $f\left(x\right) \in X$ is well-defined. We assemble this into a map $$t:\Ha\left(X\right)_1 \to X.$$ This extends to a natural structure of an \'etale $S$-groupoid $\Ha\left(X\right)$ with objects $X$, called the \textbf{Haefliger groupoid} of $X$.
\end{dfn}

\begin{rmk}
In literature, the Haefliger groupoid is usually denoted by $\Gamma\left(X\right)$, but, we wish to avoid the clash of notation with the stack of sections $2$-functor.
\end{rmk}

\begin{prop}
For $\h$ an \'etale $S$-groupoid, there is a canonical map $$\tilde \iota_{\h}:	\h \to \Ha\left(\h_0\right).$$
\end{prop}
\begin{proof}
For each $h \in \h_1$, choose a neighborhood $U$ such that $s$ and $t$ restrict to embeddings. Then $h$ induces a homeomorphism $$s\left(h\right) \in s\left(U\right) \to t\left(U\right) \ni t\left(h\right),$$ namely $t\circ s^{-1}|_{U}$. Define $\tilde \iota_{\h}$ by having it be the identity on objects and having it send an arrow $h$ to the germ at $s\left(h\right)$ of $t\circ s^{-1}|_{U}$. This germ clearly does not depend on the choice of $U$.
\end{proof}

The following proposition is immediate:

\begin{prop}
An \'etale $S$-groupoid $\h$ is effective if and only if $\tilde \iota_{\h}$ is faithful.
\end{prop}

\begin{dfn}
Let $\h$ be an \'etale $S$-groupoid. The \textbf{effective part} of $\h$ is the image in $\Ha\left(\h_0\right)$ of $\tilde \iota_{\h}$. It is denoted by  $\Ef\left(\h\right)$. This is an open subgroupoid, so it is clearly effective and \'etale. We will denote the canonical map $\h \to \Ef \h$ by $\iota_{\h}$.
\end{dfn}

\begin{rmk}
$\h$ is effective if and only if $\iota_{\h}$ is an isomorphism.
\end{rmk}

\subsection{\'Etale invariance}
Unfortunately, the assignment $\h \mapsto \Ef\left(\h\right)$ is not functorial with respect to all maps, that is, a morphism of \'etale $S$-groupoids need not induce a morphism between their effective parts. However, there are classes of maps for which this assignment is indeed functorial. In this subsection, we shall explore this functoriality.
\begin{dfn}\label{dfn:etalinv}
Let $P$ be a property of a map of spaces which forms a subcategory of $S$. We say that $P$ is \textbf{\'etale invariant} if the following two properties are satisfied:
\begin{itemize}
\item[i)] $P$ is stable under pre-composition with local homeomorphisms
\item[ii)] $P$ is stable under pullbacks along local homeomorphisms.
\end{itemize}
If in addition, every morphism in $P$ is open, we say that $P$ is a class of \textbf{open \'etale invariant} maps. Examples of such open \'etale invariant maps are open maps, local homeomorphisms, or, in the smooth setting, submersions. We say a map $\psi:\G \to \h$ of \'etale $S$-groupoids has property $P$ if both $\psi_0$ and $\psi_1$ do. We denote corresponding $2$-category of $S$-groupoids as $\left(S-\Gpd\right)^{et}_{P}$. We say a morphism $$\varphi:\Y \to \X$$ has property $P$ is there exists a homomorphism of \'etale $S$-groupoids $$\psi:\G \to \h$$ with property $P$, such that $$\varphi \cong \left[\psi\right].$$
\end{dfn}

\begin{rmk}
This agrees with the definition of a local homeomorphism of \'etale stacks given in Definition \ref{dfn:localhomeostacks} in the case $P$ is local homeomorphisms. When $P$ is open maps, under the correspondence between \'etale stacks and \'etendues (Theorem \ref{thm:etendue}), this agrees with the notion of an open map of topoi in the sense of \cite{elephant2}. When $P$ is submersions, this is equivalent to the definition of a submersion of smooth \'etendues given in \cite{Ie}.
\end{rmk}

\begin{rmk}
Notice that being \'etale invariant implies being invariant under restriction and local on the target, as in Definition \ref{dfn:local}.
\end{rmk}

\begin{prop}
Let $P$ be a property of a map of spaces which forms a subcategory of $S$. $P$ is \'etale invariant if and only if the following conditions are satisfied:
\begin{itemize}
\item[i)] every local homeomorphism is in $P$
\item[ii)] for any commutative diagram
$$\xymatrix{W \ar[d]_-{g} \ar[r]^-{f'} & Y \ar[d]^-{g'}\\
X \ar[r]^-{f} & Z,}$$
with both $g$ and $g'$ local homeomorphisms, if $f$ has property $P$, then so does $f'$.
\end{itemize}
\end{prop}

\begin{proof}
Suppose that $P$ is \'etale invariant. Then, as $P$ is a subcategory, it contains all the identity arrows, and since it is stable under pre-composition with local homeomorphisms, this implies that every local homeomorphism is in $P$. Now suppose that $f \in P$, and $$\xymatrix{W \ar[d]_-{g} \ar[r]^-{f'} & Y \ar[d]^-{g'}\\
X \ar[r]^-{f} & Z,}$$ is commutative with both $g$ and $g'$ local homeomorphisms. Then as $P$ is stable under pullbacks along local homeomorphisms, the induced map $X \times_{Z} Y \to Y$ has property $P$. Moreover, as local homeomorphisms are invariant under change of base (Definition \ref{dfn:local}), the induced map $X \times_{Z} Y \to X$ is a local homeomorphism. It follows that the induced map $W \to X \times_{Z} Y$ is a local homeomorphism, and since $f'$ can be factored as $$W \to X \times_{Z} Y \to Y,$$ and $P$ is stable under pre-composition with local homeomorphisms, it follows that $f'$ has property $P$.
Conversely, suppose that the conditions of the proposition are satisfied. Condition $ii)$ clearly implies that $P$ is stable under pullbacks along local homeomorphisms. Suppose that $e:W \to X$ is a local homeomorphism and $f:X \to Z$ is in $P$. Then as
$$\xymatrix{W \ar[d]_-{e} \ar[r]^-{f \circ e} & Z \ar[d]^-{id_Z}\\
X \ar[r]^-{f} & Z,}$$
commutes, it follows  that the $f \circ e$ has property $P$.
\end{proof}

\begin{lem}\label{lem:etin1}
For any open \'etale invariant $P$, the assignment $$\h \mapsto \Ef\left(\h\right)$$ extends to a $2$-functor

$$\Ef_P:\left(S-\Gpd\right)^{et}_{P} \to \Ef\left(S-\Gpd\right)^{et}_{P}$$
from \'etale $S$-groupoids and $P$-morphisms to effective \'etale $S$-groupoids and $P$-morphisms.
\end{lem}

\begin{proof}
Suppose $\varphi:\G \to \h$ has property $P$. Since $\Ef$ does not affect objects, we define $$\Ef\left(\varphi\right)_0=\varphi_0.$$ Given $g \in \G_1$, denote its image in $\Ef\left(\G\right)_1$ by $[g]$. Define $$\Ef\left(\varphi\right)_1\left([g]\right)=\left[\varphi\left(g\right)\right].$$ We need to show that this is well defined. Suppose that $[g]=[g']$. Let $V_g$ and $V_{g'}$ be neighborhoods of $g$ and $g'$ respectively, on which both $s$ and $t$ restrict to embeddings. Denote by $x$ the source of $g$ and $g'$. Then there exists a neighborhood $W$ of $x$ over which $$t \circ s^{-1}|_{g}$$ and $$t \circ s^{-1}|_{g'}$$ agree. Since $\varphi_1$ has property $P$, it is open, so $\varphi_1\left(V_g\right)$ is a neighborhood of $\varphi_1\left(g\right)$, and similarly for $g'$. By making $V_g$ and $V_g'$ smaller if necessary, we may assume that $s$ and $t$ restrict to embeddings on $\varphi_1\left(V_g\right)$ and $\varphi_1\left(V_g'\right)$. Since $\varphi$ is a groupoid homomorphism, it follows that $$t \circ s^{-1}|_{\varphi_1\left(V_g\right)}$$ and  $$\varphi_0 \circ t \circ s^{-1}|_{V_g}$$ agree on $W$, and similarly for $g'$. Hence, if $g$ and $g'$ induce the same germ of a locally defined homeomorphism, so do $\varphi_1\left(g\right)$ and $\varphi_1\left(g'\right)$. It is easy to check that $\Ef\left(\varphi\right)$ as defined is a homomorphism of $S$-groupoids. In particular, the following diagram commutes:

$$\xymatrix@C=1.5cm{\Ef\left(\G\right)_1 \ar[d]_-{s} \ar[r]^-{\Ef\left(\varphi\right)_1} & \Ef\left(\h\right) \ar[d]^-{s}\\
\G_0 \ar[r]^-{\varphi_0} & \h_0.}$$
Since $P$ is \'etale invariant and the source maps are local homeomorphisms, it implies that $\Ef\left(\varphi\right)_1$ has property $P$. The rest is proven similarly.
\end{proof}

\begin{thm}\label{thm:bungalo}
Let $j_P:\Ef\left(S-\Gpd\right)^{et}_{P} \hookrightarrow \left(S-\Gpd\right)^{et}_{P}$ be the inclusion. Then $\Ef_P$ is left-adjoint to $j_P$.
\end{thm}

\begin{proof}
There is a canonical natural isomorphism $$\Ef_P \circ j_P \Rightarrow id_{\Ef\left(S-\Gpd\right)^{et}_{P}}$$ since any effective \'etale groupoid is canonically isomorphic to its effective part. Furthermore, the maps $\iota_{\h}$ assemble into a natural transformation $$\iota:id_{\left(S-\Gpd\right)^{et}_{P}} \Rightarrow j_P \circ \Ef_P.$$ It is easy to check that these define a $2$-adjunction.
\end{proof}

\begin{thm}
$\Ef_P$ sends Morita equivalences to Morita equivalences.
\end{thm}
\begin{proof}
Suppose $\varphi:\G \to \h$ is a Morita equivalence. Since $\G$ and $\h$ are \'etale, this implies $\varphi$ is a local homeomorphism. Hence, in the pullback diagram
$$\xymatrix{\h_1\times_{\h_0}\G_0 \ar[d]_-{pr_1} \ar[r]^-{pr_2} & \G_0 \ar[d]^-{\varphi_0}\\
\h_1 \ar[r]^-{s} & \h_0,}$$
$pr_1$ is a local homeomorphism, and hence the map $$t \circ pr_1:\h_1 \times_{\h_0} \G_0 \to \h_0$$ is as well. We have a commutative diagram

$$\xymatrix{\h_1 \times_{\h_0} \G_0 \ar[r]^{t \circ pr_1} \ar[d] & \h_0\\
\Ef\left(\h\right)_1 \times_{\h_0} \G_0. \ar[ur]_-{t \circ pr_1}  &   }$$
The map $$\Ef\left(\h\right)_1 \times_{\h_0} \G_0 \to \h_1\times_{\h_0} \G_0 $$ is the pullback of a local homeomorphism, hence one itself, and the upper arrow $t \circ pr_1$ is a local homeomorphism. This implies $$t \circ pr_1:\Ef\left(\h\right)_1 \times_{\h_0} \G_0 \to \h_0$$ is a local homeomorphism as well. In particular, it admits local sections. Therefore $\Ef\left(\varphi\right)$ is essentially surjective.
Now suppose that $$\left[h\right]:\varphi\left(x\right) \to \varphi\left(y\right).$$ Then  $$h:\varphi\left(x\right) \to \varphi\left(y\right).$$ So there is a unique $g:x \to y$ such that $\varphi\left(g\right)=h$. Now suppose $$\left[h\right]=\left[h'\right].$$ We can again choose a unique $g'$ such that $\varphi\left(g'\right)=h'.$ We need to show that $$\left[g\right]=\left[g'\right].$$ Let $V_g$ and $V_{g'}$ be neighborhoods of $g$ and $g'$ respectively chosen so small that $s$ and $t$ of $\G_1$ restrict to embeddings on them, $s$ and $t$ of $\h_1$ restrict to embeddings on $\varphi_1\left(V_g\right)$ and $\varphi_1\left(V_g'\right),$ $\varphi_0$ restricts to an embedding on $s\left(V_g\right),$ which is possible since $\varphi_0$ is a local homeomorphism, and $$t \circ s^{-1}|_{\varphi_1\left(V_g\right)}$$ and $$t\circ s^{-1}|_{\varphi_1\left(V_g'\right)}$$ agree on $\varphi_0\left(s\left(V_g\right)\right),$ which is possible since $$\left[\varphi\left(g\right)\right]=\left[\varphi\left(g'\right)\right].$$ Then by the proof of Lemma \ref{lem:etin1}, $$t \circ s^{-1}|_{\varphi_1\left(V_g\right)}$$ and  $$\varphi_0 \circ t \circ s^{-1}|_{V_g}$$ agree on $s\left(V_g\right)$, and similarly for $g'$. Hence $$\varphi_0 \circ t \circ s^{-1}|_{V_g}$$ and $$\varphi_0 \circ t \circ s^{-1}|_{V_g'}$$ agree on $W$, but $\varphi_0$ is an embedding when restricted to $W$, hence $$t \circ s^{-1}|_{V_g}$$ and $$t \circ s^{-1}|_{V_g'}$$ agree on $W$ so $\left[g\right]=\left[g'\right].$
\end{proof}

\begin{lem}
Let $\mathcal{U}$ be an \'etale cover of $\h_0$, with $\h$ an \'etale $S$-groupoid. Then there is a canonical isomorphism between $\Ef\left(\h_{\mathcal{U}}\right)$ and $\left(\Ef\left(\h\right)\right)_{\mathcal{U}}$ (See Definition \ref{dfn:cech}).
\end{lem}
\begin{proof}
Both of these groupoids have the same object space. It suffices to show that their arrow spaces are isomorphic (and that this determines an internal functor). Suppose the cover $\mathcal{U}$ is given by a local homeomorphism $e:U \to \h_0$. An arrow in $\h_{\mathcal{U}}$ is a triple $$\left(h,p,q\right)$$ with $$h:e\left(p\right) \to e\left(p\right).$$ An arrow in $\left(\Ef\left(\h\right)\right)_{\mathcal{U}}$ is a triple $$\left(\left[h\right],p,q\right)$$ such that $\left[h\right]$ is the image of an arrow $h \in \h_1$ under $\iota_\h$ such that $$h:e\left(p\right) \to e\left(p\right).$$ Define a map
\begin{align}
\left(\h_{\mathcal{U}}\right)_1 &\to  \left(\left(\Ef\left(\h\right)\right)_{\mathcal{U}}\right)_1 \nonumber\\
\left(h,p,q\right) &\mapsto \left(\left[h\right],p,q\right)\label{eq:cake}.\end{align}
This map is clearly surjective.

We make the following claim:

$$\left[h\right]=\left[h'\right]$$
if and only if
$$\left[\left(h,p,q\right)\right]=\left[\left(h',p,q\right)\right].$$
Suppose that $$\left[h\right]=\left[h'\right].$$
Pick a neighborhood $U_h$ of $h$ in $\h_1$ such that both $s$ and $t$ are injective over it, and $U_h'$ an analogous neighborhood of $h'$. Let $W$ be a neighborhood of $s\left(h\right)=s\left(h'\right)$ over which
\begin{equation}\label{eq:yikes}
t\circ s|_{U_h}^{-1}=t\circ s|_{U_h'}^{-1}.
\end{equation}
Pick neighborhoods $V_p$ and $V_q$ of $p$ and $q$ respectively so small that $e$ is injective over them, and for all $a \in V_p$, $$e\left(a\right) \in W$$ and $$t\circ s|_{U_h}^{-1}\left(e\left(a\right)\right) \in e\left(V_q\right).$$
As the arrow space $\left(\h_{\mathcal{U}}\right)_1$ fits into the pullback diagram $$\xymatrix{\left(\h_{\mathcal{U}}\right)_1 \ar[r] \ar[d]_-{\left(s,t\right)} & \h_1 \ar[d]^-{\left(s,t\right)} \\
U \times U \ar[r] & \h_0 \times \h_0,}$$
$\left(V_p \times V_q \times U_h \right) \cap \left(\h_{\mathcal{U}}\right)_1$ is a neighborhood of $\left(h,p,q\right)$ over which both the source and target maps are injective. The set $\left(V_p \times V_q \times U_h' \right) \cap \left(\h_{\mathcal{U}}\right)_1$ is an analogous neighborhood of $\left(h',p,q\right)$. The local inverse of $s$ through $\left(h,p,q\right)$ is then given by
$$a \mapsto \left(s|_{U_h}^{-1}\left(e\left(a\right)\right),a,e|_{V_p}^{-1}\left(t\circ s|_{U_h}^{-1}\left(e\left(a\right)\right)\right)\right).$$
Hence, the germ associated to $\left(h,p,q\right)$ is the germ of $$a \mapsto e|_{V_p}^{-1}\left(t\circ s|_{U_h}^{-1}\left(e\left(a\right)\right)\right).$$ Similarly the germ associated to $\left(h',p,q\right)$ is the germ of $$a \mapsto e|_{V_p}^{-1}\left(t\circ s|_{U_h'}^{-1}\left(e\left(a\right)\right)\right).$$ From equation (\ref{eq:yikes}), it follows that these maps are identical. Moreover, supposing instead that $$\left[\left(h,p,q\right)\right]=\left[\left(h',p,q\right)\right],$$ by the above argument, it follows that $\left[h\right]=\left[h'\right]$ since $e$ is injective over $V_q$.

Hence the assignment (\ref{eq:cake}) depends only  on the  image of $\left(h,p,q\right)$ in $\Ef\left(\h_{\mathcal{U}}\right)$. So there is an induced well defined and surjective map
\begin{equation}\label{eq:turtle}
\left(\Ef\left(\h_{\mathcal{U}}\right)\right)_1 \to  \left(\left(\Ef\left(\h\right)\right)_{\mathcal{U}}\right)_1.
\end{equation}
Since, $\left[h\right]=\left[h'\right]$ implies $\left[\left(h,p,q\right)\right]=\left[\left(h',p,q\right)\right],$ it follows  that this map is also injective, hence bijective. It is easy to check that it is moreover a homeomorphism. It clearly defines a groupoid homomorphism
\end{proof}

\begin{cor}\label{cor:etadj}
There is an induced $2$-adjunction
$$\xymatrix{\Eft_{P}  \ar@<-0.5ex>[r]_-{j_P}  & \Et_{P}\ar@<-0.5ex>[l]_-{\Ef_P},}$$
between \'etale stacks with $P$-morphisms and effective \'etale stacks with $P$-morphisms, where $j_P$ is the canonical inclusion.
\end{cor}
\begin{proof}
Let $\mathcal{U}$ be an \'etale cover of $\h_0$, with $\h$ an \'etale $S$-groupoid. From the previous lemma, there is a canonical isomorphism between $\Ef\left(\h_{\mathcal{U}}\right)$ and $\left(\Ef\left(\h\right)\right)_{\mathcal{U}}.$ Let $\G$ be an effective \'etale $S$-groupoid. Then
\begin{eqnarray*}
\Hom\left(\left[\Ef\left(\h\right)\right],\left[\G\right]\right)  &\simeq& \underset{\mathcal{U}} \hc \Hom\left(\left(\Ef\left(\h\right)\right)_{\mathcal{U}}, \G\right)\\
&\simeq& \underset{\mathcal{U}} \hc \Hom\left(\Ef\left(\h_{\mathcal{U}}\right), \G \right) \\
&\simeq& \underset{\mathcal{U}} \hc \Hom\left(\h_{\mathcal{U}},j_P\G\right)\\
&\simeq& \Hom\left(\left[\h\right],j_P\left[\G\right]\right).
\end{eqnarray*}
\end{proof}
Note that this implies that $\Eft_{P}$ is a localization of $\Et_{P}$ with respect to those morphisms whose image under $\Ef_P$ become equivalences. When $P$ is local homeomorphisms, denote $P=et$. We make the following definition for later:

\begin{dfn}\label{dfn:effectlocal}
A morphism $\varphi:\Y \to \X$ between \'etale stacks is called an \textbf{effective local equivalence} if $\varphi$ is a local homeomorphism and $\Ef_{et}\left(\varphi\right)$ is an equivalence.
\end{dfn}

\section{The Amazing Properties of The Haefliger Groupoid}\label{sec:amazing}
\subsection{Properties.}
We now go on to derive many of the surprising properties of the Haefliger groupoid construction. These results will play a pivotal role in the rest of the paper.
\begin{dfn}
Let $X$ be a space. An \textbf{$X$-manifold} is a space $Y$ which admits a covering $$\left(U_\alpha \to Y\right)$$ by open subsets of $X.$ Denote the full  subcategory of $S$ spanned by $X$-manifolds by $\xm.$
\end{dfn}

\begin{rmk}\label{rmk:xmfdsh}
Note that if $f:Y \to X$ is a local homeomorphism, then $Y$ is an $X$-manifold. In particular, one has $$S^{et}/X \simeq \xm^{et}/X \simeq \Sh\left(X\right),$$ where $\xm^{et}$ denotes the subcategory of $\xm$ spanned by local homeomorphisms, and similarly for $S^{et}$.
\end{rmk}

\begin{rmk}\label{rmk:Xmfd}
For any $X,$ $\xm$ is a category of spaces in the sense of \cite{etalspme}, over topological spaces.
\end{rmk}

\begin{rmk}
If $S$ is smooth manifolds and $X$ is an $n$-manifold, then $Y$ is an $X$-manifold if and only if $Y$ is an $n$-manifold, if and only if $X$ is a $Y$ manifold. If instead of taking a space $X$ of fixed dimension we let $$X:=\coprod\limits^\infty_{n=0} \RR^n,$$ then every space in $S$ is an $X$-manifold. More generally, this holds with each $\RR^n$ replace by any $n$-manifold. Similarly for other types of manifolds.
\end{rmk}

\begin{thm}\label{thm:eqltoposhaf}
For any space $X,$ there is a canonical equivalence of topoi $$\Sh\left(\left[\Ha\left(X\right)\right]\right) \simeq \Sh\left(\xm^{et}\right),$$ where $\xm^{et}$ denotes the subcategory of $\xm$ spanned by local homeomorphisms, and sheaves are taken with respect to the induced open-cover topology.
\end{thm}

\begin{proof}
Consider the representable sheaf $y\left(X\right)$ in $\Sh\left(\xm^{et}\right).$ We claim that the canonical projection $$y\left(X\right) \to 1$$ to the terminal object is an effective epimorphism. This claim is equivalent to the statement that for any $X$-manifold $Y,$ there exists a cover $$U_\alpha \to Y$$ in $\xm^{et}$ such that each element of the cover admits some local homeomorphism $$U_\alpha \to X.$$ This is now clear by definition of $X$-manifold. It follows that the induced \'etale geometric morphism $$\Sh\left(\xm^{et}\right)/y\left(X\right) \to \Sh\left(\xm^{et}\right)$$ is of effective descent (i.e. is an effective epimorphism). Note that there is a canonical equivalence $$\Sh\left(\xm^{et}\right)/y\left(X\right) \simeq \Sh\left(\xm^{et}/X\right).$$ By Remark \ref{rmk:xmfdsh}, $$\xm^{et}/X \simeq \Sh\left(X\right).$$ Moreover, it is easy to check that the induced Grothendieck topology on $\xm^{et}/X$ agrees with the canonical topology on $\Sh\left(X\right),$ hence one has $$\Sh\left(\xm^{et}\right)/y\left(X\right) \simeq \Sh\left(X\right).$$ Denote by $p$ the effective \'etale epimorphism $$p:\Sh\left(X\right) \to \Sh\left(\xm^{et}\right).$$ The existence of $p$ already implies that $\Sh\left(\xm^{et}\right)$ is an \'etendue (see Definition \ref{dfn:etendue}). In particular, the pullback topos
$$\xymatrix{P \ar[d] \ar[r] & \Sh\left(X\right) \ar[d]^-{p}\\
\Sh\left(X\right) \ar[r]^-{p} & \Sh\left(\xm^{et}\right)}$$
has the structure of an \'etale $S$-groupoid $\h$ such that $$\b\h \simeq \Sh\left(\xm^{et}\right),$$ see Remark \ref{rmk:etendskt}.
Note that this pullback is also the pullback topos
$$\xymatrix{P \ar[d] \ar[r] & \Sh\left(\xm^{et}\right)/y\left(X\right) \ar[d]\\
\Sh\left(X\right) \ar[r]^-{p} & \Sh\left(\xm^{et}\right).}$$
Hence $P$ is equivalent to the topos of sheaves on the \'etal\'e space of the sheaf $$p^*\left(y\left(X\right)\right) \in \Sh\left(X\right).$$ This sheaf assigns an open subset $U$ of $X$ the set of local homeomorphisms $$U \to X.$$ It thus can be identified with the sheafification of the presheaf $\mathit{Emb}$ defined in Definition \ref{dfn:haf}. It is easy to check that $\h$ is fact is isomorphic to the Haefliger groupoid $\Ha\left(X\right).$ The result now follows.
\end{proof}

We have the following immediate corollary:

\begin{cor}\label{cor:morithaf}
If $X$ and $Y$ are such that $\xm = Y\mbox{-}\mathbf{Mfd},$ then $\Ha\left(X\right)$ and $\Ha\left(Y\right)$ are Morita equivalent. If $M$ and $N$ are two $n$-manifolds, then $\Ha\left(M\right)$ and $\Ha\left(N\right)$ are Morita equivalent. In particular, for any $n$-manifold $M,$ $\Ha\left(M\right)$ is Morita equivalent to $\Ha\left(\RR^n\right).$
\end{cor}

\begin{dfn}
By Remark \ref{rmk:Xmfd}, we may replace the role of the category $S$ of spaces by $\xm.$ We can therefore consider \textbf{\'etale $\xm$-stacks}, that is stacks over $\xm$ arising from \'etale groupoid objects in $\xm.$ Denote the associated $2$-category by $\Et\left(\xm\right).$
\end{dfn}

\begin{rmk}
When $S$ is smooth manifolds and $X=\coprod\limits^\infty_{n=0} \RR^n,$ $\Et\left(\xm\right)$ is the bicategory of \'etale differentiable stacks, as in Definition \ref{dfn:etalestack}.
\end{rmk}

\begin{rmk}\label{rem:xmff}
We have a canonical inclusion functor $i_X:\xm \to S,$ which induces a restriction $2$-functor
$$\left(i_X\right)^*:\St\left(S\right) \to \St\left(\xm\right).$$ It has a left adjoint
$$\left(i_X\right)_!:\St\left(\xm\right) \to \St\left(S\right);$$ $\left(i_X\right)_!$ is the left Kan extension $\Lan_{y_X}\left(y_S \circ i_X\right),$ where $y_X$ and $y_S$ are the Yoneda embeddings into stacks of $\xm$ and $S$ respectively. The unit of the adjunction $$\left(i_X\right)_! \dashv \left(i_X\right)^*$$ is always an equivalence, hence $\left(i_X\right)_!$ is full and faithful. Moreover, if $\X \simeq \left[\G\right]_{\xm}$ is an \'etale $\xm$-stack, $$\X \simeq \hc \left(\G_2 \rrrarrow \G_1 \rrarrow \G_0\right),$$ and hence
\begin{eqnarray*}
\left(i_X\right)_!\left(\left[\G\right]_{\xm}\right) &\simeq& \hc \left(\G_2 \rrrarrow \G_1 \rrarrow \G_0\right)\\
&\simeq& \left[\G\right]_{S}.
\end{eqnarray*}
Hence, $\left(i_X\right)_!$ restricts to a full and faithful embedding $$\Et\left(\xm\right) \to \Et\left(S\right).$$
Therefore, we may always choose to view an \'etale $\xm$-stack as an \'etale $S$-stack arising from an \'etale groupoid object in $\xm.$ It is easy to see that $\X$ is in the essential image of $\left(i_X\right)_!$ if and only if it admits an \'etale atlas $$Y \to \X,$$ from an $X$-manifold $Y.$
\end{rmk}

\begin{dfn}
Denote by $\HH\left(X\right)$ the \'etale $\xm$-stack $\left[\Ha\left(X\right)\right].$ It is the \textbf{Haefliger stack} of $X.$
\end{dfn}

Denote by $\Et\left(\xm\right)^{et}$ the bicategory of \'etale $\xm$-stacks and local homeomorphisms between them.

\begin{lem}\label{lem:contrc}
Let $V \subset X$ be an open subset. Then the groupoid $$\Hom_{\Et\left(\xm\right)^{et}}\left(V,\HH\left(X\right)\right)$$ of local homeomorphisms from $V$ to the Haefliger stack $\HH\left(X\right)$ is contractible.
\end{lem}

\begin{proof}
Consider the map $$V \hookrightarrow X \to \HH\left(X\right),$$ and denote it by $\varphi.$ It is a local homeomorphism, and hence an object in $$\Hom_{\Et\left(\xm\right)^{et}}\left(V,\HH\left(X\right)\right).$$ First, we will show that any other object in this groupoid admits a morphism from $\varphi.$

Let $\h$ be an arbitrary \'etale $S$-groupoid. Denote by $\Hom^{et}\left(\h,\Ha\left(X\right)\right)$ the full subgroupoid of the groupoid of internal functors spanned by those functors $$\varphi:\h \to \Ha\left(X\right)$$ with $\varphi_0$ a local homeomorphism. It follows from (\ref{eq:pizza}) of Appendix \ref{sec:prelim} that
$$\Hom_{\Et\left(\xm\right)^{et}}\left(V,\HH\left(X\right)\right) \simeq \underset{\mathcal{U} \in Cov\left(V\right)} \hc \Hom^{et} \left(V_{\mathcal{U}},\Ha\left(X\right)\right),$$
where the groupoid $V_{\mathcal{U}}$ is the \v{C}ech groupoid (Definition \ref{dfn:cech}).
Let $$\psi:V_{\mathcal{U}} \to \Ha\left(X\right)$$ an arbitrary object of $\Hom^{et} \left(V_{\mathcal{U}},\Ha\left(X\right)\right).$ To show that $\Hom_{\Et\left(\xm\right)^{et}}\left(V,\HH\left(X\right)\right)$ is connected, it suffices to show that there exists an internal natural transformation $$\varphi \circ p_\mathcal{U} \Rightarrow \psi,$$ where $$p_\mathcal{U}:V_\mathcal{U} \to V$$ is the canonical morphism of groupoids.
For this, it will be helpful to have another description of the groupoid $\Hom^{et} \left(V_{\mathcal{U}},\Ha\left(X\right)\right).$

Lets first look at the objects. Without loss of generality, suppose that the cover $\mathcal{U}$ is an open cover $\left(U_\alpha \to V\right).$ The functor $\psi$ then corresponds to cocycle data $\left(\psi_\alpha,\psi_{\alpha,\beta}\right),$ where $$\psi_\alpha:U_\alpha \to X,$$ $$\psi_{\alpha,\beta}:U_\alpha \cap U_\beta \to \Ha\left(X\right)_1,$$ such that for all $$x \in U_\alpha \cap U_\beta,$$
\begin{equation}
s\left(\psi_{\alpha,\beta}\left(x\right)\right)=\psi_\beta\left(x\right),
\end{equation}
and
\begin{equation}\label{eq:2}
t\left(\psi_{\alpha,\beta}\left(x\right)\right)=\psi_\alpha\left(x\right),
\end{equation}
subject to the cocycle condition that on all $$z \in U_\alpha \cap U_\beta \cap U_\delta,$$
\begin{equation}\label{eq:cocycccc}
\psi_{\alpha,\beta}\left(z\right) \psi_{\beta,\delta}\left(z\right)=\psi_{\alpha,\delta}\left(z\right).
\end{equation}
Let $i_{\alpha,\beta}:U_\alpha \cap U_\beta \hookrightarrow U_\alpha$ be the canonical inclusion. Then the pair $\left(i_{\alpha,\beta},\psi_{\beta,\alpha}\right)$ corresponds to a unique morphism $$\tilde \psi_{\beta,\alpha}:U_\alpha \cap U_\beta \to \psi_\alpha^*\left(\Ha\left(X\right)_1\right),$$ where $\psi_\alpha^*\left(\Ha\left(X\right)_1\right)$ is the following pullback diagram
$$\xymatrix{\psi_\alpha^*\left(\Ha\left(X\right)_1\right) \ar[d] \ar[r] & \Ha\left(X\right)_1 \ar[d]^-{s} \\
U_\alpha \ar[r]^-{\psi_\alpha} & X.}$$
Recall that $$s:\Ha\left(X\right)_1 \to X$$ is the \'etal\'e space of the sheaf $F$ which assigns an open subset $W$ of $X$ the set of local homeomorphisms $$W \to X.$$ As such, $\tilde \psi_{\beta,\alpha}$ corresponds to a section of the pullback sheaf $\psi_\alpha^*\left(F\right).$ Notice that since $\psi_\alpha$ is a local homeomorphism, it is open, and hence for each open subset $$N \subset U_\alpha,$$ $$\psi_\alpha^*\left(F\right)\left(N\right)=F\left(\psi_\alpha\left(N\right)\right).$$ It follows that $\tilde \psi_{\beta,\alpha}$ corresponds to a local homeomorphism $$\gamma_{\beta,\alpha}:\psi_\alpha\left(U_\alpha \cap U_\beta\right) \to X.$$ Explicitly, viewing $\Ha\left(X\right)_1$ (as a set) as the disjoint union of stalks of $F,$ one has
\begin{equation}\label{eq:reco}
\mathit{germ}_{\psi_\alpha\left(x\right)}\left(\gamma_{\beta,\alpha}\right)=\psi_{\beta,\alpha}\left(x\right).
\end{equation}
By equation (\ref{eq:2}), it follows that $\gamma_{\beta,\alpha}$ restricts to a local homeomorphism such that
\begin{eqnarray*}
\gamma_{\beta,\alpha}:\psi_\alpha\left(U_\alpha \cap U_\beta\right) &\to& \psi_\beta\left(U_\alpha \cap U_\beta\right)\\
\psi_\alpha\left(x\right) &\mapsto&  \psi_\beta\left(x\right).
\end{eqnarray*}
Notice that each $\psi_{\alpha,\beta}$ can be recovered from $\gamma_{\alpha,\beta}$ by equation (\ref{eq:reco}). Hence the data $\left(\psi_\alpha,\psi_{\alpha,\beta}\right)$ and $\left(\psi_\alpha,\gamma_{\alpha,\beta}\right)$ are equivalent. The cocycle condition (\ref{eq:cocycccc}) translate into the condition that for all $$z \in \psi_\delta\left( U_\alpha \cap U_\beta \cap U_\delta\right),$$ $$\gamma_{\alpha,\beta}\left(\gamma_{\beta,\delta}\left(z\right)\right)= \gamma_{\alpha,\delta}\left(z\right).$$
By a similar argument, one concludes that morphisms $$\omega:\left(\psi_\alpha,\gamma_{\alpha,\beta}\right) \to \left(\theta_\alpha,\rho_{\alpha,\beta}\right)$$ can be described as the data of for each $\alpha,$ a local homeomorphism
$$\omega_\alpha:\psi_\alpha\left(U_\alpha\right) \to \theta_\alpha\left(U_\alpha\right)$$
such that for $x \in U_\alpha,$
\begin{equation}\label{eq:iddd}
\psi_\alpha\left(x\right) \mapsto \theta_\alpha\left(x\right),
\end{equation}
subject to the natural compatibility condition that for all $\alpha$ and $\beta,$ the following diagram commutes:
$$\xymatrix@C=2.5cm{\psi_\alpha\left(U_\alpha \cap U_\beta\right) \ar[r]^-{\gamma_{\beta,\alpha}} \ar[d]_-{\omega_\alpha} & \psi_\beta\left(U_\alpha \cap U_\beta\right) \ar[d]^-{\omega_\beta}\\
\theta_\alpha\left(U_\alpha \cap U_\beta\right) \ar[r]_-{\rho_{\beta,\alpha}} & \theta_\beta\left(U_\alpha \cap U_\beta\right).}$$
Let us now return to constructing $$\omega:\varphi \circ p_\mathcal{U} \Rightarrow \psi.$$ Notice that $\varphi \circ p_\mathcal{U}$ corresponds to the data $\left(j_\alpha,id_{U_\alpha \cap U_\beta}\right),$ where each $$j_\alpha:U_\alpha \hookrightarrow X$$ is the canonical inclusion. It easy to check now that the assignment to each $\alpha$ the map $$\omega_\alpha:=\psi_\alpha:U_\alpha \to \psi_\alpha\left(U_\alpha\right)$$ determines valid data for a morphism $\omega.$ Hence, $\Hom_{\Et\left(\xm\right)^{et}}\left(V,\HH\left(X\right)\right)$ is non-empty and connected, and equivalent to the group of automorphisms of $\varphi.$

Notice that the canonical functor $$\Hom^{et}\left(V,\Ha\left(X\right)\right) \to \Hom_{\Et\left(\xm\right)^{et}}\left(V,\HH\left(X\right)\right),$$ is full and faithful. It therefore suffices to consider automorphisms of $$\varphi:V \to \Ha\left(X\right)$$ in $\Hom^{et}\left(V,\Ha\left(X\right)\right).$ This corresponds to the special case where $\mathcal{U}$ is the singleton cover by all of $V.$ Any morphism $$\omega:\varphi \Rightarrow \varphi$$ corresponds to a singleton map $$\omega:V \to V,$$ which by (\ref{eq:iddd}), must be the identity of $V.$ It follows that $\Hom_{\Et\left(\xm\right)^{et}}\left(V,\HH\left(X\right)\right)$ is contractible.
\end{proof}

\begin{thm}\label{thm:terminal}
The Haefliger stack $\HH\left(X\right)$ is a terminal object in the bicategory $\Et\left(\xm\right)^{et}.$
\end{thm}

\begin{proof}
It suffices to show that for any \'etale $\xm$-stack $\X,$ the groupoid $$\Hom_{\Et\left(\xm\right)^{et}}\left(\X,\HH\left(X\right)\right)$$ is a contractible. First, suppose that $\X$ is an $X$-manifold $Y.$ Then $Y$ can be written as a colimit of open subsets of $X,$ so it follows that $\Hom_{\Et\left(\xm\right)^{et}}\left(Y,\HH\left(X\right)\right)$ is a homotopy limit of contractible groupoids by Lemma \ref{lem:contrc}, hence contractible. Now, any \'etale $\xm$-stack $\X$ can be written as a homotopy colimit of the form $$\X \simeq \hc \left(\G_2 \rrrarrow \G_1 \rrarrow \G_0\right),$$ where each $\G_i$ is an $X$-manifold. By the same argument, the result follows.
\end{proof}

\begin{rmk}\label{rmk:subterminal}
In the larger bicategory of all \'etale $S$-stacks and \'etale morphisms $\Et^{et}\left(S\right),$ the Haefliger stack $\HH\left(X\right)$ need not be terminal. However, it is \emph{sub}terminal, in that the for all \'etale $S$-stacks $\X,$ the groupoid $\Hom_{\Et\left(\xm\right)^{et}}\left(\X,\HH\left(X\right)\right)$ is either empty or contractible. It is non-empty if and only if $\X$ is an \'etale $\xm$-stack.
\end{rmk}

\subsection{Consequences}
We now derive many consequences of Theorems \ref{thm:eqltoposhaf} and \ref{thm:terminal}.

On one hand, by Theorem \ref{thm:etalsp}, there is an equivalence of bicategories $$\Et\left(\xm\right)^{et}/\HH\left(X\right) \simeq \St\left(\HH\left(X\right)\right).$$ On the other hand, by Theorem \ref{thm:terminal}, $\HH\left(X\right)$ is terminal so the canonical projection $$\Et\left(\xm\right)^{et}/\HH\left(X\right) \to \Et\left(\xm\right)^{et}$$ is an equivalence of bicategories. Hence, one has that the bicategory of \'etale $\xm$-stacks and local homeomorphisms is a $2$-topos, namely $$\Et\left(\xm\right)^{et} \simeq \St\left(\HH\left(X\right)\right).$$ One can say even more. By Theorem \ref{thm:eqltoposhaf} and the Comparison Lemma for stacks, a straight-forward stacky analogue of the theorem in \cite{sga4} III, one has $$\St\left(\HH\left(X\right)\right)\simeq \St\left(\xm^{et}\right).$$ Combining all these results, gives us that there is an equivalence of bicategories
\begin{equation}\label{eq:yay}
\St\left(\xm^{et}\right) \simeq \Et\left(\xm\right)^{et}.
\end{equation}

We will now be a bit more careful, so that we may see what this functor does.

\begin{prop}\label{prop:incri}
The canonical (non-full) inclusion
$$k:\Et\left(\xm\right)^{et} \to \St\left(\xm\right)$$ admits a right adjoint.
\end{prop}
\begin{proof}
First, notice that the following diagram commutes up to natural isomorphism
$$\xymatrix{\Et\left(\xm\right)^{et}/\HH\left(X\right) \ar[r]^-{\sim} \ar[d] & \Et\left(\xm\right)^{et} \ar[d]^-{k}\\
\St\left(\xm\right)/\HH\left(X\right) \ar[r] & \St\left(\xm\right),}$$
where all the $2$-functors are the canonical ones. Under the equivalence $$\St\left(\xm\right)/\HH\left(X\right) \simeq \St\left(\xm/\HH\left(X\right)\right)$$  of Proposition \ref{prop:stand2}, $$\St\left(\xm\right)/\HH\left(X\right) \to \St\left(\xm\right)$$ corresponds to $$l_!:\St\left(\xm/\HH\left(X\right)\right) \to \St\left(\xm\right),$$ which has a right adjoint $l^*,$ where $$l:\xm/\HH\left(X\right) \to \xm$$ is the canonical projection. Furthermore, it follows easily from \cite{etalspme} that $$\Et\left(\xm\right)^{et}/\HH\left(X\right)\to\St\left(\xm\right)/\HH\left(X\right)$$ has a right adjoint. Explicitly, it may be written as the composite
$$\resizebox{4.8in}{!}{$\St\left(\xm\right)/\HH\left(X\right) \stackrel{\sim}{\longrightarrow} \St\left(\xm/\HH\left(X\right)\right) \stackrel{\chi^*}{\longrightarrow} \St\left(\sit\left(\HH\left(X\right),X\right)\right)=\St\left(\HH\left(X\right)\right) \stackrel{L}{\longrightarrow} \Et\left(\xm\right)^{et}/\HH\left(X\right),$}$$ where $$\chi:\sit\left(\HH\left(X\right),X\right) \to \xm/\HH\left(X\right)$$ is the full and faithful inclusion, as in Definition \ref{dfn:sit}, and $L$ is as in Theorem \ref{thm:etalsp}. The result now follows.
\end{proof}

Note that from (\ref{eq:yay}), $\Et\left(\xm\right)^{et}$ is a $2$-topos and hence in particular cocomplete. We have the following immediate corollary of Proposition \ref{prop:incri}:

\begin{cor}\label{cor:42}
The canonical inclusion
$$k:\Et\left(\xm\right)^{et} \to \St\left(\xm\right)$$ preserves and reflects all weak colimits.
\end{cor}

Denote by $\mathcal{O}\left(X\right)^{et}$ the full subcategory of $\xm^{et}$ spanned by the open subsets of $X.$ By Lemma \ref{lem:contrc}, the canonical projection $$\pi:\mathcal{O}\left(X\right)^{et}/\HH\left(X\right) \to \mathcal{O}\left(X\right)^{et}$$ is an equivalence of categories. Denote by $$i:\mathcal{O}\left(X\right)^{et} \hookrightarrow \xm^{et}$$ the canonical inclusion. By the Comparison Lemma for stacks, we have $$\left(i \pi\right)^*:\St\left(\xm^{et}\right) \to \St\left(\sit\left(\Ha\left(X\right)\right)\right)$$ is an equivalence. It follows that the equivalence (\ref{eq:yay}) may be realized as the composite
$$\St\left(\xm^{et}\right) \stackrel{\left(i \pi\right)^*}{\longlongrightarrow} \St\left(\HH\left(X\right)\right) \stackrel{L}{\longrightarrow} \Et\left(\xm\right)^{et}/\HH\left(X\right) \stackrel{\sim}{\longrightarrow} \Et\left(\xm\right)^{et}.$$ Denote this composite by $\Theta.$ In particular, we have proven that $\Theta$ is an equivalence, so we shall record as a theorem to reference later:

\begin{thm}\label{thm:whatetmeans1}
The functor $$\Theta:\St\left(\xm^{et}\right) \to \Et\left(\xm\right)^{et}$$ is an equivalence of bicategories.
\end{thm}

\begin{cor}\label{cor:whatetmeansdiff}
The functor $$\Theta:\St\left(\Mfd^{et}\right) \to \Et\left(\Mfd\right)^{et}$$ is an equivalence between \'etale differentiable stacks and local diffeomorphisms, and stacks on the site of manifolds and local diffeomorphisms.
\end{cor}

\begin{dfn}
Consider the canonical (non-full) inclusion $$j_X:\xm^{et} \to \xm.$$ It induces a \textbf{restriction functor}
$$j_X^*:\St\left(\xm\right) \to \St\left(\xm^{et}\right)$$ with a left adjoint $$\left(j_X\right)_!:\St\left(\xm^{et}\right) \to \St\left(\xm\right),$$ called the \textbf{prolongation functor.}
\end{dfn}

\begin{thm}\label{thm:prol}
$\X \in \St\left(\xm\right)$ is an \'etale $\xm$-stack if and only if it is in the essential image of the prolongation functor $\left(j_X\right)_!.$
\end{thm}

\begin{proof}
Notice that by  Corollary \ref{cor:42}, $$k\circ \Theta:\St\left(\xm^{et}\right) \to \St\left(\xm\right)$$ is weak colimit preserving, with essential image the \'etale $\xm$-stacks $\Et\left(\xm\right).$ We will show that $$k \circ \Theta \simeq \left(j_X\right)_!.$$ Since $k\circ \Theta$ is weak colimit preserving, it may be identified with the weak left Kan extension
\begin{equation}\label{eq:lanj}
\Lan_{y^{et}} \left(k\circ \Theta \circ y^{et}\right),$$ where $$y^{et}:\xm^{et} \hookrightarrow \St\left(\xm^{et}\right)
\end{equation}
is the Yoneda embedding. By construction, for $V$ an open subset of $X,$ one has $$k\circ \Theta \circ y^{et}\left(V\right) \simeq y\left(V\right),$$ where $y$ is the Yoneda embedding. If $Y$ is any $X$-manifold, and $\left(U_\alpha \to Y\right)$ is a covering of $Y$ by open subsets of $X,$ then
$$y^{et}\left(Y\right) \simeq \hc \left(\coprod y^{et}\left(U_\alpha \cap U_\beta \cap U_\gamma\right) \rrrarrow  \coprod y^{et}\left(U_\alpha \cap U_\beta\right) \rightrightarrows  \coprod y^{et}\left(U_\alpha\right)\right)$$ so that
\begin{eqnarray*}
k\circ \Theta \circ y^{et}\left(Y\right) &\simeq& \hc \left(\coprod y\left(U_\alpha \cap U_\beta \cap U_\gamma\right) \rrrarrow  \coprod y\left(U_\alpha \cap U_\beta\right) \rightrightarrows  \coprod y\left(U_\alpha\right)\right)\\
&\simeq& y\left(Y\right).
\end{eqnarray*}
It follows that $$k\circ \Theta \circ y^{et} \simeq y \circ j_X.$$ Hence, by (\ref{eq:lanj}), $$k\circ \Theta = \Lan_{y^{et}} \left( y \circ j_X\right) = \left(j_X\right)_!.$$
\end{proof}

\begin{cor}
A stack $\X \in \St\left(\Mfd\right)$ is an \'etale differentiable stack if and only if it is in the essential image of the prolongation functor $j_!,$ where $$j:\Mfd^{et} \to \Mfd$$ is the canonical functor.
\end{cor}

Suppose we are given an \'etale $\xm$-stack $\X$. It determines a stack $\tilde y^{et}\left(\X\right)$ on $\xm^{et},$ which assigns an $X$-manifold $Y$ the groupoid of local homeomorphisms from $Y$ to $\X.$

\begin{thm}\label{thm:yet}
The functor $$\tilde y^{et}:\Et\left(\xm\right)^{et} \to \St\left(\xm^{et}\right)$$ is inverse to $\Theta.$
\end{thm}

\begin{proof}
Suppose that $\X$ is an \'etale $\xm$-stack. Denote by $$!_\X:\X \to \HH\left(X\right)$$ the essentially unique local homeomorphism to $\HH\left(X\right).$ Then, again since $\HH\left(X\right)$ is terminal in $\Et\left(\xm\right)^{et},$ one has that $$\Gamma\left(!_\X\right)\left(U\right) \simeq \tilde y^{et}\left(\X\right)\left(U\right)$$ for each open subset $U$ of $X,$ where $\Gamma$ is as in Theorem \ref{thm:etalsp}. By Theorem \ref{thm:etalsp}, it follows that $$L\left(i \pi\right)^*\left(\tilde y^{et}\left(\X\right)\right)\simeq !_\X$$ and hence $$\Theta\left(\tilde y^{et}\left(\X\right)\right) \simeq \X.$$
Conversely, suppose that $\Z$ is a stack on $\Et\left(\xm\right)^{et}.$ Since $\HH\left(X\right)$ is terminal in $\Et\left(\xm\right)^{et},$ there are equivalences $$!_{\Theta\left(\Z\right)} \simeq L\left(\left(i \pi\right)^*\left(\Z\right)\right)$$ and $$\tilde y^{et}\left(\Theta\left(\Z\right)\right)\left(U\right) \simeq \Gamma\left(!_{\Theta\left(\Z\right)}\right)\left(U\right),$$ for each open subset $U$ of $X.$ It follows that $$\left(i \pi\right)^*\tilde y^{et}\left(\Theta\left(\Z\right)\right) \simeq \left(i \pi\right)^* \Z.$$ Since $\left(i \pi\right)^*$ is an equivalence, we are done.
\end{proof}

\begin{rmk}
What this means is that for any stack $\Z \in \St\left(\xm^{et}\right),$ which we can view as a (generalized) \emph{moduli problem} which is functorial with respect to local homeomorphisms, there exists a unique \'etale $\xm$-stack $\Theta \Z,$ such that for a given $X$-manifold $Y$ the groupoid $\Z\left(Y\right)$ is equivalent to the groupoid of local homeomorphisms from $Y$ to $\Theta \Z.$ Conversely, given any \'etale $\xm$-stack $\X$, it determines a stack on $X$-manifolds and local homeomorphisms $\tilde y^{et}\left(\X\right)$ by assigning an $X$-manifold $Y$ the groupoid of local homeomorphisms from $Y$ to $\X,$ and these operations are inverse to each other.
\end{rmk}

It follows from Remark \ref{rmk:sites} that if $\Z$ is an \'etale $X$-manifold stack such that $\Z \simeq \left[\h\right]_{\xm},$ for $\h$ an \'etale groupoid object in $\xm$, and $U_!\left(\Z\right)$ is its underlying topological stack, then $$\sit\left(\left[\h\right]_{\xm},\h_0\right) \simeq \sit\left(U_!\left[\h\right],U\h_0\right),$$ where the equivalence sends an open subspace $V \subset \h_0$ to its underlying open subset $U\left(V\right) \subset U\left(\h_0\right).$ More precisely, the equivalence is induced by applying $U_!.$ This induces an equivalence between their associated $2$-topoi of stacks:

$$u^\Z:\St\left(\Z\right) \stackrel{\sim}{\longrightarrow} \St\left(U_!\Z\right).$$

\begin{lem}
The following diagram commutes up to canonical homotopy:

$$\xymatrix{\St\left(\Z\right) \ar[r]^-{u^\Z} \ar[d]_-{L} & \St\left(U_!\Z\right) \ar[d]^-{L}\\
\Et\left(\xm\right)^{et}/\Z \ar[r]^-{U_!/\Z} & \Et\left(\T\right)^{et}/U_!\Z,}$$ where $\Et\left(\T\right)^{et}$ denotes the bicategory of \'etale topological stacks and local homeomorphisms between them, and by abuse of notation, each $L$ denotes the respective \'etal\'e realization functor (see Theorem \ref{thm:etalsp}).
\end{lem}

\begin{proof}
Each composite is manifestly weak colimit preserving and they both agree on representables for the site $\sit\left(\left[\h\right]_{\xm},\h_0\right).$
\end{proof}

This has the following immediate corollary:

\begin{cor}\label{cor:xmfdastop}
The functor
\begin{eqnarray*}
\Et\left(\xm\right)^{et} &\to& \Et\left(\T\right)^{et}/U_!\HH\left(X\right)\\
\X &\mapsto& U_!\left(!_\X:\X \to \HH\left(X\right)\right)
\end{eqnarray*}
is an equivalence of bicategories.
\end{cor}

\subsection{Some computations and examples.}\label{sec:examples}

\subsubsection{Binary Products}
Since $\Et\left(\xm\right)^{et}$ is a $2$-topos, in particular, it has binary products. These products are very different from the product in $\Et\left(\xm\right)$ ($k$ \emph{does not} preserve limits), so let us denote it by $\times^{et}$ to distinguish it from $\times.$

Given an \'etale $X$-manifold stack $\X,$ denote by $$\tau_\X:\xm^{et}/\X \to \xm^{et}$$ the canonical projection. Note that by Remark \ref{rmk:sites}, we have $$St\left(\xm^{et}/\X\right) \simeq St\left(\X\right).$$

\begin{prop}\label{prop:prods1}
Suppose that $\X$ and $\F$ are \'etale $X$-manifold stacks. Then the projection map $$pr_1:\X \times^{et} \F \to \X$$ is the \'etal\'e realization of $\left(\tau_\X\right)^*\tilde y^{et}\left(\F\right).$
\end{prop}

\begin{proof}
The map $pr_1$ is a local homeomorphism, so its sections encode a stack $$\Gamma\left(pr_1\right) \in \St\left(\X\right).$$ By the universal property of products, one sees that the groupoid of sections of this map over a local homeomorphism $T \to \X$ is canonically equivalent to the groupoid of local homeomorphisms $T \to \F,$ which is by definition the groupoid $\tilde y^{et}\left(\F\right)\left(T\right).$
\end{proof}

\begin{cor}\label{cor:effprod1}
Suppose that $T$ is an $X$-manifold and $\F$ is an \'etale $X$-manifold stack such that $y^{et}\left(\F\right)$ is a sheaf, then $\X \times^{et} \F$ is an $X$-manifold.
\end{cor}

\begin{proof}
By Proposition \ref{prop:prods1}, $\X \times^{et} \F$ is the \'etal\'e space of a sheaf over $T.$
\end{proof}

\begin{rmk}
In light of Theorem \ref{thm:whateffmeans}, this implies that if $T$ is an $X$-manifold, and $\F$ an effective \'etale $X$-manifold stack, then $\X \times^{et} \F$ is an $X$-manifold.
\end{rmk}

This has the following strange consequence:

\begin{cor}\label{cor:weird1}
The category $\xm^{et}$ has binary products.
\end{cor}

\begin{ex}
For example, suppose that $S=n\mbox{-}\Mfd$, the category of smooth $n$-manifolds, so that $S=\xm,$ with $$X=\RR^n.$$ Then Corollary \ref{cor:weird1} implies \emph{the category of $n$-manifolds and local diffeomorphisms has binary products!}. This is really quite bizarre, so let us give two descriptions of this product. The first is furnished by Corollary \ref{cor:effprod1}, namely, if $M$ and $N$ are $n$-manifolds, then $M \times^{et} N$ is the \'etal\'e space of the sheaf on $N$ that assigns each open subset $U$ the set of local diffeomorphism from $U$ to $M.$ By symmetry, we could also describe $M \times^{et} N$ as the \'etal\'e space of a sheaf on $M$, in the same way.

There is also another description of $M \times^{et} N$. By Theorem \ref{thm:terminal}, the \'etale stack $\HH\left(\RR^n\right)$ is terminal, so $$\Et\left(n\mbox{-}\Mfd\right)^{et} \simeq \Et\left(n\mbox{-}\Mfd\right)^{et} / \HH\left(\RR^n\right).$$ It follows that the product of two $n$-dimensional \'etale stacks \emph{in the bicategory $\Et\left(n\mbox{-}\Mfd\right)^{et}$} can be computed as the pullback $\X \times_{\HH\left(\RR^n\right)} \Y,$ along the unique local diffeomorphisms to $\HH\left(\RR^n\right).$ Consider now the special case where $\X=M$ and $\Y=N$ are $n$-manifolds. Then $M \times^{et} N$ is  the pullback $$\xymatrix{P \ar[d] \ar[r] & M \ar[d] \\ N \ar[r] & \HH\left(\RR^n\right).}$$ By Corollary \ref{cor:morithaf}, it follows that both maps with codomain $\HH\left(\RR^n\right)$ are atlases, and that $P$ is the total space of the principal $\Ha\left(M\right)\mbox{-}\Ha\left(N\right)$ bibundle realizing their Morita equivalence.
\end{ex}

\subsubsection{Groupoid presentations for moduli stacks.}
In light of Theorem \ref{thm:whatetmeans1}, given a stack $\Z \in \St\left(\xm^{et}\right),$ it is natural to want to find a presentation of the \'etale stack $\Theta\left(\Z\right)$ as an \'etale groupoid.  In this subsection, we describe how to do so explicitly, and compute a few examples.

\begin{rmk}
In \cite{etalspme}, a concrete description is given, but to employ it, one would first have to use the equivalence $$\St\left(\xm^{et}\right) \simeq \St\left(\HH\left(X\right)\right),$$ and then find a presentation in terms of a groupoid object in sheaves. This is not so computationally appealing.
\end{rmk}

We will first turn our attention to the general case of an arbitrary stack $$\Z \in \St\left(\xm^{et}\right),$$ and show how to construct an \'etale groupoid presenting it, in a functorial way. We will then give a much nicer construction in the case that $\Z$ is in fact a sheaf.

Consider the Grothendieck construction of $\Z,$ $$\int\limits_{\xm^{et}} \!\!\!\!\!\!\!\! \Z.$$ Its objects consist of pairs $\left(T,x\right)$ with $T$ an $X$-manifold and $x \in \Z\left(T\right)_0.$ A morphism $$\left(T,x\right) \to \left(T',x'\right)$$ is a pair $\left(f,\alpha\right)$ where $f:T \to T'$ is a local homeomorphism, and $$\alpha:x \to \Z\left(f\right)\left(x'\right)$$ is an isomorphism in $\Z\left(T\right).$ Note that by Theorem \ref{thm:whatetmeans1}, $$\int\limits_{\xm^{et}} \!\!\!\!\!\!\!\Z \simeq \xm^{et}/\Theta\left(\Z\right),$$ i.e. the Grothendieck construction is equivalent to the category of local homeomorphisms into $\Theta\left(\Z\right)$ with source an $X$-manifold. It follows from Remark \ref{rmk:sites}, that $$\Sh\left(\mspace{5mu} \int\limits_{\xm^{et}} \!\!\!\!\!\!\!\Z\right)$$ is equivalent to the topos of sheaves on $\Theta\left(\Z\right),$ $\Sh\left(\Theta\left(\Z\right)\right),$ which is an \'etendue. Under the equivalence between \'etale topological stacks and \'etendues (Theorem \ref{thm:etendue}), this is $U_!\left(\Theta \left(\Z\right)\right),$ the underlying topological stack of $\Theta \left(\Z\right).$

Fix a basis $B$ for the space $X$ (e.g. the maximal one). Let $\Lambda^0_B\left(Z\right)$ denote $$\coprod\limits_{V \in B} \coprod\limits_{g \in \Z\left(V\right)_0} V,$$ which carries a canonical element $\mathscr{O} \in \Z\left(\Lambda^0_B\left(Z\right)\right)$ which, when restricted to the copy of $V$ indexed by $g \in \Z\left(V\right)_0,$ is $g$ itself. Denote the pair  $$\left(\Lambda^0_B\left(Z\right),\mathscr{O}\right) \in \int\limits_{\xm^{et}} \!\!\!\!\!\!\!\Z$$ by $\Lambda_B\left(Z\right).$ Consider its representable sheaf $$y\left(\Lambda_B\left(Z\right)\right)\in \Sh\left(\mspace{5mu} \int\limits_{\xm^{et}} \!\!\!\!\!\!\!\Z\right).$$ By Remark \ref{rmk:etendskt}, following the proof of Theorem \ref{thm:eqltoposhaf}, the canonical morphism $$y\left(\Lambda^0_B\left(Z\right)\right) \to 1$$ is an epimorphism, and $$\Sh\left(\mspace{5mu} \int\limits_{\xm^{et}} \!\!\!\!\!\!\!\Z\right)/y\left(\Lambda^0_B\left(Z\right)\right) \simeq \Sh\left(\Lambda^0_B\left(Z\right)\right).$$
From this data, we can construct an \'etale groupoid in $X$-manifolds, $\G\left(\Z\right)_B,$ whose object space is given by $\Lambda^0_B\left(Z\right).$ Consider the pullback topos
$$\xymatrix{P \ar[d] \ar[r] & \Sh\left(\Lambda^0_B\left(Z\right)\right) \ar[d]^-{p}\\
\Sh\left(\Lambda^0_B\left(Z\right)\right) \ar[r]^-{p} & \Sh\left(\mspace{5mu} \int\limits_{\xm^{et}} \!\!\!\!\!\!\!\Z\right),}$$
which is canonically equivalent to sheaves on the \'etal\'e space of the sheaf $p^*y\left(\Lambda_B\left(Z\right)\right)$ on $\Lambda^0_B\left(Z\right).$ This \'etal\'e space is the space of arrows of $\G\left(\Z\right)_B.$ The sheaf $p^*y\left(\Lambda_B\left(Z\right)\right)$ assigns an open subset $W$ of $\Lambda^0_B\left(Z\right)$ the set of pairs $\left(f,\alpha\right)$ with $$f:W \to \Lambda^0_B\left(Z\right)$$ a local homeomorphism, and $$\alpha:\mathscr{O}|_{W} \to \Z\left(f\right)\left(\mathscr{O}\right)$$ an isomorphism in $\Z\left(W\right).$ In summary, the \'etale groupoid $\G\left(\Z\right)_B$ can be described as the groupoid whose space of objects is the space $\Lambda^0_B\left(Z\right),$ and whose arrows are given by germs of pairs $\left(f,\alpha\right),$ as above (with the sheaf topology), with composition induced by composition in the fibered category $$\int\limits_{\xm^{et}} \!\!\!\!\!\!\!\Z$$ in the obvious way.

\begin{thm}\label{thm:presen1}
The \'etale $X$-manifold stack $\Theta\left(\Z\right)$ is equivalent to $\left[\G\left(\Z\right)_B\right]_{\xm}.$
\end{thm}

\begin{proof}
By construction, since $p$ is an epimorphism of topoi, it follows that if $U\left(\G\left(\Z\right)_B\right)$ is the  underlying topological groupoid of $\G\left(\Z\right)_B,$ then $$\Sh\left(\mspace{5mu} \int\limits_{\xm^{et}} \!\!\!\!\!\!\!\Z\right) \simeq \B U\left(\G\left(\Z\right)_B\right),$$ and hence $$\left[U\left(\G\left(\Z\right)_B\right)\right]\simeq U_!\left(\Theta\left(\Z\right)\right),$$ i.e. $U\left(\G\left(\Z\right)_B\right)$ is a presentation for the underlying topological stack of $\Theta\left(\Z\right).$ Since $$U_!\left(\Theta\left(\Z\right)\right) \simeq \hc \left( U\left(\G\left(\Z\right)_B\right)_2 \rrrarrow U\left(\G\left(\Z\right)_B\right)_1 \rrarrow U\left(\G\left(\Z\right)_B\right)_0\right),$$ the canonical local homeomorphisms $$U\left(\G\left(\Z\right)_B\right)_n \to U_!\HH\left(X\right)$$ exhibiting each $\left(\G\left(\Z\right)_B\right)_n$ as an $X$-manifold, assemble into a local homeomorphism $$U_!\left(\Theta\left(\Z\right)\right) \to U_!\HH\left(X\right),$$ such that the following diagram $2$-commutes
$$\xymatrix@C=1cm{U\left(\Lambda^0_B\left(Z\right)\right) \ar[rr] \ar[dr] & & U_!\left(\Theta\left(\Z\right)\right) \ar[dl]\\
& U_!\HH\left(X\right). &}$$ Notice that the horizontal map is an epimorphism, since it is an atlas. Under the equivalence of Corollary \ref{cor:xmfdastop}, this corresponds to an \'etale atlas $$\Lambda^0_B\left(Z\right) \to \Theta\left(\Z\right).$$ The pullback of this atlas against itself is equivalent to the \'etal\'e space of the sheaf $p^*y\left(\Lambda_B\left(Z\right)\right).$ It follows that $\Theta\left(\Z\right)\simeq \left[\G\left(\Z\right)_B\right]_{\xm}.$
\end{proof}

Although Theorem \ref{thm:presen1} gives a presentation for the \'etale stack $\Theta\left(\Z\right),$ for any stack $\Z \in \St\left(\xm^{et}\right),$ it is quite a large and cumbersome model. Often, by examining the local geometric properties of $\Z,$ one can find a more simple presentation. We present the following example:

\begin{ex}\label{ex:symplectic1}
Let $X=\mathbb{R}^{2n},$ so that $\xm$ is the category of smooth $2n$-dimensional manifolds. Consider the functor $$\mathcal{S}_n:\left(2n\mbox{-}\Mfd^{et}\right)^{op} \to \Set$$ sending a $2n$-dimensional manifold $M$ to the set of symplectic forms on $M$. Note that this is not even a functor on $2n\mbox{-}\Mfd$, but it is functorial with respect to local diffeomorphisms, and in fact is a sheaf. Let $\S_n$ denote the associated $2n$-dimensional \'etale differentiable stack $\Theta\left(\mathcal{S}_n\right).$ Similarly, consider the sheaf $$\mathcal{S}:\left(\Mfd^{et}\right)^{op} \to \Set,$$ also assigning symplectic forms. If $\S$ denotes $\Theta\left(\mathcal{S}\right),$ then it decomposes as a disjoint union $$\S=\coprod\limits_{n=0}^{\infty} \S_n.$$ We call $\S$ the \emph{classifying stack for symplectic forms}. By Theorem \ref{thm:whatetmeans1}, for any manifold $M,$ symplectic forms on $M$ are the same as local diffeomorphisms $$M \to \S$$ (and when $\dim\left(M\right)=2n,$ such a local diffeomorphism must factor through $\S_n$). Moreover, suppose that $\X$ is an \'etale differentiable stack, presented by an \'etale Lie groupoid $\G.$ Then since $\mathcal{S}$ is a sheaf,
\begin{eqnarray*}
\Hom_{\Et\left(\Mfd\right)^{et}}\left(\X,\S\right) &\simeq& \Hom\left(\tilde y^{et}\left(\X\right), \mathcal{S}\right)\\
 &\simeq& \varprojlim \left(\mathcal{S}\left(\G_0\right) \rrarrow \mathcal{S}\left(\G_1\right)\right)\\
 &=& \mathcal{S}\left(\G_0\right)^{\G},
\end{eqnarray*}
where $\mathcal{S}\left(\G_0\right)^{\G}$ is the set of $\G$-invariant symplectic forms on $\G_0.$ Hence, local diffeomorphisms $$\X \to \S$$ are the same as symplectic forms on the \'etale stack $\X,$ in the sense of \cite{intpoisson}, Definition 4.3.

We will now turn our attention to calculating an explicit groupoid presentation for each $\S_n$. (Since the Yoneda embedding into stacks preserves coproducts, this also tells us a groupoid presentation for $\S.$) The stacks $\S_n$ have much nicer groupoid presentations than the one afforded by Theorem \ref{thm:presen1}, thanks to \emph{Darboux's theorem}, namely, any symplectic $2n$-manifold is locally symplectomorphic to $$\RR^{2n}_{can.}=\left(\RR^{2n},\sum_{i=1}^{n} dp_i \wedge dq_i\right),$$ where $\RR^{2n}$ has coordinates $\left(p_1,\ldots,p_n,q_1,\ldots,q_n\right).$ This groupoid will naturally be a groupoid object in the category of $2n$-dimensional symplectic manifolds and local symplectomorphisms, which we denote by $\mathbf{Sym}\mbox{-}2n\mbox{-}\Mfd^{et}$. Its object space will be $\RR^{2n}_{can.}.$ To construct this groupoid presentation, notice that $$\int\limits_{2n\mbox{-}\Mfd^{et}} \!\!\!\!\!\!\!\! \mathcal{S}$$ is canonically equivalent to the category $\mathbf{Sym}\mbox{-}2n\mbox{-}\Mfd^{et}.$ As a consequence of Darboux's theorem, the canonical map $$y\left(\RR^{2n}_{can.}\right) \to 1$$ is an epimorphism. Following the construction of $\G\left(\Z\right)_B,$ one arrives at an \'etale groupoid presentation for $\S$, as the groupoid with object space $\RR^{2n}_{can.}$ and whose arrows are given by germs of locally defined symplectomorphisms (with the sheaf topology). We thus recover the groupoid discussed in Remark 2 b) of \cite{Haefliger}, commonly written as $\Gamma^{Sp}_{2n}.$

On one hand, by a result of Haefliger \cite{Haefliger}, the classifying space, $B\Gamma^{Sp}_{2n},$ classifies integrably homotopic classes of codimension-$2n$ (transversally) symplectic foliations on open manifolds. Such symplectic foliations are in natural bijection with closed $2$-forms $\omega$ of constant rank $2n$ (the foliation is determined by the kernel, see \cite{quant}, Section 3). On the other hand, by \cite{NoohiH}, the weak homotopy type of $B\Gamma^{Sp}_{2n}$ is the weak homotopy type of the underlying topological stack $U_!\S_n.$ Moreover, the cohomology of $B\Gamma^{Sp}_{2n},$ can be computed as the cohomology of the total complex ${C^{\bullet}}_{DR}\left(\Gamma^{Sp}_{2n}\right)$ coming from the De Rham double complex $\Omega^{\bullet}\left(\left(\Gamma^{Sp}_{2n}\right)_\bullet\right)$ as in \cite{cohst}. In particular, $$\omega_{can.} \in C^{2}_{DR}\left(\Gamma^{Sp}_{2n}\right)=\Omega^0_{DR}\left(\left(\Gamma^{Sp}_{2n}\right)_2\right) \oplus \Omega^1_{DR}\left(\left(\Gamma^{Sp}_{2n}\right)_1\right)\oplus \Omega^2_{DR}\left(\left(\Gamma^{Sp}_{2n}\right)_0\right).$$ The form $\omega_{can.}$ is closed in this complex since as a $2$-form on $\RR^{2n},$ it is closed, and $$\left(d_0\right)^*\omega_{can.}=\left(d_1\right)^*\omega_{can.},$$ since $\omega_{can.}$ is $\Gamma^{Sp}_{2n}$-invariant. So $\left[\omega_{can.}\right]$ is a well defined cohomology class in $$\operatorname{H}^2\left(\S_n,\RR\right)\cong \operatorname{H}^2\left(B\Gamma^{Sp}_{2n},\RR\right).$$ If $\left(M,\alpha\right)$ is any symplectic $2n$-manifold, classified by a local diffeomorphism $$\varphi_\alpha:M \to \S_n,$$ then $$\varphi_\alpha^*\left[\omega_{can.}\right]=\left[\alpha\right].$$ Let $M=\left(S^2\right)^n$ be equipped with a symplectic form $\alpha$ coming from each factor of $S^2$ equipped with a volume form. Then $\alpha^n$ is a volume form on $\left(S^2\right)^n$ and $$\left[\alpha\right]^n=\varphi_\alpha^*\left[\omega_{can.}\right]^n$$ generates the top cohomology group. It follows that $\left[\omega_{can.}\right]^k$ is a non-zero cohomology class in $\operatorname{H}^{2k}\left(B\Gamma^{Sp}_{2n},\RR\right)$ for $k=1,\ldots,n$.
\end{ex}

We now turn our attention to finding a groupoid presentation for $\Theta\left(\F\right),$ when $\F$ is a \emph{sheaf}, i.e. $$\F \in \Sh\left(\xm^{et}\right).$$ In particular, this will give a different presentation for $\S_n$ than the one described above.

First, let us introduce some terminology. An object of $$\int\limits_{\xm^{et}} \!\!\!\!\!\!\! \F$$ consists of a pair $\left(M,x\right),$ with $M$ an $X$-manifold and $x \in \F\left(M\right).$ An arrow from $\left(M,x\right)$ to $\left(N,y\right)$ consists of a map $$f:M \to N$$ such that $f^*y=x.$ We call such a morphism a \textbf{local $\F$-homeomorphism}

Notice that since epimorphisms are stable under pullback, and $$y^{et}\left(X\right) \to 1$$ is an epimorphism in $\Sh\left(\xm^{et}\right),$ that it follows that the canonical projection map $$y^{et}\left(X\right) \times \F \to \F$$ is an epimorphism. By Proposition \ref{prop:prods1}, $$\Theta\left(y^{et}\left(X\right) \times \F\right) \simeq X \times^{et} \Theta\left(\F\right)$$ is the \'etal\'e space of the sheaf $\left(\tau_X\right)^*\F$ on $X,$ and hence an $X$-manifold. It follows that the canonical projection map $$pr_2:X \times^{et} \Theta\left(\F\right) \to \Theta\left(\F\right)$$ is an \'etale atlas for $\Theta\left(\F\right).$ The corresponding \'etale groupoid is given by the pullback of stacks
$$\xymatrix{\left(X \times^{et} \Theta\left(\F\right)\right) \times_\F \left(X \times^{et} \Theta\left(\F\right)\right) \ar[d]_-{s} \ar[r]^-{t} & X \times^{et} \Theta\left(\F\right) \ar[d]^-{pr_2} \\
X \times^{et} \Theta\left(\F\right) \ar[r]^-{pr_2} & \F.}$$
The map $$s:\left(X \times^{et} \Theta\left(\F\right)\right) \times_\F \left(X \times^{et} \Theta\left(\F\right)\right) \to X \times^{et} \Theta\left(\F\right)$$ is a local homeomorphism, so it encodes a sheaf over $X \times^{et} \Theta\left(\F\right).$ Given an open subset $V \subset X \times^{et} \Theta\left(\F\right),$ sections of $s$ over $V$ are in natural bijection with lifts
\begin{equation}\label{eq:lifts}
\resizebox{6.5cm}{!}{\xymatrix{& & X \times^{et} \Theta\left(\F\right) \ar[d]^-{pr_2} \\
V \ar@{^{(}->}[r] \ar@{-->}[urr]<+1.3ex> & X \times^{et} \Theta\left(\F\right) \ar[r]^-{pr_2} & \F.}}
\end{equation}

By Theorem \ref{thm:whatetmeans1}, the local homeomorphism $pr_2$ corresponds to a unique element $$\mathscr{O}_\F \in \F\left(X \times^{et} \Theta\left(\F\right)\right) \cong \Hom\left(y^{et}\left(X \times^{et} \Theta\left(\F\right)\right),\F\right).$$ It follows that lifts of the form (\ref{eq:lifts}) are in natural bijection with local homeomorphisms $$f:V \to X \times^{et} \Theta\left(\F\right)$$ such that $f^*\mathscr{O}_\F = \mathscr{O}_\F|_{V},$ i.e. local $\F$-homeomorphisms

$$\left(V,\mathscr{O_\F}|_V\right) \to \left(X \times^{et} \Theta\left(\F\right),\mathscr{O_\F}\right).$$ It follows that $\left(X \times^{et} \Theta\left(\F\right)\right) \times_\F \left(X \times^{et} \Theta\left(\F\right)\right)$ is the \'etal\'e space of the sheaf of germs of local $\F$-homeomorphisms of $\left(X \times^{et} \Theta\left(\F\right),\mathscr{O_\F}\right).$ The groupoid structure is given by composition of such germs. Let us denote this groupoid by $P_X\left(\F\right).$

Let us provide an illustrative example:

\begin{ex}\label{ex:Riemannian1}
Let $\mathcal{R}:\left(\Mfd^{et}\right)^{op} \to \Set$ be the functor which assigns a manifold $M$ its set of Riemannian metrics. Just as with the functor $\mathcal{S},$ this is not a functor on $\Mfd$, but it is functorial with respect to local diffeomorphisms, and in fact is a sheaf. By Theorem \ref{thm:whatetmeans1}, there exists an \'etale differentiable stack $\R=\Theta\left(\mathcal{R}\right),$ such that local diffeomorphisms $$M \to \R$$ are the same as Riemannian metrics on $M.$ We call $\R$ the \emph{classifying stack for Riemannian metrics}. In fact, it follows easily that for an \'etale differentiable stack $\X,$ local diffeomorphisms $\X \to \R,$ are the same as Riemannian metrics on $\X,$ in the sense of \cite{morsifold} (and may also be described by invariant Riemannian metrics on a groupoid presentation, similar to the case of symplectic forms as in Example \ref{ex:symplectic1}). We also have that $\R$ decomposes as a direct sum $$\coprod\limits_{n=0}^{\infty} \R_n,$$ where $$\R_n=\Theta\left(\mathcal{R}|_{n\mbox{-}\Mfd^{et}}\right).$$

Since each $$\mathcal{R}_n:=\mathcal{R}|_{n\mbox{-}\Mfd^{et}}$$ is a sheaf, it follows that $P_{\RR^n}\left(\mathcal{R}_n\right)$ is a groupoid presentation for $\R_n.$ The object space is the \'etal\'e space of the sheaf on $\RR^n,$ which assigns each open subset its set of Riemannian metrics. This (non-Hausdorff) $n$-manifold carries a canonical Riemannian metric $g_n:=\mathscr{O}_{\mathcal{R}_n}.$ Explicitly, if $z \in P_{\RR^n}\left(\mathcal{R}_n\right)_0,$ then from the description of this space as an \'etal\'e space, $z$ can be written as $$z=\mathit{germ}_{x} \tau$$ for $\tau$ some Riemannian metric defined on a neighborhood of $x$ in $\RR^n.$ The inner product at $z$ is $$\resizebox{5in}{!}{$g\left(x\right)_n=\left(d\varphi\right)^{-1} \circ \tau\left(x\right) \circ \left(d\varphi \otimes d\varphi\right):\mathbf{T}_z\left(P_{\RR^n}\left(\mathcal{R}_n\right)_0\right) \otimes \mathbf{T}_z\left(P_{\RR^n}\left(\mathcal{R}_n\right)_0\right) \to \mathbf{T}_z\left(P_{\RR^n}\left(\mathcal{R}_n\right)_0\right),$}$$ where $$\varphi:P_{\RR^n}\left(\mathcal{R}_n\right)_0 \to  \RR^n$$ is the local diffeomorphism encoding the sheaf of metrics. The Riemannian manifold $\left(P_{\RR^n}\left(\mathcal{R}_n\right)_0,g_n\right)$ plays the same role as $\RR^{2n}_{can.},$ as any Riemannian $n$-manifold is locally isometric to $\left(P_{\RR^n}\left(\mathcal{R}_n\right)_0,g_n\right).$ Note that $$\int\limits_{n\mbox{-}\Mfd^{et}} \!\!\!\!\!\!\! \mathcal{R}_n$$ is canonically equivalent to the category of smooth Riemannian $n$-manifolds and local isometries, $n\mbox{-}\Mfd^{li},$ in such a way that local $\mathcal{R}_n$-homeomorphisms are local isometries. It follows that $P_{\RR^n}\left(\mathcal{R}_n\right)$ is the groupoid of germs of local isometries of the Riemannian manifold $$\left(P_{\RR^n}\left(\mathcal{R}_n\right)_0,g_n\right).$$ We thus recover the groupoid discussed in Example c) of \cite{Haefliger}, commonly written as $R\Gamma^n$; its classifying space classifies integrably homotopic classes of codimension-$n$ Riemannian foliations on open manifolds (\cite{Haefliger}).
\end{ex}

\begin{ex}
Let $S=\mbox{$n$-}C^{0}\Mfd$ be the category of topological $n$-manifolds. Consider the functor $$C^{\infty}_{st}:\left(\mbox{$n$-}C^{0}\Mfd^{et}\right)^{op} \to \Set$$ which assigns a topological $n$-manifold the (possibly empty) set of smooth structures it can support. Concretely, it assigns to each topological manifold $M$ the set of maximal smooth atlases that can exist on $M.$ It is not functorial with respect to all continuous maps, but it is with respect to local homeomorphisms, since atlases can be pulled back via local homeomorphisms, and with this functoriality, $C^{\infty}_{st}$ is a sheaf. The Grothendieck construction of this sheaf is equivalent to the category $n\mbox{-}\Mfd^{et}$ of \emph{smooth} $n$-manifolds and local diffeomorphisms between them. By Theorem \ref{thm:whatetmeans1}, there exists an \'etale $S$-stack $\Theta\left(C^{\infty}_{st}\right)$ such that for a given topological $n$-manifold $M,$ smooth structures on $M$ are in bijection with local homeomorphisms $$M \to \Theta\left(C^{\infty}_{st}\right).$$ In fact, the proof of Theorem \ref{thm:eqltoposhaf} with the category of spaces being $n\mbox{-}\Mfd^{et}$ and $X=\RR^n,$ shows that $\Theta\left(C^{\infty}_{st}\right) = \HH_{\C^\infty}\left(\RR^{n}\right),$ that is, $\Theta\left(C^{\infty}_{st}\right)$ is the \'etale stack coming from the underlying topological groupoid of the smooth $n$-dimensional Haefliger groupoid. The terminal object in $\Et\left(\mbox{$n$-}C^{0}\Mfd\right)^{et}$ is $\HH_{\C^0}\left(\RR^{n}\right),$ and hence there is a unique local homeomorphism $$\HH_{\C^\infty}\left(\RR^{n}\right) \to \HH_{\C^0}\left(\RR^{n}\right).$$ For a given topological $n$-manifold $M$, smooth structures on $M$ corresponds to lifts
$$\xymatrix{& \HH_{\C^\infty}\left(\RR^{n}\right) \ar[d]\\
M \ar@{-->}[ru] \ar[r]_-{!_M} & \HH_{\C^0}\left(\RR^{n}\right).}$$
Of course, this becomes most interesting for $n>3,$ where the set of such lifts may be empty on one extreme, or a continuum on the other. There are many variants of this idea possible. For example, one can look at the sheaf on $2n\mbox{-}\Mfd^{et}$ assigning a manifold its set of complex structures, and its associated \'etale stack will be presented by the $n$-dimensional complex Haefliger groupoid.
\end{ex}

\subsection{Large categories of spaces}\label{sec:large}

In general, we have tacitly avoided assuming our category of spaces $S$ is small. This does not cause much of a complication in the case that $S$ is smooth manifolds for example, since every manifold can be covered by open subsets of (possibly different) $\RR^n$'s. Hence, the Cartesian manifolds of the form $\RR^n$, form a set of \emph{topological generators} of $\Sh\left(\Mfd\right)$ in the sense of \cite{sga4}. This is what allows us to choose the space $$X=\coprod\limits^\infty_{n=0} \RR^n$$ and have the equivalence $$\xm \cong \Mfd.$$

\begin{dfn}
A category of spaces $S$ is \textbf{of small topological generation} if there exists a \emph{small} subcategory $G$ of $S$ such that every space $Y$ in $S$ can be covered by open subsets of spaces in $G.$ The objects of $G$ are said be \textbf{generators} of $S$.
\end{dfn}

\begin{rmk}
We could also ask for the seemingly stronger condition that each space admits a cover by spaces in $G$, (or even demand that $G$ consists of a single space, by Remark \ref{rmk:onesp}) and arrive at an equivalent concept, but for a different subcategory of generators, $G'.$ ($G'$ consists of all open subspaces of elements of $G$.) For such a $G'$ there is an equivalence $\St\left(G'\right) \simeq \St\left(S\right),$ where the former is stacks with respect to the induced Grothendieck topology on $G'.$ This follows immediately from the Comparison Lemma \cite{sga4} III.
\end{rmk}

\begin{rmk}\label{rmk:onesp}
If $S$ is generated by a set $G,$ then $$X:=\coprod\limits_{Z_\alpha \in G} Z_\alpha$$ has the property that $\xm = S.$
\end{rmk}

In particular, Theorem \ref{thm:prol} applies, namely:

\begin{cor}\label{cor:prol2}
If $S$  is of small topological generation, then $\X \in \St\left(S\right)$ is an \'etale stack if and only if it is in the essential image of the prolongation functor $j_!,$ where $$j:S^{et} \to S$$ is the canonical functor.
\end{cor}

However, this type of trick will not work when $S$ is topological spaces, or $S$ is schemes, so we must take a bit of care. We will assume that we have a Grothendieck universe $\mathcal{U}$ contained in a larger Grothendieck universe $\mathcal{V},$ such that $S$ is locally $\mathcal{U}$-small and has a $\mathcal{V}$-small set of objects. We will denote by $\Gpd$ the bicategory of essentially $\mathcal{U}$-small groupoids, and $\widehat{\Gpd}$ the bicategory of essentially $\mathcal{V}$-small groupoids (and similarly for $\Set$). Consider the subcategory of the functor bicategory $\Gpd^{S^{op}}$ consisting of those functors which satisfy descent, and denote it by $\St\left(S\right)$. Note that in reality $St\left(S\right)$ may not be a $2$-topos in $\mathcal{U}$ if $S$ does not have a $\mathcal{U}$-small set of generators (e.g. when $S$ is topological spaces or schemes), since it may fail to be locally presentable. However, this is mostly a technical annoyance, and we can still refer to elements of $\St\left(S\right)$ as stacks, since $\St\left(S\right)$ is equivalent to the subcategory of $\widehat{\St}\left(S\right)$ of stacks of $\mathcal{V}$-small groupoids on $S$ (which is a $2$-topos in $\mathcal{V}$), consisting of those which take values in $\mathcal{U}$-small groupoids. In particular, \'etale $S$-stacks lie in $\St\left(S\right).$ Moreover, $\St\left(S\right)$ still has $\mathcal{U}$-small colimits and by \cite{htt} (Remark 6.3.5.17) the inclusion $$\St\left(S\right) \to \widehat{\St}\left(S\right)$$ preserves these. What we cannot fix is that there will not be a space $X$ such that $S = \xm.$ However, a certain modification of Corollary \ref{cor:prol2} holds even for such large categories of spaces:

\begin{dfn}
For every space $X$ in $S,$ denote by $\tau_X:\xm^{et} \to S$ the composite $$\xm^{et} \stackrel{j_X}{\longrightarrow} \xm \stackrel{i_X}{\longrightarrow} S.$$ The restriction functor $$\tau_X^*:\St\left(S\right) \to \St\left(\xm^{et}\right)$$ has a left adjoint $$\left(\tau_X\right)_!:\St\left(\xm^{et}\right) \to \St\left(S\right)$$ called the \textbf{relative prolongation functor.}
\end{dfn}

\begin{thm}\label{thm:prol3}
A stack $\X \in \St\left(S\right)$ is an \'etale stack if and only if it is in the essential image of the relative prolongation functor $\left(\tau_X\right)_!,$ for some space $X$ in $S.$
\end{thm}

\begin{proof}
Notice that $\left(\tau_X\right)_!=\left(i_X\right)_! \circ \left(j_X\right)_!$, and $\left(i_X\right)_!$ is fully faithful and restrict to an embedding of \'etale $\xm$-stacks into \'etale stacks by Remark \ref{rem:xmff}. So, if $\X \simeq \left(\tau_X\right)_!\left(\Z\right)$ for some $\Z \in \St\left(\xm^{et}\right),$ then by Theorem \ref{thm:prol}, $\left(j_X\right)_!\left(\Z\right)$ is an \'etale $\xm$-stack and hence by Remark \ref{rem:xmff}, $\X$ is an \'etale stack. Conversely, suppose that $\X \simeq \left[\G\right]_S$ is an \'etale stack. Notice that $\G$ is canonically an \'etale groupoid object in $\G_0\mbox{-}\Mfd,$ and hence $\X \simeq \left(i_{\G_0}\right)_!\left(\left[\G\right]_{\G_0\mbox{-}\Mfd}\right).$ By Theorem \ref{thm:prol} and Theorem \ref{thm:yet}, $\left[\G\right]_{\G_0\mbox{-}\Mfd} \simeq \left(j_{\G_0}\right)_!\left(\tilde y^{et}\left(\left[\G\right]_{\G_0\mbox{-}\Mfd}\right)\right),$ and hence $$\X \simeq \left(\tau_{\G_0}\right)_!\left(\tilde y^{et}\left(\left[\G\right]_{\G_0\mbox{-}\Mfd}\right)\right).$$
\end{proof}

\begin{rmk}
A version Theorem \ref{thm:prol3} remains valid for the \'etale site of affine schemes and classifies Deligne-Mumford stacks (with no separation conditions, as in \cite{dagl}). We will outline the idea, but we will not give the full proof here. The proof will be given in \cite{higherme}, in a much more general context. However, there also exists a more down-to-earth proof which we will sketch here. Let $Q$ be an algebraic space (with no separation conditions), and let $Q\mbox{-}\mathbf{AlgSp}$ denote the full subcategory of algebraic spaces, consisting of those which are locally \'etale isomorphism to $Q.$ Denote by $Q\mbox{-}\mathbf{AlgSp}^{et}$ the subcategory whose morphisms are all \'etale, and denote by $$j_Q:Q\mbox{-}\mathbf{AlgSp}^{et} \to \mathbf{AlgSp}$$ the inclusion into algebraic spaces. Denote by $A\left(Q\right)$ a fixed set of affine schemes which can form an \'etale cover of $Q,$ and denote by $A^{et}$ the subcategory of $Q\mbox{-}\mathbf{AlgSp}^{et}$ on those affines. Notice that we have a commutative diagram
$$\xymatrix{A^{et} \ar[r]^{j_A} \ar[d] &  \mathbf{Aff} \ar[d]\\
Q\mbox{-}\mathbf{AlgSp}^{et} \ar[r]_-{j_Q} & \mathbf{AlgSp},}$$ and by the comparison lemma \cite{sga4}, a diagram
$$\xymatrix{\St\left(A^{et},et\right) \ar[r]^{\left(j_A\right)_!} \ar[d] &  \St\left(\mathbf{Aff},et\right) \ar[d]\\
\St\left(Q\mbox{-}\mathbf{AlgSp}^{et},et\right) \ar[r]_-{\left(j_Q\right)_!} & \St\left(\mathbf{AlgSp},et\right),}$$
such that the vertical maps are both equivalences. Hence, the essential image of $\left(j_A\right)_!$ in $\St\left(\mathbf{Aff},et\right)$ may be identified with the essential image of $\left(j_Q\right)_!$ in $\St\left(\mathbf{AlgSp},et\right).$  The \emph{Haefliger groupoid} $\Ha\left(Q\right)$ of $Q$ is a groupoid object  in algebraic spaces, all of whose structure maps are \'etale, and has base $Q$. The arrows may be constructed from the sheaf $E$ on the small \'etale site of $Q$ which assigns an \'etale morphisms $U \to Q$ the set of \emph{all} \'etale morphisms $U \to  Q.$ This sheaf $E$ is the sheaf of sections of a unique algebraic space $\Ha\left(Q\right)_1 \to Q$ \'etale over $Q.$ By a similar proof as with the ordinary Haefliger groupoid, one can show that the essential image of $\left(j_Q\right)_!$ is those Deligne-Mumford stacks which arise from a groupoid object $\G$ in algebraic spaces, all of whose structure maps are \'etale, such  that $\G_0$ and $\G_1$ are \'etale locally isomorphic to $Q.$ Putting this all together yields that $\X \in \St\left(\mathbf{Aff}\right)$ is a Deligne-Mumford stack (again with no separation conditions) if and only if it is in the essential image of $\left(j_A\right)_!$ for some set of affines $A.$ Put another way, Deligne-Mumford stacks are precisely solutions to those moduli problems arising as the prolongation of some moduli problem on a set of affine schemes which is functorial with respect to \'etale morphisms. This very nice categorical description is good reason to allow in the definition of Deligne-Mumford stack, stacks without any separation conditions on their diagonals.
\end{rmk}

\section{Gerbes and Effectivity}\label{sec:gerbeff}
\subsection{Internal homotopy theory and gerbes}
Recall that for any Grothendieck site $\left(\C,J\right),$ the inclusion $$i:\Sh\left(\C,J\right) \hookrightarrow \St\left(\C,J\right)$$ from sheaves into stacks admits a left adjoint $\pi_0.$ Concretely, $\pi_0\left(\X\right)$ is the sheafification of the presheaf on $\C$ assigning an object $C$ the set of \emph{isomorphism classes} of the groupoid $\X\left(C\right).$ This presheaf is in general not a sheaf (hence the need for sheafification).

\begin{dfn}
A morphism $f:\Y \to \X$ in $\St\left(\C,J\right)$ is said to be \textbf{truncated} if, when regarded as an object of $\St\left(\C/\X\right),$ $f$ is a sheaf.
\end{dfn}

\begin{rmk}\label{rmk:sheafo}
A stack $\X$ in $\St\left(\C,J\right)$ is a sheaf, if and only if its canonical map $$\X \to 1$$ to the terminal object is truncated.
\end{rmk}

\begin{rmk}\label{rmk:trunrep}
It follows from \cite{etalspme}, Theorem \ref{thm:shhhhvs}, that a local homeomorphism $$f:\X \to \Y$$ in $\St\left(S\right),$ is truncated if and only if it is representable. This is also equivalent to the induced geometric morphism $$\Sh\left(f\right):\Sh\left(\X\right) \to \Sh\left(\Y\right)$$ being an \'etale geometric morphism.
\end{rmk}

\begin{rmk}
A truncated morphism in $\St\left(\C,J\right)$ is the same as a \emph{$0$-truncated} morphism in the $\infty$-topos $\Sh_\infty\left(C,J\right),$ in the sense of \cite{htt}.
\end{rmk}

\begin{dfn}
A morphism $f:\Y \to \X$ is said to be \textbf{connected}, if it is an epimorphism, and the diagonal map $$\Y \to \Y \times_{\X} \Y$$ is an epimorphism.
\end{dfn}

\begin{rmk}
A connected morphism in $\St\left(\C/\Y\right)$ is the same as a \emph{$1$-connective} morphism in the $\infty$-topos $\Sh_\infty\left(C,J\right),$ in the sense of \cite{htt}.
\end{rmk}

\begin{rmk}\label{rmk:locfull}
The condition on $f:\Y \to \X$ to have the diagonal $$\Y \to \Y\times_\X \Y$$ an epimorphism can be phrased concretely as saying that $f$ is \emph{locally full} in the following sense:

Recall that a morphism $f:\Y \to \X$ is a weak-natural transformation, that is, it is an assignment to each object $C$ of $\C$ a functor $$f\left(C\right):\Y\left(C\right) \to \X\left(C\right)$$ and to each morphism $g:C \to D$ a natural transformation
$$
\xymatrix@C=1.5cm{
\Y(D)\ar[r]^{f(D)}\ar[d]_{\Y(g)}&\X(D)\ar[d]^{\X(g)}\\
\Y(C)\ar[r]_{f(C)}\ar@2{->}[ur]|{f\left(g\right)}&\X(C),}
$$
such that on each pair of composable arrows, the natural pentagon commutes. The map $f$ is locally full if:

For all $C \in \C,$ $y_0,y_1 \in \Y\left(C\right)_0,$ and $$h:f\left(C\right)\left(y_0\right) \to f\left(C\right)\left(y_1\right)$$ in $\X\left(C\right),$ there exists a cover $\left(\lambda_\alpha:C_\alpha \to C\right)$ such that for each $\alpha,$ there exists $$g_\alpha:\Y\left(\lambda_\alpha\right)\left(y_0\right) \to \Y\left(\lambda_\alpha\right)\left(y_1\right)$$ such that the following diagram commutes:

$$\xymatrix@C=3.5cm@R=1.5cm{\left(f\left(C_\alpha\right) \circ \Y\left(\lambda_\alpha\right)\right)\left(y_0\right) \ar[r]^-{f\left(C_\alpha\right)\left(g_\alpha\right)} \ar[d]_-{f\left(\lambda_\alpha\right)\left(y_0\right)} & \left(f\left(C_\alpha\right) \circ \Y\left(\lambda_\alpha\right)\right)\left(y_1\right) \ar[d]^-{f\left(\lambda_\alpha\right)\left(y_1\right)}\\
\left(\X\left(\lambda_\alpha\right) \circ f\left(C\right)\right)\left(y_0\right) \ar[r]_-{\X\left(\lambda_\alpha\right)\left(h\right)} & \left(\X\left(\lambda_\alpha\right) \circ f\left(C\right)\right)\left(y_0\right).}$$
\end{rmk}

Recall the following definition of a gerbe \cite{Giraud}:

\begin{dfn}\label{dfn:gerbe}
A \textbf{gerbe} over a Grothendieck site $\left(\C,J\right)$ is a stack $\g$ over $\C$ such that
\begin{itemize}
\item[i)] the unique map $\g \to 1$ to the terminal stack is an epimorphism, and
\item[ii)] the diagonal map $\g \to \g \times \g$ is an epimorphism.
\end{itemize}
\end{dfn}

Some readers may be more used to another definition:

The first condition means that for any object $C \in \C_0$, the unique map $C \to *$ \emph{locally} factors through $\g \to *$, up to isomorphism. Spelling this out means that there exists a cover $\left(f_\alpha:C_\alpha \to C\right)$ of $C$ such that each groupoid $\g\left(C_\alpha\right)$ is non-empty. This condition is often phrased by saying $\g$ is \emph{locally non-empty}.

The second condition means that for all $C$, any map $C \to \g \times \g$ \emph{locally} factors through the diagonal $\g \to \g \times \g$ up to isomorphism. Spelling this out, any map $C \to \g \times \g$, by Yoneda, corresponds to objects $x$ and $y$ of the groupoid $\g\left(C\right)$. The fact that this map locally factors through the diagonal means that, given two such objects $x$ and $y$, there exists a cover $\left(g_\beta:C_\beta \to C\right)$ of $C$ such that for all $\beta$, $$\g\left(f_\beta\right)\left(x\right) \cong \g\left(f_\beta\right)\left(y\right)$$ in $\g\left(C_\beta\right)$. This condition is often phrased by saying $\g$ is \emph{locally connected}.

If it were not for the locality of these properties, then this would mean that each $\g\left(C\right)$ would be a non-empty and connected groupoid, hence, equivalent to a group.

\begin{dfn}
The full subcategory of $\St\left(\C,J\right)$ on all gerbes, is called the $2$-category of gerbes and is denoted by $\scr{Gerbe}\left(\C\right)$.
\end{dfn}

\begin{rmk}
On one hand, a stack $\X$ is a gerbe if and only if the canonical map $$\X \to 1$$ is connected. On the other hand, a morphism $f:\Y \to \X$ is connected if and only if, when viewed as an object of $\St\left(\C/\Y\right)$, $f$ is a gerbe.
\end{rmk}

\begin{prop}\label{prop:gerbterm}
A stack $\X$ in $\St\left(\C,J\right)$ is a gerbe if and only if $\pi_0\left(\X\right)$ is terminal.
\end{prop}

\begin{proof}
This follows immediately from \cite{htt} Proposition 6.5.1.12.
\end{proof}

\begin{prop}\label{prop:facto}
Let $\mathfrak{C}$ be the subcategory of $\St\left(\C,J\right)$ spanned by connected morphisms, and $\mathfrak{T}$ be the subcategory spanned by truncated ones. Then the pair $\left(\mathfrak{C},\mathfrak{T}\right)$ forms an orthogonal factorization system on $\St\left(\C,J\right).$
\end{prop}

\begin{proof}
Let $f:\Y \to \X$ be a morphism in $\St\left(\C,J\right).$ We may regard $f$ as a morphism in the $\infty$-topos $\Sh_\infty\left(C,J\right).$ By \cite{htt}, Example 5.2.8.16, the class of $0$-truncated morphisms and $1$-connective morphisms form an orthogonal factorization system. Therefore, up to a contractible space of choices, $f$ factors uniquely as
\begin{equation}\label{eq:factor}
\Y \stackrel{g}{\longrightarrow} \Y' \stackrel{h}{\longrightarrow} \X,
\end{equation}
such that $g$ is $1$-connective and $h$ is $0$-truncated. (This is a weak-factorization, in that $h\circ g \simeq f.$) Since $h$ is $0$-truncated, it is also $1$-truncated. Since $\X$ is also $1$-truncated (as it is a stack of groupoids), by \cite{htt} Lemma 5.5.6.14, so is $\Y'.$ This implies that the factorization (\ref{eq:factor}), actually takes place in $\St\left(\C,J\right).$
\end{proof}
\begin{rmk}
\emph{Any} $2$-topos carries a factorization system of this form.
\end{rmk}

\begin{rmk}\label{rmk:facto}
A concrete way to produce the factorization (\ref{eq:factor}) is as follows. Regard $f$ as an object of $\St\left(\C/\X\right).$ Consider the unit of the adjunction $\pi^\X_0 \dashv i$ for $\St\left(\C/\X\right),$ at $f$:
$$\eta_f:f \to \pi^\X_0\left(f\right).$$ By abuse of notation, this is a $2$-commutative diagram
$$\xymatrix@C=0.8cm{\Y \ar[dr]_-{f} \ar[rr]^-{\eta_f} & & \pi^\X_0\left(f\right) \ar[dl]^-{f'}\\
& \X &.}$$
The factorization (\ref{eq:factor}) is given by $$\Y \stackrel{\eta_f}{\longrightarrow} \pi^\X_0\left(f\right) \stackrel{f'}{\longrightarrow} \X.$$  In particular, if $$\X \to 1$$ is the canonical map, then factorization is given by $$\X \stackrel{\eta_\X}{\longrightarrow} \pi_0\left(\X\right) \to 1.$$
\end{rmk}

\begin{prop}\label{prop:quu}
Suppose that $X$ is a sheaf, and $f:\Y \to X$ is any map from a stack. Then $f$ is connected if and only if the induced morphism $$\pi_0\left(\Y\right) \stackrel{\pi_0\left(f\right)}{\longlongrightarrow} \pi_0\left(X\right) \stackrel{\eta_X}{\longlongrightarrow} X$$ is an equivalence.
\end{prop}

\begin{proof}
Suppose that $f$ is connected. Then, since $X$ is a sheaf, by Remark \ref{rmk:sheafo}, $$\Z \stackrel{f}{\longrightarrow} \X \to 1$$ is a factorization of the unique map $\Z \to 1$ by a connected morphism followed by a truncated one. By Remark \ref{rmk:facto} and the uniqueness of this factorization (Proposition \ref{prop:facto}), it follows that the induced morphism $$\pi_0\left(\Y\right) \to X$$ is an equivalence. Conversely, suppose that $X$ is a sheaf, and $f:\Y \to X$ is any map from a stack. Then $f$ is connected if and only if the induced morphism $$\pi_0\left(\Y\right) \to X$$ is an equivalence. By naturality, the following diagram $2$-commutes:
$$\xymatrix@C=2.5cm{\Y \ar[d]_-{\eta_\Y} \ar[r]^-{f} & X \ar[d]^-{\eta_\X}\\
\pi_0\left(\Y\right) \ar[r]_-{\pi_0\left(f\right)} & \pi_0\left(X\right).}$$ Notice that $\eta_\Y$ is connected by Remark \ref{rmk:facto}, and $\pi_0\left(f\right)$ is an equivalence by assumption, and $\pi_0\left(\eta_X\right)$ is an equivalence since $X$ is sheaf. It follows that $f$ is connected.
\end{proof}

\begin{rmk}\label{rmk:sheafconn}
It follows that $f:\Y \to X$, when $X$ a sheaf, is connected if and only if $\pi_0\left(f\right)$ is an isomorphism of sheaves. \end{rmk}

\begin{prop}
If $f:\Y \to \X$ is connected, then $\pi_0\left(f\right)$ is an isomorphism of sheaves.
\end{prop}
\begin{proof}
If $f$ is connected, then $$\Y \stackrel{\eta_\X \circ f}{\longlongrightarrow} \pi_0\left(\X\right) \to 1$$ is a factorization of the unique map $\Y \to 1$ by a connected map, followed by a truncated one. However, so is $$\Y \stackrel{\eta_\Y}{\longlongrightarrow} \pi_0\left(\Y\right) \to 1.$$ By uniqueness of factorizations (Proposition \ref{prop:facto}), it follows that $$\pi_0\left(f\right):\pi_0\left(\Y\right) \to  \pi_0\left(\X\right)$$ is an isomorphism.
\end{proof}

\begin{rmk}\label{rmk:locful2}
The converse is not true when $\X$ is not a sheaf. However, note that if $\pi_0\left(f\right)$ is an isomorphism, then in particular $\pi_0\left(f\right)$ is an epimorphism of sheaves, which implies that $f$ is an epimorphism. So, $f$ is connected if and only if, additionally, the diagonal map $$f:\Y \to \Y\times_\X \Y$$ is an epimorphism.
\end{rmk}

\subsection{A categorical characterization of ineffectivity}
\begin{thm}\label{thm:whateffmeans}
The following conditions on an \'etale stack $\X$ are equivalent:
\begin{itemize}
\item[1)] $\X$ is effective.
\item[2)] If $\X \simeq \left[\G\right],$ the unique local homeomorphism $\X \to \HH\left(\G_0\right)$ is representable.
\item[3)] If $\X \simeq \left[\G\right],$ under the equivalence $$\tilde y^{et}:\Et\left(\G_0\mbox{-}\Mfd\right)^{et}\stackrel{\sim}{\longrightarrow}\St\left(\G_0\mbox{-}\Mfd^{et}\right),$$
    $\tilde y^{et}\left(\X\right)$ is a sheaf.
\item[4)] For any space $X$ such that $\X$ is an \'etale $X$-manifold stack, the unique local homeomorphism $\X \to \HH\left(X\right)$ is representable.
\item[5)] For any space $X$ such that $\X$ is an \'etale $X$-manifold stack, under the equivalence $$\tilde y^{et}:\Et\left(\xm\right)^{et}\stackrel{\sim}{\longrightarrow}\St\left(\xm^{et}\right),$$
    $\tilde y^{et}\left(\X\right)$ is a sheaf.
\end{itemize}
\end{thm}
\begin{proof}
$$1)\Longleftrightarrow 2):$$
Consider the weak pullback diagram of $S$-groupoids:
$$\xymatrix{P \ar[d]_-{p} \ar[r] & \G \ar[d]^-{\tilde \iota_\G}\\
\G_0 \ar[r] & \Ha\left(\G_0\right).}$$
Concretely, we may describe $P$ as the groupoid whose objects form the space
$$P_0:=\set{\left(x,y,\varphi\right)\in \G_0 \times \G_0 \times \Ha\left(\G_0\right)_1\mspace{3mu}|\mspace{3mu}\varphi:x \to y},$$ where an arrow from $\left(x,y,\varphi\right)$ to $\left(x',y',\varphi'\right)$ exists if and only if $x=x',$ and if so is given by an arrow $$g:y \to y'$$ in $\G,$ such that $$\left[g\right] \circ \varphi=\varphi'.$$ Since representability in this context is a local property, and $$\G_0 \to \HH\left(\G_0\right)$$ is an atlas, it follows that $$!_\X:\X \to \HH\left(\G_0\right)$$ is representable if and only if $\left[P\right]$ is representable. Notice that if $\left[P\right]$ is representable, then $\Gamma\left(\left[p\right]\right)$ is a sheaf over $\G_0,$ and one can identify $\left[P\right]$ with its \'etal\'e space, hence in particular $\left[P\right]$ is representable. It follows that $!_\X$ is representable if and only if $\left[P\right]$ is a sheaf. This in turn is if and only if each isotropy group of $P$ is trivial. A quick calculation shows that the isotropy group of $\left(x,y,\varphi\right)$ is those $g$ in the isotropy group of $y$ in $\G,$ such that $$\left[g\right]=1_y,$$ in other words, it is the kernel of $$Aut\left(y\right) \to \mathit{Diff}_y\left(\G_0\right)$$- the ineffective isotropy group, as in Definition \ref{dfn:ineffgp}. By definition, $\X$ is effective if and only if each of these kernels are trivial.
$$2)\Longleftrightarrow 3):$$
Notice that $y^{et}\left(\X\right)$ is a sheaf if and only if $\Gamma\left(!_\X\right)$ is a sheaf over $\HH\left(\G_0\right).$ Hence, we are done by \cite{etalspme}, Theorem \ref{thm:shhhhvs}.
Similarly, $4)\Longleftrightarrow 5).$
$$3)\Longleftrightarrow 5).$$
$5) \Longrightarrow 3)$ is obvious. To see $3) \Longrightarrow 5)$, note by Remark \ref{rmk:trunrep}, it suffices to show that $$\X \to \HH\left(X\right)$$ is truncated. Notice there is a factorization
\begin{equation}\label{eq:compeff}
\X \to \HH\left(\G_0\right) \to \HH\left(X\right)
\end{equation}
By Remark \ref{rmk:subterminal}, $\HH\left(\G_0\right),$ regarded as an \'etale $X$-manifold stack, is subterminal, which means that $$\HH\left(\G_0\right) \to \HH\left(X\right)$$ is, in the terminology of \cite{htt}, $\left(-1\right)$-truncated (and hence in particular, truncated, in our sense). It follows that the  composite (\ref{eq:compeff}) is truncated, and hence representable.
\end{proof}

Theorem \ref{thm:whateffmeans} is the true categorical meaning of \emph{effectivity} of \'etale stacks. Indeed, by Theorem \ref{thm:prol3}, we know that \'etale stacks are precisely stacks in  $\St\left(S\right)$ which are prolongations of some stack on the category of $X$-manifolds and local homeomorphisms, for some space $X.$ By Theorem \ref{thm:whateffmeans}, it follows immediately that those \'etale stacks which are \emph{effective} are those which arise as prolongations of sheaves:

\begin{cor}
An \'etale stack $\X$ is effective if and only if $\X\simeq \left(\tau_X\right)_!\left(F\right)$ for some sheaf $F$ on the category of $X$-manifolds and local homeomorphisms, for some space $X,$ with notation as in Theorem \ref{thm:prol3}.
\end{cor}

The following corollary of Theorem \ref{thm:whateffmeans} is also immediate:

\begin{cor}
For any space $X,$ the equivalence
$$\tilde y^{et}:\Et\left(\xm\right)^{et}\stackrel{\sim}{\longrightarrow}\St\left(\xm^{et}\right),$$
restricts to an equivalence
$$\Eft\left(\xm\right)^{et}\stackrel{\sim}{\longrightarrow}\Sh\left(\xm^{et}\right),$$
between effective \'etale $X$-manifold stacks and local homeomorphisms, and sheaves on the category of $X$-manifolds and local homeomorphisms. (In particular $\Eft\left(\xm\right)^{et}$ is equivalent to a $1$-category.)
\end{cor}

One may interpret this corollary as saying that effective \'etale stacks are those \'etale stacks $\X$ such that the stack of local homeomorphisms into $\X$ is a sheaf of sets, rather than a stack of groupoids.

By uniqueness of left adjoints, we also have the following:

\begin{cor}\label{cor:effispizero}
Under the equivalence $$\tilde y^{et}:\Et\left(\xm\right)^{et}\stackrel{\sim}{\longrightarrow}\St\left(\xm^{et}\right),$$
the functor $\Ef_{et}:\Et\left(\xm\right)^{et} \to \Eft\left(\xm\right)^{et},$ left adjoint to the inclusion as in Corollary \ref{cor:etadj}, corresponds to $$\pi_0:\St\left(\xm^{et}\right) \to \Sh\left(\xm^{et}\right).$$
\end{cor}

\begin{ex}\label{ex:Riemannian2}
This has applications to classifications of transverse structures in foliation theory. To see this, one first needs to recast some known results into the setting of \'etale differentiable stacks. Let $\HH$ denote the Haefliger stack of $\coprod\limits_n^\infty \RR^n.$

\begin{thm} \cite{Ie}
For $M$ a smooth manifold, equivalence classes of submersions $$M \to \HH$$ are in bijection with regular foliations on $M$. (If it factors through $\HH\left(\RR^q\right)$ then it is a $q$-codimensional foliation.)
\end{thm}

If $F:M \to \HH$ is such a submersion, then the induced foliation is just the kernel $$\ker\left(\mathbf{T}F\right):\mathbf{T}M \to F^*\mathbf{T}\HH,$$ viewed as an involutive subbundle of $\mathbf{T}M,$ which makes sense since $F^*\mathbf{T}\HH$ is a vector bundle since $\HH$ is \'etale - see \cite{hepworth}.

It is also proven in \cite{Ie} (in the language of smooth \'etendues) that any submersion $\X \to \Y$ between \'etale differentiable stacks, factors uniquely as $$\X \to \Y' \to \Y$$ with $\X \to \Y'$ a submersion with connected fibers, and $$\Y' \to \Y$$ a representable local diffeomorphism. (We will not elaborate on the meaning of ``connected fibers'' in this paper.) Moreover, it is shown that if $\F:M \to \HH$ is a submersion classifying a foliation on $M$, the factorization is given by $$M \to \left[\operatorname{Hol}\left(M,\F\right)\right] \to \HH,$$ where $\operatorname{Hol}\left(M,\F\right)$ is the holonomy groupoid \cite{holom} of the foliation. Armed with Theorem \ref{thm:terminal}, we know the map $$\left[\operatorname{Hol}\left(M,\F\right)\right] \to \HH$$ is the unique local diffeomorphism. We denote $$M//\F:=\left[\operatorname{Hol}\left(M,\F\right)\right]$$ and call it the \textbf{stacky leaf space.} In the nice cases where the actual leaf space (obtained as a quotient space) is a manifold, it is the same as the stacky leaf space.

Let $\R$ be the classifying stack for Riemannian metrics, as in Example \ref{ex:Riemannian1}. Suppose that $$M \to \R$$ is a \emph{submersion}. Then it factors uniquely as $M \to \R' \to \R,$ with $M \to \R'$ a submersion with connected fibers, and $$\R' \to \R$$ a representable local diffeomorphism. However, one can also consider the composite $$M \to \R \to \HH$$  by the unique local diffeomorphism $\R \to \HH.$ This is a submersion $\F:M \to \HH$ so it classifies a foliation. Moreover, it can be factored as the submersion with connected fibers $M \to \R'$ follows by the local diffeomorphism $$\R' \to \R \to \HH.$$ By Theorem \ref{thm:whateffmeans}, the unique map $$\R \to \HH$$ is a representable local diffeomorphism, hence so is the composite $$\R' \to \R \to \HH.$$ By uniqueness of factorizations, one has that $\R'\simeq M//\F,$ so that one gets a local diffeomorphism $$M//\F \to \R$$  between the stacky leaf space and the classifying stack for Riemannian metrics. This corresponds to a Riemannian metric \emph{on} the stacky leaf space, by Example \ref{ex:Riemannian1}. Moreover, the following result was proven by Camilo Angulo as part of a master class project that I advised:

\begin{thm}
Riemannian metrics on $M//\F$ are in natural bijection with \\ \underline{transverse metrics} on $M$ with respect to the foliation $\F$ (in the sense of \cite{Fol}).
\end{thm}
Combining this with the previous discussion, one has that submersions into the classifying stack for Riemannian metrics, classify Riemannian foliations:

\begin{thm}
For a given manifold $M,$ equivalence classes of submersions $$M \to \R$$ into the classifying stack of Riemannian metrics are in natural bijection with Riemannian foliations of $M$.
\end{thm}
By an analogous argument, one sees that submersions $M \to \R_q$ are the same as Riemannian foliations \emph{of codimension $q$}. Note that the \emph{homotopy type} of the underlying topological stack $U_!\left(\R_q\right)$ in the sense of \cite{NoohiH} is that of $BR\Gamma^q,$ so the stack $\R_q$ contains at least as much information as the classifying space $BR\Gamma^q,$ however it is more refined as it is able to classify Riemannian foliations up to isomorphism, on \emph{all manifolds} in a direct geometric way, whereas the classifying space only classifies Riemannian foliations up to homotopy, and on \emph{open} manifolds, and moreover, the way in which these foliations are classified by $BR\Gamma^q$ is subtle (one needs to use $BR\Gamma^q$ and a natural $q$-dimensional vector bundle sitting over it: see \cite{Haefliger}).
\end{ex}

\begin{ex}
In a similar vein, if $\S$ is the classifying stack for symplectic forms, as in Example \ref{ex:symplectic1}, it can be shown that submersions $M \to \S,$ with $M$ a manifold, are in natural bijection with foliations with transverse symplectic structure, or equivalently constant rank closed $2$-forms on $M$ (see \cite{quant}, Section 3), and similarly submersions $M \to \S_n$ classify closed $2$-forms of constant rank $2n.$ Just as in Example \ref{ex:Riemannian2}, the homotopy type of $\S_n$ is that of $B\Gamma^{Sp}_{2n},$ but $\S_n$ is a much richer object.
\end{ex}

\begin{thm}\label{thm:loceffeqisgerb}
Let $f:\Z \to \X$ be a local homeomorphism of \'etale stacks, with $\X$ effective. Then $\Gamma\left(f\right) \in \St\left(\X\right)$ is a gerbe if and only if $f$ is an effective local equivalence (Definition \ref{dfn:effectlocal}), i.e.  the induced map $\Ef\left(\Z\right) \to \X$ is an equivalence.
\end{thm}

\begin{proof}
The map $f$ encodes a gerbe if and only if it is connected. By Theorem \ref{thm:whateffmeans}, under the equivalence,
$$\tilde y^{et}:\Et\left(\xm\right)^{et}\stackrel{\sim}{\longrightarrow}\St\left(\xm^{et}\right),$$
$\X$ is a sheaf, so we are done by Remark \ref{rmk:sheafconn}.
\end{proof}

\begin{cor}
If $\X$ is an orbifold, it encodes a small gerbe over its effective part $\Ef\left(\X\right)$ via $$\iota_\X:\X \to \Ef\left(\X\right),$$ where $\iota$ is the unit of the adjunction in Theorem \ref{thm:bungalo}.
\end{cor}

Theorem \ref{thm:loceffeqisgerb} is the same result as Corollary 6.4 of \cite {gerbes},  however the proof here is much more conceptual and clear. The meaning of this theorem is any \'etale stack $\Y$ arises as the \'etal\'e realization of some gerbe $\g$ over an effective \'etale stack $\X.$ More specifically, $\Y$ arises as the gerbe over is effective part $\Ef\left(\Y\right)$ given by sections of the canonical map $$\Y \to \Ef\left(\Y\right).$$
In some sense, this reduces the structure theory of \'etale stacks, to those of effective \'etale stacks. We will return to this point shortly in Section \ref{sec:geom}. We now take this opportunity to make a subtle correction to \cite{gerbes} Corollary 6.7:

\begin{thm}\label{thm:localfullgerb}
Let $f:\Z \to \X$ be a local homeomorphism of \'etale stacks. Then $\Gamma\left(f\right) \in \St\left(\X\right)$ is a gerbe if and only if $f$ is \underline{locally} full (in the sense of Remark \ref{rmk:locfull}) and an effective local equivalence.
\end{thm}

\begin{proof}
The map $f$ encodes a gerbe if and only if it is connected. By Remark \ref{rmk:locful2} combined with Corollary \ref{cor:effispizero}, this is if and only if $f$ is an effective local equivalence, and diagonal $$\Z \to \Z \times_\X \Z$$ is an epimorphism. By Remark \ref{rmk:locfull}, the latter is equivalent to $f$ being locally full.
\end{proof}

\begin{rmk}
In the preprint \cite{gerbes}, the condition above read ``full'' instead of ``locally full.''
\end{rmk}

\subsection{Geometric meaning of ineffective data}\label{sec:geom}

\begin{thm}
Let $\X$ be an \'etale stack. A small stack $\Z$ over $\X$ is a small gerbe if and only if for every point $$x:* \to \X,$$ the stalk $\Z_x$ of $\Z$ at $x$ is equivalent to a group (viewed as a one-object groupoid).
\end{thm}
\begin{proof}
Fix an \'etale groupoid $\h$ such that $\X \simeq \left[\h\right]$ and $\tilde x \in \h_0$ a point such  that $x \cong p \circ \tilde x,$ where $p:\h_0 \to \X$ is the atlas associated to $\h$. Suppose that $\g$ is a small gerbe over $\X$. Then, since $\g$ is locally non-empty, $$\g_x\simeq  \underset{\tilde x \in U} \hc\g\left(U\right),$$ is a non-empty groupoid. Furthermore, since $\g$ is locally connected, it follows that $$\underset{\tilde x \in U} \hc\g\left(U\right)$$ is also connected, hence, equivalent to a group.

Conversely, suppose that $\Z$ is a small stack and that $$\Z_x\simeq  \underset{\tilde x \in U} \hc\Z\left(U\right)$$ is equivalent to a group. This means it is a non-empty and connected groupoid. It follows that $\Z$ is locally non-empty and locally connected, hence a gerbe.
\end{proof}

\begin{thm}
Suppose that $\X$ is an \'etale stack. $\X$ is canonically a gerbe $\g$ over the effective part of $\X,$ $\Ef\left(\X\right),$ and for a point $x$ of $\X,$ the stalk $\g_x$ is equivalent to the ineffective isotropy group of $x,$ (viewed as a one-object groupoid).
\end{thm}

\begin{proof}
Suppose that $\X \simeq \left[\G\right],$ for some \'etale $S$-groupoid $\G.$ We can factor $x$ through the atlas provided by $\G,$ as $$* \stackrel{\tilde x}{\longrightarrow} \G_0 \to \X.$$ Consider the weak pullback of $S$-groupoids
$$\xymatrix{P \ar[d] \ar[r] & \G \ar[d]^-{\iota_\G}\\
{*} \ar[r]^-{\tilde x} & \Ef\left(\G\right).}$$
The objects $P_0$ of $P$ are the source-fiber of $\tilde x$ in $\G,$ that is the subspace of $\Ef\left(\G\right)_1,$ consisting of arrows $$\varphi:\tilde x \to y,$$ for some $y \in \G_0.$ An arrow in $P$ between $$\varphi:\tilde x \to y$$ and $$\varphi':\tilde x \to y'$$ is the choice of a $g:y \to y'$ in $\G,$ such that $$\left[g\right] \circ \varphi = \varphi'.$$ Since $${\iota_\G}:\G \to \Ef\left(\G\right)$$ is essentially surjective, it follows that $P$ is in fact transitive and equivalent to a group (this is also ensured by the fact that ${\iota_\G}$ encodes a gerbe). To find a group to which $P$ is equivalent, we may simply choose the object $$1_{\tilde x}:\tilde x \to \tilde x$$ and compute its isotropy group. However, one easily sees, by direct inspection, that this isotropy group is the kernel of $$\G_{\tilde x} \to \mathit{Diff}_{\tilde x}\left(\G_0\right),$$ which is the definition of the ineffective isotropy group.
\end{proof}

The significance of this theorem is the following:

Suppose we are given an effective \'etale stack $\X$ and a small gerbe $\g$ over it. By taking stalks, we get an assignment to each point $x$ of $\X$ a group $\g_x.$ From this data, we can build a new \'etale stack by taking the \'etale realization of $\g$. Denote this new \'etale stack by $\Y$. If $\g$ is non-trivial, $\Y$ is not effective, but it has $\X$ as its effective part and, for each point $x$ of $\X$, the stalk $\g_x$ is equivalent to the ineffective isotropy group of $x$ in $\Y$. In particular, if $\X$ is a space $X$, $\Y$ is an \'etale stack which ``looks like $X$'' except that each point $x \in X,$ instead of having a trivial automorphism group, has a group equivalent to $\g_x$ as an automorphism group. In this case, every automorphism group consists entirely of purely ineffective automorphisms.

\begin{rmk}
To a certain extent, the structure of the bicategory of \'etale stacks is determined by the bicategory of effective \'etale stacks. More precisely, if $P$ is any open \'etale invariant class of morphisms, as in Definition \ref{dfn:etalinv}, we can consider the bicategory $\Et_P$ whose objects are \'etale stacks, and whose morphisms are all those that are of class $P.$ Denote by $\Eft_{P}$ the full subcategory on those \'etale stacks which are effective. Taking the effective part (as in Corollary \ref{cor:etadj}) produces a functor $$\Eft_{P}:\Et_P \to \Eft_{P}$$ which is a Co-Cartesian fibration in the sense of \cite{htt} classifying the trifunctor $$\scr{Gerbe}\left(\mspace{3mu}\cdot\mspace{3mu}\right):\Eft_{P}^{op} \to \left(2,1\right)\mbox{-}\mathit{Cat}$$ which assigns to each effective \'etale stack $\X,$ its $\left(2,1\right)$-category of small gerbes. Less formally, this allows one to view objects of $\Et_P$ as ``gerbed effective \'etale stacks.'' See \cite{gerbes}, Section 7, for more details.
\end{rmk}

\appendix

\section{Categories of Spaces.} \label{sec:catspaces2}
We now formalize exactly what properties are needed of a category of spaces for the results of this paper to apply to it. We follow nearly the same definitions as \cite{etalspme} except that we replace the role of the category of locales with that of (sober) topological spaces.

\begin{dfn}\label{dfn:admiss}
Let $\T$ denote the category of sober topological spaces and let $U: S \to \T$ be a category over $\T.$ Let $S^{et}$ be a subcategory of $S$, on the same objects, such that for every $$f: X \to T$$ in $S^{et},$ $U\left(f\right)$ is a local homeomorphism. $U$ induces a functor $$U^{et}: S^{et} \to \T^{et},$$ where $\T^{et}$ denotes the category of topological spaces and local homeomorphisms. $S^{et}$ is called \textbf{admissible} if the following conditions hold:
\begin{itemize}
\item[1)] Every isomorphism in $S$ is in $S^{et}.$
\item[2)] If $f$ and $g$ are composable arrows of $S$ such that $gf$ exists, then if $f$ and $g$ are in $S^{et}$, so is $gf.$ If instead $gf$ and $g$ are in $S^{et},$ so is $f.$
\item[3)] $S$ has pullbacks along morphisms in $S^{et},$ $S^{et}$ is stable under these, and $U$ preserves these pullbacks.
\item[4)] The canonical functor $S^{et} \to S$ preserves coequalizers.
\item[5)] If $f: UT \to UZ$ is a local homeomorphism and there exists a family of morphisms in $S^{et}$ $$\left(T_\alpha \to T\right)_\alpha$$ such that the induced morphism $$\coprod\limits_\alpha UT_\alpha \to UT$$ is a surjective local homeomorphism, and each composite $$UT_\alpha \to UT \to UZ$$ is equal to $U\left(\varphi_\alpha\right)$ for some $\varphi_\alpha$ in $S,$ then there exists a morphism $\varphi: T \to Z$ in $S$ such that $U\left(\varphi\right)=f.$
\item[6)] The induced functor $$U^{et}: S^{et} \to \T^{et}$$ is faithful and locally an equivalence in the following sense:\\

For every object $T \in S,$ the induced functor $$U^{et}_T: S^{et}/T \to \T^{et}/UT$$ is an equivalence of categories.
\end{itemize}
If $S^{et}$ is admissible, morphisms in $S^{et}$ are called \textbf{$S$-local homeomorphisms.}
\end{dfn}

\begin{dfn}\label{dfn:catspace}
Let $U: S \to \T$ be a category over $\T$ with an admissible subcategory $S^{et}$ of $S$-local homeomorphisms. Then $S$ is called a \textbf{category of spaces} if the following conditions hold:
\begin{itemize}
\item[a)] $S$ has and $U$ preserves arbitrary coproducts.
\item[b)] If $\varphi: UT \to X$ is a surjective local homeomorphism, then there exists a $Z$ in $S$ such that $U\left(Z\right) \cong X.$
\end{itemize}
If $S$ is a category of spaces, we will often refer to objects of $S$ simply as \textbf{spaces}, morphisms as \textbf{continuous}, and $S$-local homeomorphisms as \textbf{local homeomorphisms}.
\end{dfn}

\begin{rmk}
This definition deviates slightly from \cite{etalspme} in that we now require arbitrary coproducts in $a)$ above rather than only finite ones.
\end{rmk}

The following is a non-exhaustive list of examples of categories of spaces in the sense of Definition \ref{dfn:catspace}. In each case, the functor $U$ is obvious, so shall be omitted.

\begin{itemize}
\item[I)] Sober topological spaces and local homeomorphisms.
\item[II)] Any type of manifold (e.g. smooth manifolds, $C^k$ manifolds, analytic manifolds, complex manifolds, super manifolds...) with the appropriate version of local diffeomorphism, provided we remove all separation conditions. For example, manifolds will neither be assumed paracompact nor Hausdorff.
\item[III)] Schemes over any fixed base and Zariski local homeomorphisms. When viewed as maps of locally ringed spaces, Zariski local homeomorphisms are those maps $$\left(f,\varphi\right): \left(X,\mathcal{O}_X\right) \to \left(Y,\mathcal{O}_Y\right)$$ such that $f$ is a local homeomorphism and $\varphi: f^*\left(\mathcal{O}_Y\right) \stackrel{\sim}{\longrightarrow} \mathcal{O}_X$ is an isomorphism. Again, we do not impose any separation conditions.
\end{itemize}

\begin{rmk}
By the conventions of Definition \ref{dfn:catspace}, in this paper, if $S$ is taken to be, for example, the category of smooth manifolds, the phrase ``continuous map'' will mean a smooth map, and ``local homeomorphism'' will mean local diffeomorphism. Similarly for the other examples above.
\end{rmk}

\begin{dfn}
We say a collection of local homeomorphisms $\left(T_\alpha \to T\right)_\alpha$ in $S$ is a \textbf{covering family of local homeomorphisms} if the induced morphism $$\coprod\limits_\alpha UT_\alpha \to UT$$ is a surjective local homeomorphism of topological spaces. The family is called a \textbf{open covering family} if each map $$UT_\alpha \to UT$$ is  an open embedding. Each of these notions of covering family define a Grothendieck pre-topology on $S.$ They both generate the same Grothendieck topology, which we shall call the \textbf{open cover topology} on $S.$ We shall denote by $\Sh\left(S\right)$ and $\St\left(S\right)$ the category of sheaves on $S$ and the bicategory of stacks on $S$ respectively, both with respect to the open cover topology.
\end{dfn}
For a list of consequences of the definitions, we refer the reader to \cite{etalspme}.

\section{\'Etale Stacks and their Sheaf Theory}\label{sec:prelim}
In this appendix, we recall the basic theory of \'etale stacks and establish the notational conventions concerning them. We then give a summary of the \'etal\'e realization construction from \cite{etalspme}, which is used in an essential way in this paper. Finally, we end with a brief discussion about stalks of stacks.

Denote by $S$ a category of space satisfying the axioms of Definition \ref{dfn:catspace}.

\begin{dfn}
An \textbf{$S$-groupoid} is a groupoid object in $S.$ For example, a \textbf{topological groupoid} is a groupoid object in $\T$, the category of topological spaces. Explicitly, it is a diagram

$$\xymatrix{ {\G_1 \times _{\G_0}\G_1} \ar[r]^(0.6){m} & \G_1 \ar@<+.7ex>^s[r]\ar@<-.7ex>_t[r] \ar@(ur,ul)[]_{\hat i}  & \G_0 \ar@/^1.65pc/^{1} [l] }$$
of topological spaces and continuous maps satisfying the usual axioms. Forgetting the topological structure (i.e. applying the forgetful functor from $\T$ to $\Set$), one obtains an ordinary small groupoid. Throughout this article, we shall denote the source and target maps of a groupoid by $s$ and $t$ respectively.

$S$-groupoids form a $2$-category with continuous functors as $1$-morphisms and continuous natural transformations as $2$-morphisms, respectively. (Recall that e.g. when $S$ is smooth manifolds, by continuous, we mean smooth.) We will denote this $2$-category by $S\mbox{-}\Gpd$.
\end{dfn}

\begin{rmk}
A \textbf{Lie groupoid} is a groupoid object in smooth manifolds such that the source and target maps are \emph{submersions}. Traditionally, Lie groupoids are required to have a Hausdorff object space, however, as every manifold is locally Hausdorff, any Lie groupoid in the sense we defined is Morita equivalent to one that meets this requirement. (See Definition \ref{dfn:Morita}.) We will not dwell on this issue as we will soon restrict our attention to \'etale groupoids.
\end{rmk}

Consider the $2$-category $\Gpd^{S^{op}}$\ of weak presheaves in groupoids over $S$, that is contravariant (possibly weak) $2$-functors from the category $S$ into the 2-category of (essentially small) groupoids $\Gpd$\footnote{Technically speaking, we may have to restrict ourselves to a Grothendieck universe of such spaces. If $S$ is smooth manifolds, we may avoid this by replacing $\Sts$ with stacks on Cartesian manifolds, i.e., manifolds of the form $\mathbb{R}^n$, which forms a small site.}.

We recall the $2$-Yoneda Lemma:

\begin{lem}\cite{FGA}
If $C$ is an object of a category $\C$ and $\X$ a weak presheaf in $\Gpd^{\C^{op}}$, then there is a natural equivalence of groupoids
$$\Hom_{\Gpd^{\C^{op}}}\left(C,\X\right) \simeq \X\left(C\right),$$
where we have identified $C$ with its representable presheaf under the Yoneda embedding.
\end{lem}
If $G$ is a topological group or a Lie group, then a standard example of a weak presheaf is the functor that assigns to each space the category of principal $G$-bundles over that space (this category is a groupoid). More generally, let $\G$ be an $S$-groupoid. Then $\G$ determines a weak presheaf on $S$ by the rule

\begin{equation*}
X \mapsto \Hom_{S\mbox{-}\Gpd}\left(\left(X\right)^{(id)},\G\right),
\end{equation*}
where $\left(X\right)^{(id)}$ is the $S$-groupoid whose object space is $X$ and has only identity morphisms. This defines an extended Yoneda $2$-functor $\tilde y: S\mbox{-}\Gpd \to \Gpd^{S^{op}}$ and we have the obvious commutative diagram

$$\xymatrix{S  \ar[d]_{\left(\mspace{2mu} \cdot \mspace{2mu}\right)^{(id)}} \ar[r]^{y} & \Set^{S^{op}} \ar^{\left(\mspace{2mu} \cdot \mspace{2mu}\right)^{(id)}}[d]\\
S\mbox{-}\Gpd \ar_{\tilde y}[r] & \Gpd^{S^{op}},}$$
where $y$ denotes the Yoneda embedding. We denote by $\left[\G\right]$ the associated stack on $S$, $a \circ \tilde y\left(\G\right)$, where $a$ is the stackification $2$-functor (with respect to the open cover Grothendieck topology). $\left[\G\right]$ is called the \textbf{stack completion} of the groupoid $\G$.

\begin{rmk}
There is a notion of principal bundle for topological groupoids and Lie groupoids, and $\left[\G\right]$ is in fact the functor that assigns to each space the category of principal $\G$-bundles over that space.
\end{rmk}

\begin{dfn}
A stack $\X$ on $\T$ is a \textbf{topological stack} if it is equivalent to $\left[\G\right]$ for some topological groupoid $\G$. A stack $\X$ on $\Mfd$, the category of smooth manifolds, is a \textbf{differentiable stack} if it is equivalent to $\left[\G\right]$ for some Lie groupoid $\G$.
\end{dfn}

\begin{dfn}
An $S$-groupoid $\G$ is \textbf{\'etale} if its source map $s$ (and therefore also its target map $t$) is a local homeomorphism.
\end{dfn}

\begin{dfn}\label{dfn:etalestack}
A stack $\X$ on $S$ is \textbf{\'etale} if it is equivalent to $\left[\G\right]$ for some \'etale $S$-groupoid $\G$. An \textbf{\'etale topological stack} is an \'etale stack on $S=\T,$ and an \textbf{\'etale differentiable stack} is an \'etale stack on $S=\Mfd.$ For any $S,$ we denote the full subcategory of $\St\left(S\right)$ on the \'etale stacks by $\Et\left(S\right).$
\end{dfn}

\begin{dfn}
A morphism $f: \Y \to \X$ of stacks is said to be \textbf{representable} if for any map from a space $T \to \X$, the weak $2$-pullback $T \times_{\X} \Y$ is (equivalent to) a space.
\end{dfn}

\begin{dfn}\label{dfn}
A morphism $\varphi: \X \to \Y$ between stacks is an \textbf{epimorphism} (or in the language of \cite{htt}, $0$-connective)  if it is locally essentially surjective in the following sense:

For every space $X$ and every morphism $f: X \to \Y,$ there exists an open cover  $\mathcal{U}=\left(U_i \hookrightarrow X\right)_i$ of $X$ such that for each $i$ there exists a map $\tilde f_i: U_i \to \Y,$ such that the following diagram $2$-commutes:

$$\xymatrix@M=5pt{U_i \ar@{^(->}[d] \ar[r]^-{\tilde f_i} & \Y \ar[d]^-{\varphi}\\
X \ar[r]^-{f} & \X.}$$
In words, this just means any map $X \to \Y$ from a space $X$ locally factors through $\varphi$ up to isomorphism.
\end{dfn}

\begin{dfn}
An \textbf{atlas} for a stack $\X$ is a representable epimorphism $X \to \X$ from a space $X$.
\end{dfn}

\begin{rmk}
A stack $\X$ comes from an $S$-groupoid if and only if it has an atlas. If $X \to \X$ is an atlas, then $\X$ is equivalent to the stack completion of the groupoid $X \times_\X X \rightrightarrows X$. Conversely, for any $S$-groupoid $\G$, the canonical morphism $\G_0 \to \left[\G\right]$ is an atlas.
\end{rmk}

\begin{dfn}\label{dfn:local}
Let $P$ be a property of a map of spaces. It is said to be \textbf{invariant under change of base} if for all $$f:  Y \to X$$ with property $P$, if $$g: Z \to X$$ is any representable map, the induced map $$Z \times_X Y \to Z$$ also has property $P$. The property $P$ is said to be \textbf{invariant under restriction}, if this holds whenever $g$ is an open embedding. Being invariant under change of base implies being invariant under restriction. A property $P$ which is invariant under restriction is said to be \textbf{local on the target} if any $$f:  Y \to X$$ for which there exists an open cover $\left(U_\alpha \to X\right)$  such that the induced map $$\coprod\limits_\alpha  {U_\alpha  }  \times_{X} Y \to \coprod\limits_\alpha  {U_\alpha  } $$ has property $P$, must also have property $P$.
\end{dfn}

Examples of such properties are being an open map, local homeomorphism, proper map, closed map etc.

\begin{prop}
A stack $\X$ over $S$ is \'etale if and only if it admits an \'etale atlas $p: X \to \X$, that is a representable epimorphism which is a local homeomorphism.
\end{prop}

\begin{proof}
This follows from the fact that if $\G$ is any $S$-groupoid, the following diagram is $2$-Cartesian:
$$\xymatrix{\G_1 \ar[d]_{s} \ar[r]^{t} & \G_0 \ar[d]\\
\G_0 \ar[r] & \left[\G\right],}$$
where the map $\G_0 \to \left[\G\right]$ is induced from the canonical map $\G_0 \to \G$.
\end{proof}

\begin{rmk} Traditionally speaking, a \textbf{differentiable stack} is a stack $\X$ equivalent to $\left[\G\right]$ where $\G$ is a Lie groupoid. This is equivalent to it having an atlas which is a representable submersion.
\end{rmk}

\begin{dfn}\label{dfn:Morita}
An internal functor $\varphi: \h \to \G$ of $S$-groupoids is a \textbf{Morita equivalence} if the following two properties hold:

\begin{itemize}
\item[i)] (Essentially Surjective)\\
The map $t \circ pr_1: \G_1 \times_{\G_0} \h_0 \to \G_0$ admits local sections, where $\G_1 \times_{G_0} \h_0$ is the fibered product

$$\xymatrix{\G_1 \times_{G_0} \h_0 \ar[r]^-{pr_2} \ar[d]_-{pr_1} & \h_0 \ar[d]^-{\varphi} \\
\G_1 \ar[r]^-{s} & \G_0.}$$
\item[i)] (Fully Faithful)
The following is a fibered product:

$$\xymatrix{\h_1 \ar[r]^-{\varphi} \ar[d]_-{\left(s,t\right)} &\G_1 \ar[d]^-{\left(s,t\right)}  \\
\h_0 \times \h_0 \ar[r]^-{\varphi \times \varphi} & \G_0 \times \G_0.}$$
\end{itemize}
Two $S$-groupoids $\Ll$ and $\K$ are \textbf{Morita equivalent} if there is a chain of Morita equivalences $\Ll \leftarrow \h \rightarrow \K$. Moreover, $\Ll$ and $\K$ are Morita equivalent if and only if $\left[\Ll\right] \simeq \left[\K\right]$
\end{dfn}

Every internal functor $\h \to \G$ induces a map $\left[\h\right] \to \left[\G\right]$ and the induced functor $$\Hom\left(\h,\G\right) \to \Hom\left(\left[\h\right],\left[\G\right]\right)$$ is full and faithful, but not in general essentially surjective. However, any morphism $$\left[\h\right] \to \left[\G\right]$$ arises from a chain $$\h \leftarrow \K \rightarrow \G,$$ with $\K \to \h$ a Morita equivalence. In fact, the class of Morita equivalences admits a calculus of fractions, and stacks arising from $S$-groupoids are equivalent to the bicategory of fractions of $S$-groupoids with inverted Morita equivalences. For details see \cite{Dorette}.

By the above discussion, any morphism $$\varphi:\X \to \Y$$ between \'etale stacks arising from some internal functor of \'etale groupoids $\h \to \G$ such that $\left[\h\right] \simeq \X$ and $\left[\G\right]\simeq\Y.$ Hence, the following definition makes sense:

\begin{dfn}\label{dfn:localhomeostacks}
A morphism $\varphi:\X \to \Y$ between \'etale stacks is a \textbf{local homeomorphism} if it arises from an internal functor of \'etale groupoids $\phi:\h \to \G,$ such that $$\phi_0:\h_0 \to \G_0$$ is a local homeomorphism.
\end{dfn}

\begin{dfn}
By an \textbf{\'etale cover} of a space $X$, we mean a surjective local homeomorphism $U \to X$. In particular, for any open cover $\left(U_\alpha\right)$ of $X$, the canonical projection $$\coprod\limits_{\alpha} U_\alpha \to X$$ is an \'etale cover.
\end{dfn}

\begin{dfn}\label{dfn:cech}
Let $\h$ be an $S$-groupoid. If $\mathcal{U}=U \to \h_0$ is an \'etale cover of $\h_0$, then one can define the \textbf{\v{C}ech groupoid} $\h_{\mathcal{U}}$. Its objects are $U$ and the arrows fit in the pullback diagram

$$\xymatrix{\left(\h_{\mathcal{U}}\right)_1 \ar[r] \ar[d]_-{\left(s,t\right)} & \h_1 \ar[d]^-{\left(s,t\right)} \\
U \times U \ar[r] & \h_0 \times \h_0,}$$
and the groupoid structure is induced from $\h$. There is a canonical map $\h_{\mathcal{U}} \to \h$ which is a Morita equivalence. Moreover,
\end{dfn}
\begin{equation}\label{eq:pizza}
\Hom\left(\left[\h\right],\left[\G\right]\right) \simeq \underset{\mathcal{U} \in Cov\left(\h_0\right)} \hc \Hom_{S\mbox{-}\Gpd} \left(\h_{\mathcal{U}},\G\right),
\end{equation}
where the weak 2-colimit above is taken over a suitable 2-category of \'etale covers. For details see \cite{Andre}.


Applying equation (\ref{eq:pizza}) to the case where $\left[\h\right]$ is a space $X$, by the Yoneda Lemma we have

\begin{equation*}
\left[\G\right] \left(X\right) \simeq \underset{\mathcal{U} \in Cov\left(X\right)} \hc \Hom_{S\mbox{-}\Gpd} \left(X_{\mathcal{U}},\G\right).
\end{equation*}
Our category $S$ of spaces comes equipped with a functor $U:S \to \T,$ such that the restriction $2$-functor $$U^*:\St\left(\T\right) \to \St\left(S\right)$$ has a left adjoint
$$U_!:\St\left(S\right) \to \St\left(\T\right)$$ called the \textbf{prolongation} $2$-functor, which is equivalent to (the Yoneda embedding composed with) $U$ on representables. Moreover, if $\left[\G\right]_{S}$ is an \'etale $S$-stack, $$\left[\G\right]_{S} \simeq \hc \left(\G_2 \rrrarrow \G_1 \rrarrow \G_0\right),$$ and hence
\begin{eqnarray*}
U_!\left(\left[\G\right]_{S}\right) &\simeq& \hc \left(U\G_2 \rrrarrow U\G_1 \rrarrow U\G_0\right)\\
&\simeq& \left[U\G\right]_{\T}.
\end{eqnarray*}
It follows that $U_!$ assigns an \'etale $S$-stack $\X$ its \textbf{underlying} topological stack.

\begin{dfn}
Let $\C$ be a $2$-category, and $C$ an object. The \textbf{slice $2$-category} $\C/C$ has as \textbf{objects} morphisms $\varphi: D \to C$ in $\C$. The \textbf{morphisms} are $2$-commutative triangles of the form
$$\xymatrix@R=0.6cm@C=0.6cm{D \ar[rrdd]_-{\varphi}
 \ar[rrrr]^{f} \ar@{} @<-5pt>  [rrrr]| (0.6) {}="a"
 							& &  & & E \ar[lldd]^-{\psi}\\
						&\ar@{} @<5pt>  [rd]| (0.4) {}="b"
\ar @{=>}^{\alpha}  "b";"a"&	& & \\
						& & C, & &}$$
with $\alpha$ invertible. A \textbf{$2$-morphism} between a pair of morphisms  $\left(f,\alpha\right)$ and $\left(g,\beta\right)$ going between $\varphi$ and $\psi$ is a $2$-morphism in $\C$ $$\omega: f \Rightarrow g$$ such that the following diagram commutes:
$$\xymatrix@R=0.6cm@C=0.6cm{\psi f \ar@{=>}[rr]^-{\psi\omega}& & \psi g\\ & \varphi \ar@{=>}[lu]^-{\alpha} \ar@{=>}[ru]_-{\beta}.}$$
\end{dfn}

We end by a standard fact which we will find useful later:

\begin{prop}\label{prop:stand2}
For any stack $\X$ on $S$, there is a canonical equivalence of $2$-categories $\St\left(S/\X\right)\simeq \St\left(S\right)/\X$.
\end{prop}

The construction is as follows:

Given $\Y \to \X$ in $\St\left(S\right)/\X$, consider the stack $$\tilde \Y(T \to \X): =\Hom_{\St\left(S\right)/\X}\left(T \to \X,\Y \to \X\right).$$ Given a stack $\W$ in $\St\left(S/\X\right)$, consider it as a fibered category $\int{\W} \to S/\X$. Then since $S/\X \simeq \int{\X}$ (as categories), the composition $\int{\W} \to \int{\X} \to S$ is a category fibered in groupoids presenting a stack $\tilde W$ over $S$, and since the diagram

$$\xymatrix{\int{\W} \ar[rd] \ar[d] & \\
\int{\X} \ar[r] & S}$$
commutes,  $\int{\W} \to \int{\X}$ corresponds to a map of stacks $\tilde \W \to \X$.

We leave the rest to the reader.

\subsection{Grothendieck topoi}\label{sec:topoi2}
A concise definition of a Grothendieck topos is as follows:

\begin{dfn}
A category $\E$ is a Grothendieck topos if it is a reflective subcategory of a presheaf category $Set^{\C^{op}}$ for some small category $\C$,

\begin{equation}\label{eq:topos2}
\xymatrix@C=1.5cm{\E \ar@<-0.5ex>[r]_{j_*} & Set^{\C^{op}} \ar@<-0.5ex>[l]_{j^*}},
\end{equation}
with $j^* \rt \mspace{2mu} j_*$, such that the left adjoint $j^*$ preserves finite limits. From here on in, topos will mean Grothendieck topos.
\end{dfn}

\begin{rmk}
It is standard that this definition is equivalent to saying that $\E$ is equivalent to $\Sh_J\left(\C\right)$ for some Grothendieck topology $J$ on $\C$, see for example \cite{Ieke}.
\end{rmk}

\begin{dfn}
A \textbf{geometric morphism} from a topos $\E$ to a topos $\F$ is a an adjoint pair

$$\xymatrix@C=1.5cm{\E \ar@<-0.5ex>[r]_{f_*} & \F \ar@<-0.5ex>[l]_{f^*}},$$
with $f^* \rt \mspace{1mu} f_*$, such that $f^*$ preserve finite limits. The functor $f_*$ is called the \textbf{direct image} functor, whereas the functor $f^*$ is called the \textbf{inverse image} functor.
\end{dfn}

In particular, this implies, somewhat circularly, that equation (\ref{eq:topos2}) is an example of a geometric morphism.

Topoi form a $2$-category. Their arrows are geometric morphisms. If $f$ and $g$ are geometric morphisms from $\E$ to $\F$, a $2$-cell $$\alpha: f \Rightarrow g$$ is given by a natural transformation $$\alpha: f^* \Rightarrow g^*.$$ In this paper, we will simply ignore all non-invertible $2$-cells to arrive at a $(2,1)$-category of topoi, $\Top$.

\subsection{\'Etal\'e realization}
Given any space $X,$ there is a canonically associated topos of sheaves, namely $\Sh\left(UX\right),$ where $$U:S \to \T$$ is the underlying space functor. This produces a $2$-functor $$S \to \Top.$$ By (weak) left-Kan extension, we obtain a $2$-adjoint pair $\Sh \rt \S$
$$\xymatrix@C=1.5cm{\Gpd^{S^{op}} \ar@<-0.5ex>[r]_{\Sh} & \Top \ar@<-0.5ex>[l]_{\S},}$$
where $\Gpd^{S^{op}}$ denotes the bicategory of weak presheaves in groupoids. In fact, the essential image of $\S$ lies entirely within the bicategory of stacks over $S$, $\Sts$, where $S$ is equipped with the open-cover Grothendieck topology \cite{bunge}. So, by restriction, we obtain an adjoint pair

\begin{equation}\label{eq:sheaves2}
\xymatrix@C=1.5cm{\St(S) \ar@<-0.5ex>[r]_{\Sh} & \Top \ar@<-0.5ex>[l]_{\S}.}
\end{equation}

\begin{dfn}\label{dfn:smallsheaves}

For $\X$ a stack over $S$, we define the topos of \textbf{small sheaves} over $\X$ to be the topos $\Sh(\X)$.

\end{dfn}

\begin{rmk}\label{rmk:nv}
Suppose that $\X \simeq \left[\G\right]$ for some $S$-groupoid $\G$. Then we may consider the nerve $N\left(\G\right)$ as a simplicial object in $S.$

$$\xymatrix{ \G_0& \G_1  \ar@<+.7ex>[l] \ar@<-.7ex>[l]& {\G_2 \cdots}  \ar@<0.9ex>[l] \ar@<0.0ex>[l] \ar@<-0.9ex>[l]}.$$
By composition with the Yoneda embedding, we obtain a simplicial stack $$y \circ N\left(\G\right): \Delta^{op} \to \St(S).$$ The weak colimit of this diagram is the stack $\left[\G\right]$. Since $\Sh$ is a left adjoint, it follows that $\Sh\left(\left[\G\right]\right)$ is the weak colimit of the simplicial topos

$$\xymatrix{ \Sh\left(\G_0\right)& \Sh\left(\G_1\right)  \ar@<+.7ex>[l] \ar@<-.7ex>[l]& {\Sh\left(\G_2\right) \cdots}  \ar@<0.9ex>[l] \ar@<0.0ex>[l] \ar@<-0.9ex>[l]}.$$
From \cite{cont}, it follows that $\Sh\left(\left[\G\right]\right) \simeq \B\G$, the classifying topos of $\G$.
\end{rmk}

\begin{dfn}\label{dfn:etendue}
A topos $\E$ is an \textbf{\'etendue} if there exists a well-supported object $E \in \E$ (i.e. $E \to 1$ is an effective epimorphism) such that the slice topos $\E/E$ is equivalent to $\Sh\left(X\right)$ for some topological space $X$.
\end{dfn}

\begin{thm}\label{thm:sga-etendue} (\cite{sga4}, Expos\'e $iv$, Exercice $9.8.2$)
A topos $\E$ is an \'etendue if and only if $\E \simeq \B\G$ for some \'etale topological groupoid $\G$ .
\end{thm}

\begin{rmk}\label{rmk:etendskt}
A proof of this theorem may be found in \cite{Dorette}, Theorem $9$. One direction of this proof, namely, how to construct from an \'etendue $\E$ an \'etale groupoid $\G,$ such that $\B\G\simeq \E,$ is used many times in this article, so we reiterate the idea:\\
Suppose that $E \in \E,$ is well supported and $\E/E \simeq \Sh\left(X\right).$ Set $\G_0:=X.$ The canonical geometric morphism $$\Sh\left(X\right)\simeq \E/E \to \E$$ is hence both \'etale and an effective epimorphism in the $(2,1)$-category of topoi. The latter means that $\E$ is the homotopy colimit of the truncated semi-simplicial \v{C}ech diagram
\begin{equation}\label{eq:cechnerve}
\Sh\left(X\right) \times_{\E} \Sh\left(X\right) \times_{\E} \Sh\left(X\right) \rrrarrow \Sh\left(X\right) \times_{\E} \Sh\left(X\right) \rrarrow \Sh\left(X\right).
\end{equation}
Note however that from the $2$-pullback diagram of topoi
$$\xymatrix{\Sh\left(X\right) \times_{\E} \Sh\left(X\right) \ar[d]_-{s} \ar[r]^-{t} & \E/E \ar[d]^-{p}\\
\Sh\left(X\right) \ar[r]^-{p} & \E,}$$ it follows that $\Sh\left(X\right) \times_{\E} \Sh\left(X\right)$ is equivalent to the topos of sheaves on the \'etal\'e space $\G_1$ of the sheaf $p^*\left(E\right)$ on $X.$ Moreover, the maps $s$ and $t$ become the source and target maps of an \'etale groupoid $\G_1 \rrarrow \G_0,$ such that the \v{C}ech diagram (\ref{eq:cechnerve}) is the functor $\Sh$ applied to the nerve of $\G.$ By Remark \ref{rmk:nv}, it follows that $\E \simeq \B\G.$
\end{rmk}

\begin{thm}\label{thm:etendue}\cite{Dorette}
$\Sh$ induces an equivalence between the bicategory of \'etale topological stacks and the bicategory of \'etendues.
\end{thm}

Here is a more concrete description of classifying topoi of groupoids:

\begin{dfn}
Given an $S$-groupoid $\G$, a (left) \textbf{$\G$-space} is a space $E$ equipped with a \textbf{moment map} $\mu: E \to \G_0$ and an \textbf{action map} $$\rho: \G_1 \times_{\G_0} E \to E,$$
where
$$\xymatrix{\G_1 \times_{\G_0} E \ar[r]  \ar[d] & E \ar^-{\mu}[d] \\
\G_1 \ar^-{s}[r] & \G_0\\}$$
is the fibered product, such that the following conditions hold:

\begin{itemize}
\item[i)] $\left(gh\right) \cdot e = g \cdot \left(h \cdot e\right)$ whenever $e$ is an element of $E$ and $g$ and $h$ elements of $\G_1$ with domains such that the composition makes sense,
\item[ii)] $1_{\mu\left(e\right)} \cdot e =e$ for all $e \in E,$ and
\item[iii)] $\mu\left(g \cdot e\right) = t \left(g\right)$ for all $g \in \G_1$ and $e \in E.$
\end{itemize}
A map of $\G$-spaces is simply an equivariant map, i.e., a map $$(E,\mu,\rho) \to (E',\mu',\rho')$$ is map $f: (E,\mu,) \to (E',\mu')$ in $S/\G_0$ such that

$$f(he)=hf(e)$$
whenever this equation makes sense.
\end{dfn}

\begin{dfn}
A $\G$-space $E$ is an $\G$-\textbf{equivariant sheaf} if the moment map $\mu$ is a local homeomorphism. The category of $\G$-equivariant sheaves and equivariant maps forms the \textbf{classifying topos} $\B\G$ of $\G$.
\end{dfn}

Associated to each atlas, there is also a canonical site:

\begin{dfn}\label{dfn:sit}
Given an \'etale stack $\X$ with an \'etale atlas $X \to \X$, let $\sit\left(\X,X\right)$ denote the following category: The objects of are open subsets of $X$ and the arrows are pairs $\left(f,\alpha\right)$, such that

$$\xymatrix@R=0.4cm@C=0.4cm{U \ar@{^{(}->}[rd]
 \ar@<+0.5ex>[rrrr]^{f} \ar@{} @<-5pt>  [rrrr]| (0.6) {}="a"
 							& &  & & V \ar@{^{(}->}[ld]\\
						&X \ar[rd] \ar@{} @<5pt>  [rd]| (0.4) {}="b"
\ar @{=>}^{\alpha}  "b";"a"&	& X \ar[ld]& \\
						& & \X & &}.$$
In other words, it is the full subcategory of $\St\left(S\right)/\X \simeq \St\left(S/\X\right)$ (Proposition \ref{prop:stand2}) spanned by objects of the form $U \hookrightarrow X \to \X$, with $U  \subseteq X$ open. It comes equipped with a canonical Grothendieck topology (a family $(\left(f_i,\alpha_i\right)_i$ is a covering family if and only if $\left(f_i\right)_i$ is) and
$$\Sh\left(\sit\left(\X,X\right)\right) \simeq \Sh\left(\X\right).$$ For $\h$ an \'etale $S$-groupoid, $\sit\left(\h\right)$ will denote the site $\sit\left(\left[\h\right],\h_0\right).$ $\sit\left(\left[\h\right],\h_0\right)$ can also be described as the category whose objects are the open subsets of $\h_0,$ and whose arrows $U \to V$ are sections $\sigma$ of the source-map $s: \h_1 \to \h_0$ over $U$ such that $t \circ \sigma:  U \to V$ as a map in $S$. Composition is by the formula $\tau \circ \sigma(x): =\tau\left(t\left(\sigma(x\right)\right).$ From this description, taking into account Definition \ref{dfn:admiss} and \cite{etalspme} Proposition \ref{prop:b2}, it follows that this site is canonically equivalent to $\sit\left(U_!\left[\h\right],U\h_0\right),$ the analogous site for the underlying topological stack. We refer the reader to \cite{etalspme}, Section \ref{sec:smallsite} for more details.
\end{dfn}

\begin{dfn}
By a \textbf{small stack} over an \'etale stack $\X\simeq\left[\h\right]$, we mean a stack $\Z$ over $\sit\left(\h\right)$. We denote the $2$-category of small stacks over $\X$ by $\St\left(\X\right)$.
\end{dfn}

\begin{rmk}\label{rmk:sites}
This definition does not depend on the choice of presenting groupoid since, if $\G$ is another groupoid such that $\left[\G\right]\simeq \X$, then $$\Sh\left(\sit\left(\G\right)\right) \simeq \B\G \simeq \B\h \simeq \Sh\left(\sit\left(\h\right)\right)$$ and hence $\St\left(\sit\left(\G\right)\right) \simeq \St\left(\sit\left(\h\right)\right)$ by the Comparison Lemma for stacks \cite{sga4}. A more intrinsic equivalent definition is that a small stack over $\X$ is a stack over the topos $\Sh\left(\X\right)$ in the sense of Giraud in \cite{Giraud}, that is a stack over $\Sh\left(\X\right)$ with respect to the canonical Grothendieck topology, which in this case is generated by jointly epimorphic families. Finally, there is another canonical large site of definition for sheaves over $\X$: the category $S^{et}/\X$ of local homeomorphisms $T \to \X$ with $T$ a space. (See the second remark after Corollary \ref{cor:repshvs} in \cite{etalspme}.)
\end{rmk}

\begin{thm}\label{thm:etalsp} (\cite{etalspme}, Corollary \ref{cor:real}.)
For any \'etale stack $\X$, there is an adjoint equivalence of $2$-categories
$$\xymatrix{\St\left(\X\right)  \ar@<-0.5ex>[r]_-{L}  & Et\left(\X\right)\ar@<-0.5ex>[l]_-{\Gamma}},$$
between small stacks over $\X$ and the $2$-category of \'etale stacks over $\X$ via a local homeomorphism.
\end{thm}

Here $L$ is the \textbf{\'etal\'e realization} functor, and $\Gamma$ is the ``stack of sections'' functor. More explicitly:

If $\X\simeq \left[\h\right],$ $\Gamma(f: \Y \to \X)$ assigns an open subset $U$ of $\h_0$ the groupoid  whose objects are pairs $\left(\sigma,\alpha\right)$ which fit into a $2$-commutative diagram
$$\xymatrix@R=0.6cm{& & \Y \ar[dd]^-{f} & &\\
& & \ar@{} @<3pt>  [r]| (-0.1) {}="a"  &\\
U \ar@{^{(}->}[r] \ar@<+.7ex>[rruu]^-{\sigma} & \h_0\ar[r] \ar@{} @<+12pt>  [r]  | (-.1) {}="b" \ar @{=>}^{\alpha}  "b";"a"& \X,&}$$
and whose morphisms $\left(\sigma,\alpha\right) \to \left(\sigma',\alpha'\right)$ are $2$-cells
$$ \xygraph{!{0;(2,0):(0,.5)::}
{U}="a" [r] { \Y}="b"
"a":@/^{1.5pc}/"b"^-{\sigma}|(.4){}="l"
"a":@/_{1.5pc}/"b"_-{\sigma'}
"l" [d(.3)]  [r(0.1)] :@{=>}^{\omega} [d(.5)]} $$
such that the following diagram commutes:
$$\xymatrix{j_{\h}\left(U\right) \ar@{=>}[r]^-{\alpha} \ar@{=>}[rd]_-{\alpha'} & f \circ \sigma \ar@{=>}[d]^-{f\omega}\\
& f\circ \sigma',}$$
where $j_{\h}\left(U\right)$ is the composite $U \hookrightarrow \h_0 \to \X.$
Moreover, this construction is functorial in $\X$:
\begin{thm}\label{thm:inv2}
Suppose $f: \Y \to \X$ is a morphism of \'etale stacks. Then following diagram $2$-commutes:

$$\xymatrix{\St\left(\X\right) \ar[r]^-{L} \ar[d]_{f^*} & \St\left(S\right)/\X \ar[d]^-{f^*}\\
 \St\left(\Y\right) \ar[r]^-{L} &  \St\left(S\right)/\Y.}$$
\end{thm}
See \cite{etalspme}, Section \ref{sec:inverseimage} for further detail.

\subsection{Stalks}
\begin{dfn}
For $\Z$ a small stack over an \'etale stack $\X$, and $$x:* \to \X$$ a point of $\X,$ the \textbf{stalk} of $\Z$ at $x$ is the groupoid $x^*\left(\Z\right),$ where we have made the identification $\St\left(*\right)\simeq \Gpd.$ We denote this stalk by $\Z_x$.
\end{dfn}

As we have just seen, by Theorem \ref{thm:inv2} this stalk may be computed as the fiber of $$L\left(\Z\right) \to \X$$ over $x,$ i.e. the weak pullback $* \times_{\X} L\left(\Z\right),$ which is a constant stack with value $x^*\left(\Z\right)$. This stalk can also be computed analogously to stalks of sheaves:

\begin{lem}
Let $x \in X$ be a point of a space, and let $\Z$ be a small stack over $X$. Then the stalk at $x$ of $\Z$ can be computed by $$\Z_x\simeq  \underset{x \in U} \hc\Z\left(U\right),$$ where the weak colimit is taken over the open neighborhoods of $x$ regarded as a full subcategory of $\mathcal{O}\left(X\right).$
\end{lem}
\begin{proof}
The $2$-functor
\begin{eqnarray*}
\St\left(X\right) &\to& \Gpd\\
\Z &\mapsto& \underset{x \in U}\hc  \Z\left(U\right),\\
\end{eqnarray*}
is clearly weak colimit preserving. If $\Z=V \subseteq X$ is a representable sheaf, i.e., an open subset of $X$, then $$\Z_x\simeq \underset{x \in U} \hc \Hom\left(U,V\right) \simeq \underset{x \in U} \varinjlim \Hom\left(U,V\right),$$ and the latter expression is equivalent to the singleton set if $x \in V$ and the empty set otherwise. This set is the same as the fiber of $V$ over $x,$ i.e. the stalk $V_x \cong x^*\left(V\right).$ So $$\Z \mapsto \underset{x \in U} \hc \Z\left(U\right)$$ is weak colimit preserving and agrees with $x^*$ on representables, hence is equivalent to $x^*$.
\end{proof}

\begin{cor}\label{cor:stalk}
Let $x:* \to \X$ be a point of an \'etale stack $\X\simeq \left[\h\right],$ with $\h$ an \'etale groupoid. Pick a point $\tilde x \in \h_0$ such that $x\cong p \circ \tilde x$ where $$p:\h_0 \to \X$$ is the atlas associated to $\h$. Let $\Z$ be a small stack over $X$. Then the stalk at $x$ of $\Z$ can be computed by $$\Z_x\simeq  \underset{\tilde x \in U} \hc\Z\left(U\right),$$ where the weak colimit is taken over the open neighborhoods of $\tilde x$ in $\h_0$ regarded as a full subcategory of $\mathcal{O}\left(\h_0\right).$
\end{cor}

\begin{proof}
Since $x\cong p \circ \tilde x,$ it follows that $$x^* \simeq \tilde x^* \circ p^*.$$ By definition, for $U$ an open subset of $\h_0,$ $$p^*\left(\Z\right)\left(U\right)\simeq\Z\left(U\right).$$ Hence,

\begin{eqnarray*}
\Z_x&=& x^*\Z\\
 &\simeq& \tilde x^*\left(p^*\Z\right)\\
 &\simeq& \left(p^*\Z\right)_{\tilde x}\\
 &\simeq& \underset{\tilde x \in U} \hc \left(p^*\Z\right)\left(U\right)\\
 &\simeq& \underset{\tilde x \in U} \hc \Z\left(U\right).\\
\end{eqnarray*}
\end{proof}

\bibliographystyle{hplain}
\bibliography{meanet}

\end{document}